\documentclass[a4paper,reqno]{amsart}%[reqno]{amsart}%\documentclass[reqno]{amsart}

\usepackage{a4wide}
\usepackage[ansinew]{inputenc}%\usepackage[utf8]{inputenc}			% beide packages für Umlaute
\usepackage[english,ngerman]{babel}
\usepackage{mathabx}
\usepackage{subfigure}
\usepackage{listings}
\usepackage{mathtools}
\usepackage{nicefrac}%{xfrac}
\usepackage{float}
\usepackage{geometry}
\usepackage{dsfont}
\usepackage{amssymb}
\usepackage{stmaryrd}
\usepackage{graphicx}
\usepackage{caption}	% für \captionof bei groupes
\usepackage{esint} % average integral

\title[Convergence of a scheme for a thin film equation]{Convergence of a fully discrete variational scheme for a thin-film equation}

\author{Horst Osberger}
\address{Horst Osberger \\ Zentrum Mathematik \\ TU M\"unchen \\ Boltzmannstr. 3 \\ D-85748 Garching \\ Germany}
\email{osberger@ma.tum.de}

\author{Daniel Matthes}
\address{Daniel Matthes \\ Zentrum Mathematik \\ TU M\"unchen \\ Boltzmannstr. 3 \\ D-85748 Garching \\ Germany}
\email{matthes@ma.tum.de}

\begin{document}
\raggedbottom																						% <-------- Tipp von Schorsch wegegn großer Abstände...
\newcommand{\setR}{\mathbb{R}}
\newcommand{\setRpp}{\mathbb{R}_{>0}}
\newcommand{\setRp}{\mathbb{R}_{\geq0}}
\newcommand{\setN}{\mathbb{N}}
\newcommand{\dd}{\,\mathrm{d}}
\newcommand{\dn}{\mathrm{d}}
\newcommand{\indy}{\mathbb{I}}
\newcommand{\eins}{\mathds{1}}
\newcommand{\Omegac}{\overline{\Omega}}

% Variable names
\newcommand{\theX}{\mathrm{X}}
\newcommand{\theY}{\mathrm{Y}}
\newcommand{\theZ}{\mathrm{Z}}
\newcommand{\theU}{\mathrm{U}}
\newcommand{\xvec}{\vec{\mathrm x}}
\newcommand{\yvec}{\vec{\mathrm y}}
\newcommand{\zvec}{\vec{\mathrm z}}
\newcommand{\vvec}{\vec{\mathrm v}}
\newcommand{\wvec}{\vec{\mathrm w}}
\newcommand{\wspr}[2]{\left\langle#1,#2\right\rangle_\delta}
\newcommand{\wnrm}[1]{\left\|#1\right\|_\delta}
\newcommand{\rvec}{\mathrm \rho'}
\newcommand{\rd}{\rho'}
\newcommand{\rdd}{\rho''}
\newcommand{\rddd}{\rho'''}
\newcommand{\rkap}{B}%{\varpi}
\newcommand{\gvec}{\vec{g}}
\newcommand{\s}{{\widetilde{\xi}}_k}
\newcommand{\Nt}{{N_\tau}}

% Conversions
\newcommand{\convf}{\mathbf{u}}
\newcommand{\cf}{\mathbf{u}}
\newcommand{\cfM}{\mathbf{w}}
\newcommand{\cz}{\mathbf{z}}
\newcommand{\cX}{\mathbf{X}}
\newcommand{\cXd}{\mathbf{X}}%{\widetilde{\mathbf{X}}}

% Spaces
\newcommand{\theh}{\delta} % {{\boldsymbol \xi}}
\newcommand{\thep}{\delta} % {{\boldsymbol \xi}}
\newcommand{\dm}{{\overline{\delta}}}
\newcommand{\xspc}{\mathfrak{X}}
\newcommand{\xseq}{\mathfrak{x}}
\newcommand{\xseqN}{\xseq_\delta}
\newcommand{\zseq}{\mathfrak{z}}
\newcommand{\dens}{\mathcal{P}_2(\Omega)}
\newcommand{\densN}{\mathcal{P}_\delta(\Omega)}
\newcommand{\Lloc}{L_{\operatorname{loc}}}

% Norms
\newcommand{\tv}[1]{{\mathrm{TV}}\left[#1\right]}
\newcommand{\ti}[1]{\left\lbrace#1\right\rbrace_{\tau}}%{\overline{#1}}%{\left\langle#1\right\rangle_\tau}
\newcommand{\tti}[1]{\left\langle#1\right\rangle_{\tau}}%{\overline{#1}}%

% Matrices
\newcommand{\Wmat}{\mathrm{W}}
\newcommand{\Wmatn}{\widetilde{\mathrm{W}}}

% Time
\newcommand{\tn}{n\tau}
\newcommand{\tnm}{(n-1)\tau}
\newcommand{\ed}{\eta'}

% Error-terms for weak formulation
\newcommand{\Ai}{A_1^n}%{{(I)_\Delta^n}}
\newcommand{\Aii}{A_2^n}%{{(II)_\Delta^n}}
\newcommand{\Aiii}{A_3^n}%{{(III)_\Delta^n}}
\newcommand{\Aiv}{A_4^n}%{{(IV)_\Delta^n}}
\newcommand{\Av}{A_5^n}%{{(IV)_\Delta^n}}
\newcommand{\Avi}{A_6^n}%{{(IV)_\Delta^n}}
\newcommand{\Avii}{A_7^n}%{{(IV)_\Delta^n}}
\newcommand{\Aio}{A_1}
\newcommand{\Aiio}{A_2}
\newcommand{\Aiiio}{A_3}
\newcommand{\Aivo}{A_4}
\newcommand{\Avo}{A_5}
\newcommand{\Avio}{A_6}
\newcommand{\Aviio}{A_7}
\newcommand{\Ais}{A_i}
\newcommand{\Ris}{R_i}
\newcommand{\Ri}{R_1}%{{R_{(I),\Delta}^n}}
\newcommand{\Rii}{R_2}%{R_{(II),\Delta}^n}
\newcommand{\Riii}{R_3}%{R_{(III),\Delta}^n}
\newcommand{\Riv}{R_4}%{R_{(IV),\Delta}^n}
\newcommand{\Rv}{R_5}%{R_{(IV),\Delta}^n}
\newcommand{\Rvi}{R_6}
\newcommand{\Rvii}{R_7}
\newcommand{\Rviii}{R_8}
\newcommand{\er}{\mathrm{e}}
\newcommand{\Res}{R}
\newcommand{\Resn}{R^n}

% Indices
\newcommand{\ival}{{\mathbb{I}_K}}
\newcommand{\ivalp}{{\mathbb{I}_K^+}}
\newcommand{\hval}{{\mathbb{I}^{1/2}_{K}}}
\newcommand{\aval}{{\mathbb{I}_K}}
\newcommand{\kmh}{{k-\frac{1}{2}}}
\newcommand{\kph}{{k+\frac{1}{2}}}
\newcommand{\kpd}{{k+\frac{3}{2}}}
\newcommand{\kmd}{{k-\frac{3}{2}}}
\newcommand{\jmh}{{j-\frac{1}{2}}}
\newcommand{\jph}{{j+\frac{1}{2}}}
\newcommand{\jmd}{{j-\frac{3}{2}}}
\newcommand{\kpmh}{{k\pm\frac{1}{2}}}
\newcommand{\kmph}{{k\mp\frac{1}{2}}}
\newcommand{\kappm}{{\kappa-\frac12}}
\newcommand{\kappp}{{\kappa+\frac12}}

\newcommand{\Kmh}{{K-\frac{1}{2}}}
\newcommand{\Kph}{{K+\frac{1}{2}}}
\newcommand{\imh}{{\frac{1}{2}}}
\newcommand{\iph}{{\frac{3}{2}}}
\newcommand{\Kmd}{{K-\frac{3}{2}}}

% Other symbols
\newcommand{\thegrad}{\operatorname{grad}_{\mathbb W}}
\newcommand{\eps}{{\varepsilon}}
\newcommand{\wass}{\mathcal{W}_2}
\newcommand{\wassN}{\mathbf{W}_2}
\newcommand{\grad}{\partial_{\xvec}}
\newcommand{\wgrad}{\nabla_\delta}
\newcommand{\nci}{\mathfrak{F}}
\newcommand{\Hh}{\mathcal{H}^{[1]}}
\newcommand{\Hhz}{\mathbf{H}^{[1]}_\delta}
\newcommand{\olHh}{\overline{\mathbf{H}^{[1]}}}%{\overline{\Hhz}}
\newcommand{\entq}{\mathcal{H}^{[2]}}
\newcommand{\entqz}{\mathbf{H}^{[2]}_\delta}
\newcommand{\E}{\mathcal{H}^{[3/2]}}
\newcommand{\Ez}{\mathbf{H}^{[3/2]}_\delta}
\newcommand{\fish}{\mathcal{F}}
\newcommand{\Di}{\mathcal{E}}
\newcommand{\Diz}{\mathbf{E}_\delta}		
\newcommand{\DiVX}{\mathbf{E}^V}
\newcommand{\DiV}{\mathcal{E}^V}
\newcommand{\DiVz}{\mathbf{E}^V_\delta}
\newcommand{\DiVy}{\mathbf{E}^V_\Delta}
\newcommand{\olDi}{\overline{\mathbf{E}^V}}
\newcommand{\pot}{\mathcal{V}}
\newcommand{\potz}{\mathbf{V}_\delta}
\newcommand{\qmom}{\mathbf{m}_2}
\newcommand{\qmombar}{\overline{\mathbf{m}}_{2}}

\newcommand{\Phiz}{\mathbf{\Phi}}
\newcommand{\hatf}{\theta}
\newcommand{\hatF}{\Theta}
\newcommand{\vel}{\operatorname{v}}
\newcommand{\D}{\operatorname{D}_\delta}
\newcommand{\reff}{{\operatorname{ref}}}
\newcommand{\Dr}{\operatorname{D}_{\operatorname{ref}}}
\newcommand{\ee}{\mathbf{e}}

\newcommand{\M}{[0,M]}%{\mathcal{M}}
\newcommand{\Md}{\mathcal{M}}
\newcommand{\intom}{\int_\Omega}
\newcommand{\intL}{\int_\Omega}%{\int_{-L}^L}
\newcommand{\intM}{\int_0^M}
\newcommand{\intMd}{\int_{\Md}}

% Interpolation operators
\newcommand{\baru}{u}%{\bar{u}}
\newcommand{\hatz}{\widehat{z}}
\newcommand{\hatu}{\widehat{u}}
\newcommand{\psia}{\psi_\alpha}

% New symbols für numerc part
\newcommand{\vech}{{\vec\delta}}
\newcommand{\xseqNN}{\xseq_\vech}
\newcommand{\densNN}{\mathcal{P}_\vech(\Omega)}

% Math operators
\newcommand{\essinf}{\operatorname*{ess\,inf}}
\newcommand{\loc}{\text{loc}}
\newtheorem{thm}{Theorem}
\newtheorem{prp}[thm]{Proposition}
\newtheorem{lem}[thm]{Lemma}
\newtheorem{cor}[thm]{Corollary}
\newtheorem{rmk}[thm]{Remark}
\newtheorem{dfn}[thm]{Definition}
\newtheorem{ass}[thm]{Assumption}
\newtheorem{xmp}[thm]{Example}
\newenvironment{proof2}[1]{\vspace{3mm}\noindent\textit{Proof of #1.}}{\hfill$\Box$ \vspace{5mm}}%{\noindent \textit{Proof of #1 .}}{\hfill$\Box$}

\selectlanguage{english}

\begin{abstract}
		This paper is concerned with a rigorous convergence analysis of a fully discrete Lagrangian scheme 
    for the Hele-Shaw flow, which is the fourth order thin-film equation with linear mobility in one space dimension.
    The discretization is based on the equation's gradient flow structure in the $L^2$-Wasserstein metric.
    Apart from its Lagrangian character --- which guarantees positivity and mass conservation ---
    the main feature of our discretization is that it dissipates \emph{both} the Dirichlet energy and the logarithmic entropy.
    The interplay between these two dissipations paves the way to proving convergence of the discrete approximations 
    to a weak solution in the discrete-to-continuous limit.
    Thanks to the time-implicit character of the scheme, no CFL-type condition is needed.
    Numerical experiments illustrate the practicability of the scheme.
\end{abstract}

\maketitle
%\tableofcontents 

%-----------------------------------------------------------------------------------------------------------------------------------------------------------------
%-----------------------------------------------------------------------------------------------------------------------------------------------------------------
%-----------------------------------------------------------------------------------------------------------------------------------------------------------------
\section{Introduction}\label{sec:intro}
%
%
%-----------------------------------------------------------------------------------------------------------------------------------------------------------------
\subsection{The equation and its properties}
In this paper, we propose and study a fully discrete numerical scheme 
for the following nonlinear fourth order equation on the bounded domain $\Omega=(a,b)$, $a<b$
with no-flux boundary conditions:
\begin{align}
  \label{eq:thinfilm}
  \partial_t u = -\big(u u_{xxx}\big)_x + \big(V_x u\big)_x &\quad\textnormal{for } t>0 \;\textnormal{and}\; x\in\Omega.\\
	\label{eq:bc}
  u u_x=0,\quad  u u_{xxx}=0 &\quad\text{for $t>0$ and $x\in\partial\Omega$}.
\end{align}
We assume that the potential $V\in C^2(\Omega)$ is non-negative with bounded second derivative,
\begin{align}
  \label{eq:Vhypo}
  V\ge0, \quad \Lambda:=\sup|V_{xx}|<\infty,
\end{align}
a typical choice being $V(x)=\frac\Lambda2x^2$.
Equation \eqref{eq:thinfilm} belongs to the family of thin film equations,
\begin{align}
  \label{eq:gthinfilm}
  \partial_t u = -\operatorname{div}(m(u)\operatorname{D}\Delta u).
\end{align}
Equations of this form are used as reduced models for laminar flow with a free liquid-air interface \cite{Oron}.
The case of linear mobility $m(u)=u$ studied here
is further used to describe the pinching of thin necks in a Hele-Shaw cell,
hence \eqref{eq:thinfilm} is often referred to as the \emph{Hele-Shaw flow}.

The analysis of the fourth order degenerate problems in \eqref{eq:gthinfilm} is by no means trivial.
There exists a rich literature on the existence and long-time behavior of solutions,
see e.g. \cite{Bertsch,Passo,BGruen,Gruen,BernisF,Ulusoy,CaTothin}.
These results \emph{could not} be obtained by straight-forward extension of the techniques for second order parabolic equations.
The most decisive difference between second and fourth order is the lack of comparison principles for the latter.
Instead, energy and entropy methods play a key role in the analysis.
Naturally, classical parabolic theory applies in zones on which the solution $u$ is strictly positive, 
yielding $C^\infty$-regularity there.
However, one is typically interested in solutions that are \emph{not} strictly positive but have a compact, time-dependent support.
% there are counter examples \cite{BGruen} which indicate the existence of
% initially positive solutions that become zero on points or even intervals at a later time.
For such solutions, one only has the regularity that is induced by the energy/entropy estimates,
which is usually something of the type $L^\infty([0,T];H^1(\Omega))\cap L^2([0,T];H^2(\Omega))$, but no better.

It is known that \eqref{eq:gthinfilm} carries a variety of structural properties:
solutions conserve mass and non-negativity, there exists a large class of Lyapunov functionals \cite{Laugesen},
and it can be written as a gradient flow in the energy landscape of the following (modified) Dirichlet functional,
\begin{align}
  \label{eq:info_tf}
  \DiV(u) = \Di(u) + \pot(u),
  \quad\textnormal{with}\quad 
  \Di(u) = \frac{1}{2}\intom\big(\partial_x u\big)^2\dd x,\quad
  \pot(u) = \intom V(x) u(x) \dd x,
\end{align}
with respect to the $L^2$-Wasserstein metric \cite{GiOt}.
The two main a priori estimates for the well-posedness theory of  \eqref{eq:thinfilm} 
are obtained from the dissipation of $\DiV$, 
and from the dissipation of an auxiliary Lyapunov functional, namely the entropy,
\begin{align*}
  \Hh(u) = \intom u\log u\dd x.
\end{align*}
Formally, the respective estimates are immeditaly obtained from an integration by parts;
for $V\equiv0$, one obtains
\begin{align}
  \label{eq:dissnrj_tf}
  -\frac{\dn}{\dd t}\Di(u) &= \intom u_x\,(uu_{xxx})_{xx}\dd x = \intom u\,u_{xxx}^2\dd x, \\
  \label{eq:dissentl_tf}
  -\frac{\dn}{\dd t}\Hh(u) &= \intom \log u\,(uu_{xxx})_x\dd x = \intom u_{xx}^2\dd x.
\end{align}
Notice that energy dissipation does \emph{not} provide $L^2([0,T];H^3(\Omega))$-regularity,
due to the degeneracy of the integrand in regions where $u$ vanishes.
In principle, the famous Bernis estimates \cite{Bernis} could be used to extract an $L^2([0,T];H^3(\Omega))$-bound on $u^{3/2}$,
but we shall not discuss this ansatz here since these --- quite delicate --- estimates seem impossible to discretize.

In the numerical approximation of solutions to \eqref{eq:thinfilm}, 
it is natural to ask for a structure-preserving discretization that inherits at least some of the properties listed above.
A minimal criterion is to guarantee nonnegativity of solutions --- which turns out to be already a challenge.
Here, we try to push the structure preservation as far as possible,
with a scheme that translates both the Lagrangian and the gradient flow nature of \eqref{eq:thinfilm} from continuous to discrete,
and which inherits simultaneously the Dirichlet energy and the entropy as Lyapunov functionals.
We even obtain a discretized version of the estimate \eqref{eq:dissentl_tf},
and that is the key element for the convergence analysis.

Our discretization ansatz is closely related to the one that has been introduced and analyzed recently by the authors \cite{dlssv3}
in the context of the fourth order \emph{quantum drift diffusion} (QDD) equation,
also known as \emph{Derrida-Lebowitz-Speer-Spohn} (DLSS) equation:
\begin{align}
  \label{eq:dlss}
  \partial_t u = -\big( u (\log u)_{xx}\big)_{xx} + \big( V_x u\big)_x.
\end{align}
This equation is a gradient flow in the $L^2$-Wasserstein metric as well \cite{GST}.
In place of $\DiV$, the flow potential is given by the perturbed Fisher information
\begin{align*}
  \fish^V(u) = \fish(u) + \pot(u), \quad \text{with}\quad \fish (u) = 2\intom \big(\partial_x\sqrt{u}\big)^2\dd x,
\end{align*}
and $\pot$ as above.
There exists a non-obvious connection between \eqref{eq:dlss} and the linear heat equation \cite{DMfourth}, 
that is best understood as a relation between $\fish$, $\Hh$ and the $L^2$-Wasserstein metric \cite{MMS}.
The main feature of the particular discretization of \eqref{eq:dlss} used in \cite{dlssv3} is the preservation of that relation,
and that paves the way to a relatively complete analytical treatment of the resulting numerical scheme.
In the context at hand, 
the connection between $\Di$ and $\Hh$ --- that is, the origin of the relation \eqref{eq:dissentl_tf} --- is less obvious,
but on grounds of the ideas developed in \cite{dlssv3}, 
we are able to define a special discretization that admits a discretized version of \eqref{eq:dissentl_tf}.

% So far, numerous (semi-)discretizations for fourth order equations (especially for the so-called 
% \emph{dlss equation}) that guarantee positivity
% have been proposed in the literature \cite{BEJnum,CJTnum,JPnum,JuVi}.
% In all of them, positivity actually appears as a consequence of another, more fundamental feature:
% each of these schemes also inherits a Lyapunov functional, 
% either a logarithmic/power-type entropy \cite{BEJnum,CJTnum,JPnum},
% or a variant of the Fisher information \cite{BEJnum,DMMnum,JuVi}.
% It is thus a more elegant way to find discretizations that inherits positivity by contruction. 
% A very first step in this direction is the approach for the dlss-equation from \cite{DMMnum},
% which preserves the equation's Lagrangian representation, and thus enforces positivity by construction.
% A quit similar scheme to the one in \cite{DMMnum} was then perforemd in \cite{dlssv3}, 
% where a discrete version of \eqref{eq:magic} forms the key ingredient for 
% the first -- as far as we know -- convergence proof for a fully-discrete numerical scheme for the fourth order dlss equation. 

\subsection{Definition of the discretization}
We are now going to present a discretization for \eqref{eq:thinfilm} 
that approximates weak solutions to \eqref{eq:thinfilm} of finite positive mass $M\in\setRpp$.
Basically, we follow the ansatz from \cite{dlssv3}, but we shall deviate in the discretization of the potential of the flow.
First, the equation \eqref{eq:thinfilm} is re-written in terms of Lagrangian coordinates:
since each $u(t,\cdot)$ is of fixed mass $M$, we can introduce time-dependent Lagrangian maps $\theX(t,\cdot):\M\to\Omega$
% --- the so-called \emph{pseudo-inverse distribution function} of $u(t,\cdot)$ ---
implicitly by
\begin{align}
  \label{int:pseudo}
  \xi = \int_{a}^{\theX(t,\xi)}u(t,x)\dd x, \quad \text{for each $\xi\in\M$}.
\end{align}
%Here, we write $\M=[0,M]$ for the mass domain.
For the moment, we ignore the ambiguity in the definition of $\theX(t,\xi)$ outside of the support of $u(t)$.
Expressed in terms of $\theX$, and after elementary manipulations, the Hele-Shaw equation \eqref{eq:thinfilm} becomes:
\begin{align}
  \label{eq:zeq_thinfilm}
  \partial_t\theX = \partial_\xi\left(\frac{1}{2}Z^3\partial_{\xi\xi}Z + \frac{1}{4} Z^2\partial_{\xi\xi}\big(Z^2\big)\right) + V(\theX), 
  \quad \text{where} \quad Z(t,\xi):=\frac1{\partial_\xi\theX(t,\xi)}=u\big(t,\theX(t,\xi)\big). 
\end{align}
It is easily seen that equation \eqref{eq:zeq_thinfilm} is the $L^2$-gradient flow for
\begin{align*}
  \DiVX(\theX) := \DiV(u\circ\theX) 
  = \frac12\intM\left[\frac{1}{\theX_\xi}\right]_\xi^2\frac{1}{\theX_\xi}\dd\xi
  +\intM V(\theX)\dd\xi,
\end{align*}
with respect to the usual $L^2$-norm on $L^2(\M;\setR)$.
This directly reflects the gradient flow structure of \eqref{eq:thinfilm} with respect to the $L^2$-Wasserstein metric.
 
Equation \eqref{eq:zeq_thinfilm} is now discretized as follows.
First, fix a spatio-temporal discretization parameter $\Delta=(\tau;\thep)$,
where $\tau>0$ is a time step size, 
and $\delta=M/K$ for some $K\in\setN$ defines an equidistant partition of $\M$ 
into $K$ intervals $[\xi_{k-1},\xi_k]$ of length $\delta$ each, i.e., $\xi_k=k\delta$ for $k=0,1,\ldots,K$.
%Let $\D^1$ be the associated central finite difference operator that associates difference quotients to the mid-points of the grid,
Accordingly, introduce the central first and second order finite difference operators $\D^1$ and $\D^2$
for discrete functions defined either on the $\xi_k$'s or on the interval midpoints $\xi_{k+1/2}=(k+1/2)\delta$
in the canonical way; see Section \ref{sec:ansatz} for details.

%with intermediate values $y_\kappa$ for $k=-\frac12,\frac12,\ldots,K-\frac12,K+\frac12$, 
%we have $(\D^1\yvec)_k=(y_\kph-y_\kmh)/\delta$ for $k=0,\ldots,K$.
%Accordingly, $\D^2$ is the respective second order difference operator,
%so that $(\D^2\yvec)_\kappa=(y_{\kappp}-2y_\kappa+y_{\kappm})/\delta^2$ for $\kappa=\frac12,\ldots,K-\frac12$.
At each time $t=n\tau$, the Lagrangian map $\theX(t,\cdot)$ is approximated by a monotone vector 
\[ \xvec_\Delta^n = (x_1^n,\ldots,x_{K-1}^n)\in\setR^{K-1} \quad\text{with $a<x^n_1<\cdots<x_{K-1}^n < b$} \] 
in the sense that $X(n\tau,k\delta)\approx x^n_k$.
We will further use the convention that $x_0=a$ and $x_K=b$.
For brevity, introduce the vectors $z_\Delta^n=\cz[\xvec_\Delta^n]$ %=(z^n_{\frac12},z^n_{\frac32},\ldots,z^n_{K-\frac12})$ 
with entries
\begin{align}
  \label{eq:zvec}
  \big(\cz[\xvec]\big)_\kappa 
  = \frac{\delta}{x_\kappp - x_\kappm}=\frac1{[\D^1\xvec]_\kappa} \quad \text{for}\quad\kappa=\frac12,\frac32,\ldots,K-\frac12,
\end{align}
and $z_{-\frac{1}{2}}=z_\imh$ and $z_\Kph=z_\Kmh$ by convention.
These vectors approximate the function $Z$ in \eqref{eq:zeq_thinfilm} such that $Z(n\tau,\kappa\delta)\approx z^n_\kappa$.

The fully discrete evolution for the $\xvec_\Delta^n$ is now obtained 
from the following standard discretization of \eqref{eq:zeq_thinfilm} with central finite differences:
\begin{align}
  \label{eq:dgf_tf}
  \frac{x^n_k-x^{n-1}_k}{\tau}
  = \D^1\left[\frac{1}{2}(\zvec^n)^3\D^2[\zvec^n] + \frac{1}{4}(\zvec^n)^2\D^2[(\zvec^n)^2]\right]_k + V_x(x_k).
\end{align}
Note that there are infinitely many equivalent ways to re-write the right-hand side of equation \eqref{eq:zeq_thinfilm},
and accordingly infinitely many (non-equivalent!) central finite-difference discretizations.
Another one, having different properties, is studied in \cite{HO}.
Our convergence result only applies to the particular form \eqref{eq:dgf_tf},
since only for that one, we obtain ``the right'' Lyapunov functionals 
that provide the a priori estimates for the discrete-to-continuous limit.

Finally, we define a time-dependent, spatially piecewise constant density function $u_\Delta^n:\Omega\to\setR_{\ge0}$ 
from the sequence  $\xvec_\Delta:=(\xvec_\Delta^n)_{n=0}^\infty$ via
\begin{align}
  \label{eq:cf}
  u_\Delta^n = \cf_\delta[\xvec_\Delta^n] := \sum_{k=1}^K\frac\delta{x^n_k-x^n_{k-1}} \indy_{(x_{k-1},x_k]}.
\end{align}
By definition, the densities are non-negative and of time-independent mass,
\[ \intom u_\Delta^n\dd x = \sum_{k=1}^K \int_{x_{k-1}}^{x_k}\frac{\delta}{x^n_k-x^n_{k-1}} = K\delta = M. \]
Finally, we introduce the piecewise constant interpolation $\ti{u_\Delta}:\setRp\times\Omega\to\setRp$ in time by
\begin{align}
  \label{eq:tinterpolate}
  \ti{u_\Delta}(t) = u_\Delta^n \quad\text{for $(n-1)\tau<t\le n\tau$},
\end{align}
and $\ti{u_\Delta}(0)=u_\Delta^0$.

%
%-----------------------------------------------------------------------------------------------------------------------------------------------------------------
\subsection{Main results}
For the statement of our first result, fix a discretization parameter $\Delta=(\tau;\thep)$.
On monotone vectors $\xvec\in\setR^{K-1}$ with densities $\zvec=\cz[\xvec]$,
introduce the functionals
\begin{align*}
  % \label{eq:HFzQz}
  \Hhz(\xvec) := \delta\sum_{k=1}^K \log(z_\kmh),
  % \quad\entqz(\xvec) := \frac{\delta}{4}\sum_{k=1}^Kz_\kmh
  \quad
  % \label{eq:Tz}
  \DiVz(\xvec) := \frac{\delta}{2}\sum_{k=1}^K\frac{z_\kph+z_\kmh}{2}\left(\frac{z_\kph-z_\kmh}{\delta}\right)^2 + \delta\sum_{k=0}^KV_x(x_k),
\end{align*}
which are discrete replacements for the entropy and the modified Dirichlet energy functionals, respectively.
%\begin{equation}\begin{split}\label{eq:HFzQzTz}
% \Hhz(\xvec) &:= \delta\sum_{k=1}^K \log(z_\kmh),\quad
% \entqz(\xvec) := \frac{\delta}{4}\sum_{k=1}^Kz_\kmh, \quad\textnormal{and}\\
% \DiVVz(\xvec) &:= \frac{\delta}{2}\sum_{k=1}^K\frac{z_\kph+z_\kmh}{2}\left(\frac{z_\kph-z_\kmh}{\delta}\right)^2 + \sum_{k=0}^KV_x(x_k).
% \end{split}\end{equation}
% These choices are made such that $\Hhz$ and $\entqz$ are the restriction of $\Hh$ and $\entq$, 
% i.e. $\Hhz(\xvec)=\Hh(\cf_\delta[\xvec])$ and $\entqz(\xvec)=\entq(\cf_\delta[\xvec])$,
% and $\DiVVz$ mimics \eqref{eq:magic2} in the discrete setting, see section \ref{sec:discretization} for a more detailed explanation.
%
\begin{thm}\label{thm:pre_tf}
  From any monotone discrete initial datum $\xvec_\Delta^0$, 
  a sequence of monotone $\xvec_\Delta^n$ satisfying \eqref{eq:dgf_tf} can be constructed 
  by inductively defining $\xvec_\Delta^n$ as a global minimizer of
  \begin{align}\label{eq:dmm}
    \xvec\mapsto\frac\delta{2\tau}\sum_k \big(x_k-x_k^{n-1})^2 + \DiVz(\xvec).
  \end{align}
  This sequence of vectors $\xvec_\Delta^n$
  dissipates both the Boltzmann entropy and the discrete Dirichlet energy,
  \begin{align*}
    \Hhz(\xvec_\Delta^n)\le\Hhz(\xvec_\Delta^{n-1}) \quad\text{and}\quad \DiVz(\xvec_\Delta^n)\le\DiVz(\xvec_\Delta^{n-1}).
  \end{align*}
\end{thm}
To state our main result about convergence, recall the definition \eqref{eq:tinterpolate} of the time interpolant.
Further, $\Delta$ symbolizes a whole sequence of mesh parameters from now on,
and we write $\Delta\to0$ to indicate that $\tau\to0$ and $\delta\to0$ simultaneously.
\begin{thm}
  \label{thm:main_tf}
  Let a non-negative initial condition $u^0\in H^1(\Omega)$ of finite second moment be given
	and fix a time horizont $T>0$.
  Choose initial approximations $\xvec_\Delta^0$ 
  such that $u^0_\Delta=\cf[\xvec_\Delta^0]\rightharpoonup u^0$ weakly in $H^1(\Omega)$ as $\Delta\to0$,
  and
  \begin{align}\label{eq:genhypo_tf}
    \olDi := \sup_\Delta\DiVz(\xvec^0_\Delta)<\infty,
    \quad
    \olHh := \sup_\Delta\Hhz(\xvec^0_\Delta)<\infty.
    %\quad
    %\qmombar := \sup_\Delta\left(\delta\sum_{k=0}^K(x^0_k)^2\right)<\infty.
  \end{align}
  For each $\Delta$, construct a discrete approximation $\xvec_\Delta$ 
  according to the procedure described in Theorem \ref{thm:pre_tf} above.
  Then, there are a subsequence with $\Delta\to0$ and a limit function $u_*\in C([0,T]\times\Omega)$ such that:
  \begin{itemize}
  \item $\ti{u_\Delta}$ converges to $u_*$ locally uniformly on $[0,T]\times\Omega$,
  \item $u_*\in L^2([0,T];H^1(\Omega))$, 
  \item $u_*(0)=u^0$,
  \item $u_*$ satisfies the following weak formulation of \eqref{eq:thinfilm} with no-flux boundary conditions \eqref{eq:bc}:
    \begin{align}\label{eq:weak_intro_tf}
      \int_0^T\int_\setR \partial_t\eta\rho u_* \dd t\dd x + \int_0^T \eta N(u_*,\rho)\dd t = 0,
    \end{align}
    for any test functions $\rho\in C^{\infty}(\Omega)$ with $\rd(a)=\rd(0)=0$, and $\eta\in C_c^{\infty}((0,T))$,
    where the operator $N$ is given by
    \begin{align}
      \label{eq:N_tf}
      N(u,\rho) := \frac12\intom \left( (u^2)_x \rho_{xxx} + 3 u_x^2 \rho_{xx}\right) \dd x + \intom V_x u \rho_x \dd x.
    \end{align}
  \end{itemize}
\end{thm}
\begin{rmk}
  \begin{enumerate}
  \item \emph{Quality of convergence:}
    Since $\ti{u_\Delta}$ is piecewise constant in space and time,
    uniform convergence is obviously the best kind of convergence that can be achieved.
    % The estimate $\sqrt{u}\in L^2_\loc(\setRnn;H^2(\Omega))$ from \cite{GST,JMdlss} seems out of reach.
    % In the proof, we make ad hoc definitions $\hatu_\Delta^n$ of the $u_\Delta^n$ that are piecewise affine,
    % and are such that the $\ti{\hatu_\Delta}$ converge to $u$ in $L^2_\loc(\setRnn;H^1(\Omega))$.
  \item \emph{Rate of convergence:}
    Numerical experiments with smooth initial data $u^0$ show that 
    the rate of convergence if of order $\tau+\delta^2$, see Section \ref{sec:num_tf}.
    % The weak consistency that we can prove rigorously is only of order $\tau+\delta^{1/4}$, 
    % see \eqref{eq:weakform1}\&\eqref{eq:weakform2}.
    % \item \emph{Initial condition:}   
    %   We emphasize that our only hypothesis on $u^0$ is $\Hh(u^0)<\infty$, 
    %   which allows the same general initial conditions as in \cite{GST,JMdlss}.
    %   If $\F(u^0)$ happens to be finite, and also $\sup_\Delta\DiVz(\xvec_\Delta^0)<\infty$,
    %   then the uniform convergence of $\ti{\baru_\Delta}$ holds up to $t=0$.
    %% \item \emph{Lyapunov functionals:}
    %%   It can be shown that, for almost every $t>0$, both $\Hh(u_*(t))$ and $\F(u_*(t))$ are equal
    %%   to monotonically decaying functions.
    % \item \emph{Long time behavior:}
    %   By means of the Csiszar-Kullback inequality, the exponential decay of $\Hh(u_*(t))$ to zero implies 
    %   exponential convergence of $u_*$ to the constant function $u_\infty\equiv M/(b-a)$ in $L^1(\Omega)$.
  \item \emph{No uniqueness:}
    Since our notion of solution is very weak,
    we cannot exclude that different subsequences of $\ti{u_\Delta}$ converge to different limits.
  \item \emph{Initial approximation:}
    The assumptions in \eqref{eq:genhypo_tf} are not independent:
    boundedness of $\DiVz(\xvec_\Delta^0)$ implies boundedness of $\Hhz(\xvec_\Delta^0)$ from above.
  \end{enumerate}
\end{rmk}

\subsection{Relation to the literature}
The idea to derive numerical discretizations for solution of Wasserstein gradient flows from the Lagrangian representation
is not new in the literature. 
A very general (but rather theoretical) treatise was given in \cite{Kinderlehrer}.
Several practical schemes have been developed on grounds 
of the Lagrangian representation for this class of evolution problems,
mainly for second-order diffusion equations \cite{Budd,BCW,MacCamy,Russo},
but also for chemotaxis systems \cite{BCC},
for non-local aggregation equations \cite{CarM,Mary}, % if existent, cite {CarWol},
and for variants of the Boltzmann equation \cite{GosT2}.

Lagrangian schemes for fourth order equations are relatively rare.
Alternative Lagrangian discretizations for \eqref{eq:thinfilm} or related thin-film type equations 
have been proposed and analyzed in \cite{Naldi,GosT2},
but no rigorous convergence analysis has been carried out.
We also mention two schemes \cite{DMMnum,dlssv3} for the quantum drift diffusion equation, 
that is formally similar to \eqref{eq:thinfilm}. 
In \cite{dlssv3}, the idea to enforce dissipation of two Lyapunov functionals has been developed,
and was used to rigorously study the discrete-to-continuous limit.
For the analysis here, we shall borrow various ideas from \cite{dlssv3}.

A comment is in place on related Lagrangian schemes in higher spatial dimensions.
Here, and also in our related works \cite{HO,dde,dlssv3},
the most significant benefit from working on a one-dimensional interval, 
is that the space of densities is flat with respect to the $L^2$-Wasserstein metric;
it is of non-positive curvature in higher dimensions, 
which makes the numerical approximation of the Wasserstein distance significantly more difficult.
Just recently, a very promising approach for a truely structure-preserving discretization in higher space dimensions 
has been made \cite{Merigot}.
There, a numerical solver for second order drift diffusion equations with aggregation in multiple space dimensions
is introduced that preserves 
--- in addition to the Lagrangian and the gradient flow aspects ---
also ``some geometry'' of the optimal transport.
These manifold structural properties enable the authors to rigorously perform a (partial) convergence analysis.
It is currently unclear if that approach can be pushed further to deal with fourth order equations as well.

Among the numerous non-Lagrangian approaches to numerical discretization of thin film equations,
we are aware of two contributions \cite{GrRu,BeZh} in which the idea to enforce simultaneous dissipation
of energy and entropy has been implemented.
The discretization is performed using finite elements \cite{GrRu} and finite differences \cite{BeZh}, respectively.
The discrete to continuous limit has been rigorously analyzed for both schemes.
In difference to the convergence result presented here, 
certain positivity hypotheses on the limit solution are either assumed a priori \cite{BeZh},
or are incorporated in the weak form of the limit equation \cite{GrRu}.
See, however, \cite{Gruen2} for an improvement of the convergence result.

The primary challenge in our convergence analysis is to carry out all estimates
under \emph{no additional assumptions on the regularity} of the limit solution $u_*$.
In particular, we would like to deal with compactly supported solutions of a priori low regularity at the edge of the support.
Also, we allow very general initial conditions $u^0$.
Without sufficient a priori smoothness, 
we cannot simply use Taylor approximations and the like to estimate the difference between $\ti{u_\Delta}$ and $u_*$.
Instead, we are forced to derive new a priori estimates directly from the scheme, using our two Lyapunov functionals.

On the technical level, the main difficulty is that our scheme is \emph{fully discrete}, 
which means that we are working with spatial difference quotients instead of derivatives.
Lacking a discrete chain rule, 
the derivation of the relevant estimates turns out to be much harder than for the original problem \eqref{eq:thinfilm}.
For instance, we are able to prove a compactness estimate for $u_\Delta$, but not for its inverse distribution function,
although both estimates would be equivalent in a smooth setting.
This forces us to switch back and forth between 
the original \eqref{eq:thinfilm} and the Lagrangian \eqref{eq:zeq_thinfilm} formulation of the thin-film equation.

%-----------------------------------------------------------------------------------------------------------------------------------------------------------------
\subsection{Key estimates}\label{sec:intro_struct}
We give a very formal outline for the derivation of the two main a priori estimate on the fully discrete solutions.

The first main estimate is related to the gradient flow structure of \eqref{eq:thinfilm}: 
it is the potential flow of the modified Dirichlet energy $\DiV$ with respect to the Wasserstein metric $\wass$.
The consequences, which are immediate from the abstract theory of gradient flows \cite{AGS}, 
are that $t\mapsto\DiV(u(t))$ is monotone, 
and that each solution ``curve'' $t\mapsto u(t)$ is globally H\"older-$\frac12$-continuous with respect to $\wass$.
In order to inherit these properties to our discretization, 
the latter is constructed as a gradient flow of a flow potential $\DiVz$ (which approximates $\DiV$ in a certain sense)
with respect to a particular metric on the space of monotone vectors (which is related to $\wass$).
See Section \ref{sec:ansatz} below for details.
The corresponding fully discrete energy estimates are collected in Proposition \ref{prp:energy}.
We are not able to give a meaning to the full energy dissipation relation \eqref{eq:dissnrj_tf} on the discrete side,
but this is irrelevant to our analysis.

The second, equally important discrete estimate mimicks \eqref{eq:dissentl_tf}.
Unfortunately, the $L^2$-norm of $u_{xx}$ is an inconvenient quantity to deal with, for two reasons.
First, we need to perform most of the estimates in the Lagrangian picture, where
\begin{align*}
  \intom u_{xx}^2\dd x= \frac14\intM Z(Z^2)_{\xi\xi}^2\dd\xi
\end{align*}
is algebraically more difficult to handle than
the equivalent functional
\begin{align}
  \label{eq:ghj}
  \intom u^2\,(\log u)_{xx}^2\dd x = \intM Z^3Z_{\xi\xi}^2\dd\xi
\end{align}
which we shall eventually work with, see Lemma \ref{lem:dissipation_tf}.
Our discretization \eqref{eq:dgf_tf} is taylor-made in such a way that 
entropy dissipation yields a discrete version of \eqref{eq:ghj}.

Second, the formulation of an $H^2$-estimate would require 
a global $C^{1,1}$-interpolation of the piecewise constant densities $u_\Delta$ that respects positivity,
which seems impractical.
Instead, we settle for a control on the total variation 
of the first derivative $\partial_\xi\hatu_\Delta$ of a simple $C^{0,1}$-interpolation $\hatu_\Delta$,
see Proposition \ref{prp:BVbound_tf}.
This TV-control is a perfect replacement for the $H^2$-estimate in \eqref{eq:dissentl_tf},
and is the source for compactness, see Proposition \ref{prp:convergence2_tf}.

% Here, $\dens$ denotes the set of all density functions on the domain $\setR$ with bounded second moment,
% \begin{align*}
%   \dens := \left\lbrace u:\setR\to\setR_+ :\; \intom u(x)\dd x = M,\;\textnormal{ and }\; \intom |x|^2 u(x) \dd x <+\infty\right\rbrace,
% \end{align*}
% equipped with the $L^2$-Wasserstein metric $\wass$.
% \begin{align}
%   \Hh(u) = \intom u\log(u) \dd x \quad\textnormal{and}\quad \entq(u) = \frac{1}{4}\intom u^2\dd x,
% \end{align}
%
%-----------------------------------------------------------------------------------------------------------------------------------------------------------------
\subsection{Structure of the paper}
Below, we start with a detailed description of our numerical scheme
as a discrete Wasserstein-like gradient flow and discuss structural consistency of our approach.
In Section \ref{sec:apriori}, we derive various a priori estimates on the fully discrete solutions.
This leads to the main convergence result in Proposition \ref{prp:convergence2_tf},
showing the existence of a limit function $u_*$ for $\Delta\to0$.
This limit function satisfies the weak formulation of \eqref{eq:thinfilm} stated in \eqref{eq:weak_intro_tf};
this is shown in Section \ref{sec:weak}. 
Finally, we report on numerical experiments and discuss the observed rate of convergence in Section \ref{sec:num_tf}.

%%%%%%%%%%%%%%%%%%%%%%%%%%%%%%%%%%%%%%%%%%%%%%%%%%%%%%%%%%%%%%%%%%%%%%%%%%%%%
\section{Definition of the fully discrete scheme}
%%%%%%%%%%%%%%%%%%%%%%%%%%%%%%%%%%%%%%%%%%%%%%%%%%%%%%%%%%%%%%%%%%%%%%%%%%%%%
%
%-----------------------------------------------------------------------------------------------------------------------------------------------------------------
%\subsection{The numerical scheme as discrete gradient flow}\label{sec:discretization}
%
The main aim of this section is to interprete the discrete equations \eqref{eq:dgf_tf} 
as time steps in the minimizing movement scheme for a suitable discretization $\DiVz$ of the functional $\DiV$
with respect to a Wasserstein-like metric on a finite-dimensional submanifold $\densN$ of $\dens$.
% As outlined in the introduction, the key point is to preserve the relation \eqref{eq:magic2}.
%
%-----------------------------------------------------------------------------------------------------------------------------------------------------------------
\subsection{Ansatz space and discrete entropy/information functionals}\label{sec:ansatz}
Fix $K\in\setN$, let $\delta:=1/K$, and define $\xi_k=Mk/K$ for $k=0,1,\ldots,K$.
For further reference, we introduce the sets of integer and half-integer indices
\begin{align*}
  \ival = \{0,1,\ldots,K\},\quad \ivalp = \{1,\ldots,K-1\},\quad
	\text{and}\quad
  \hval = \Big\{\frac12,\frac32,\ldots,K-\frac12\Big\}.
\end{align*}
First and second order central difference operators $\D^1$ and $\D^2$ are defined in the usual way:
if $\yvec=(y_\ell)_{\ell\in\ival}$ is a discrete function defined for integer indices $\ell\in\ival$ (i.e., on the nodes $\xi_\ell$),
then $\D^1\yvec$ and $\D^2\yvec$ are defined 
on half-integer indices $\kappa\in\hval$ (i.e., on the intervals $[\xi_\kappm,\xi_\kappp]$),
and on the ``inner'' integer indices $k\in\ivalp$, respectively,
with
\begin{align*}
  [\D^1\yvec]_\kappa  = \frac{y_\kph-y_\kmh}\delta,\quad
  [\D^2\yvec]_k = \frac{y_{\kappa+1}-2y_\kappa+y_{\kappa-1}}{\delta^2}. 
\end{align*}
If $\yvec=(y_\lambda)_{\lambda\in\hval}$ is defined for half-integer indices $\lambda\in\hval$ instead,
then these definitions are modified in the obvious way
to have $\D^1\yvec$ and $\D^2\yvec$ defined for integers $k\in\ival$ and half-integers $\kappa\in\hval$, respectively;
$\yvec$ needs to be augmented with additional values for $y_{-\imh}$ and $y_{\Kph}$ in this case.

Next, we introduce the set of monotone vectors
\begin{align*}
  \xseqN := \big\{ (x_0,\ldots,x_K) \,\big|\, x_0 < x_1 < \ldots < x_{K-1} < x_K \big\} \subseteq \setR^{K-1}.
\end{align*}
Each $\xvec\in\xseqN$ corresponds to a vector $\zvec=(z_{1/2},z_{3/2},\ldots,z_{K-1/2})$ of density values $z_\kappa$ via \eqref{eq:zvec}.
Our convention is that $z_{-\imh}= z_\imh$ and $z_{\Kph}=z_\Kmh$.
For a function $f:\xseqN\to\setR$, 
its first and second differential, $\grad f:\xseq\to\setR^{K-1}$ and $\grad^2f:\xseq\to\setR^{(K-1)\times(K-1)}$, respectively, 
are defined by $[\grad f(\xvec)]_k=\partial_{x_k} f(\xvec)$ and by $[\grad f(\xvec)]_{k\ell}=\partial_{x_k}\partial_{x_\ell} f(\xvec)$.
Further, $f$'s gradient $\wgrad f$ is given by $\wgrad f(\xvec) = \delta^{-1}\grad f(\xvec)$. 
For vectors $\vvec,\wvec\in\setR^{K-1}$, the scalar product $\wspr{\cdot}{\cdot}$ is defined 
by
\begin{align*}
  \wspr{\vvec}{\wvec} = \delta\sum_{k=0}^{K}v_kw_k, 
  \quad\text{with induced norm}\quad 
  \wnrm{\vvec} = \sqrt{\wspr{\vvec}{\vvec}}.
\end{align*}
\begin{xmp}
  Each component $z_\kappa$ of $\zvec=\cz_\theh[\xvec]$ is a function on $\xseqN$, 
  and
  \begin{align}
    \label{eq:zrule}
    \grad z_\kappa = -z_\kappa^2\,\frac{\ee_{\kappa+\frac12}-\ee_{\kappa-\frac12}}{\delta},
  \end{align}
  where $\ee_k\in\setR^{K-1}$ is the $k$th canonical unit vector, with the convention $\ee_0=\ee_K=0$. 
\end{xmp}
The main object of interest is the finite-dimensional submanifold $\densN$ of $\dens$ 
that consists of all locally constant density functions of the form $u=\cf_\theh[\xvec]$, 
with $\cf_\theh$ given in \eqref{eq:cf}, where $\xvec\in\xseqN$.
To each density function $u=\cf_\theh[\xvec]\in\densN$ we associate its \emph{Lagrangian map}
as the monotonically increasing function $\cX=\cX_\theh[\xvec]:\M\to\Omega$
that is piecewise linear with respect to $(\xi_0,\xi_1,\ldots,\xi_K)$ and satisfies $\cX(\xi_k)=x_k$ for $k=0,\ldots,K$.
The density $u$ and its Lagrangian map $\cX$ are related by
\[u\circ\cX=\frac{1}{\cX_\xi}.\]
\begin{rmk}\label{rmk:wass}
  In one space dimension, the Wasserstein metric on $\dens$ 
  is isometrically equivalent to the $L^2$-norm on the flat space of Lagrangian maps, see e.g. \cite{VilBook}.
  Our norm $\wnrm{\xvec-\yvec}$ is not identical 
  but equivalent to the $L^2$-norm between the Lagrangian maps $\cX_\theh[\xvec]$ and $\cX_\theh[\yvec]$.
  Consequently, there exist $K$-independent constants $c_1,c_2>0$, such that
  \begin{align}\label{eq:metricequivalent}
    c_1\wnrm{\xvec-\yvec} \leq \wass(\cf_\theh[\xvec],\cf_\theh[\yvec]) \leq c_2\wnrm{\xvec-\yvec},\quad\text{for all}\quad \xvec,\yvec\in\xseqN.
  \end{align}
  See \cite[Lemma 7]{dde} for a proof.
\end{rmk}
Next, consider two functionals $\Hhz,\entqz:\xseqN\to\setR$ given as follows: 
\begin{align*}
  \Hhz(\xvec) &= \Hh(\cf_\theh[\xvec]) 
  = \intom \cf_\theh[\xvec]\log\big(\cf_\theh[\xvec]\big)\dd x
  = \delta\sum_{k=1}^K \log(z_\kmh), \\
  \entqz(\xvec) &= \entq(\cf_\theh[\xvec]) 
  = \frac{1}{4}\intom \big(\cf_\theh[\xvec]\big)^2\dd x
  = \frac{\delta}{4}\sum_{k=1}^K z_\kmh.
\end{align*}
Here $\Hhz$ is just the restriction of the the entropy $\Hh$ to $\xseqN$,
and $\entqz$ is the restriction of the quadratic Renyi entropy
\begin{align*}
  \entq(u) = \frac14\intom u^2\dd x = \frac14\intM Z\dd\xi. 
\end{align*}
Using \eqref{eq:zrule}, we obtain an expicit representation of the gradients,
\begin{equation}
  \begin{split}\label{eq:gradHFzQz}
    \grad\Hhz(\xvec) = \delta\sum_{\kappa\in\hval} z_\kappa \frac{\ee_{\kappm}-\ee_{\kappp}}\delta,
    \qquad
    \grad\entqz(\xvec) = \frac{\delta}{4}\sum_{\kappa\in\hval} z_\kappa^2 \frac{\ee_{\kappm}-\ee_{\kappp}}\delta,
  \end{split}
\end{equation}
and --- for further reference --- also of the Hessians,
\begin{equation}
  \begin{split}\label{eq:HessHFzQz}
    \grad^2\Hhz(\xvec) &= \delta\sum_{\kappa\in\hval} z_\kappa^2 \left(\frac{\ee_{\kappm}-\ee_{\kappp}}\delta\right)\left(\frac{\ee_{\kappm}-\ee_{\kappp}}\delta\right)^T, \\
    \grad^2\entqz(\xvec) &= \frac{\delta}{2}\sum_{\kappa\in\hval} z_\kappa^3 \left(\frac{\ee_{\kappm}-\ee_{\kappp}}\delta\right)\left(\frac{\ee_{\kappm}-\ee_{\kappp}}\delta\right)^T.
  \end{split}
\end{equation}
A key property of our simple discretization ansatz is the preservation of convexity.
\begin{lem}
  The functionals $\Hhz$ and $\entqz$ are convex on $\xseq$.
\end{lem}
\begin{proof}
  This follows by inspection of the Hessians \eqref{eq:HessHFzQz}.  
\end{proof}
A conceptually different discretization is needed for the energy functional $\Di$ from \eqref{eq:info_tf}, 
which is identically $+\infty$ on $\densN$:
% We mimic property \eqref{eq:magic2} by defining
\begin{align}
  \label{eq:HS_step1}
  \Diz(\xvec) := \wspr{\wgrad\Hhz(\xvec)}{\wgrad\entqz(\xvec)}.
\end{align}
Substitution of the explicit representations \eqref{eq:gradHFzQz} in the definition \eqref{eq:HS_step1} yields
\begin{align*}
  \Diz(\xvec) = \wspr{\wgrad\Hhz(\xvec)}{\wgrad\entqz(\xvec)} 
  = \frac{\delta}{2}\sum_{k\in\hval}\frac{z_\kappp+z_\kappm}{2}\left(\frac{z_\kappp-z_\kappm}{\delta}\right)^2.
\end{align*}
It remains to define a discrete counterpart for the potential $\pot$.
A change of variables yields in the definition in \eqref{eq:info_tf} yields
\begin{align*}
  \pot(u) = \intom V(x) u(x) \dd x = \intM V(\theX) \dd\xi,
\end{align*}
Thus, a natural discretization $\potz$ of $\pot$ is given by
\begin{align*}
  \potz(\xvec) = \delta\sum_{k\in\ival} V(x_k).
\end{align*}
In summary, our discretization $\DiVz$ of $\DiV$ is
\begin{align*}
  \DiVz(\xvec) = \Diz(\xvec) + \potz(\xvec) = \wspr{\wgrad\Hhz(\xvec)}{\wgrad\entqz(\xvec)} + \potz(\xvec).
\end{align*}
%
% The weighted gradient $\wgrad\DiVz$ of $\DiVz$ can now be determined by those of $\Hhz$, $\entqz$, and $\potz$
% \begin{align}\label{eq:gradTz}
%   \wgrad\DiVz(\xvec) = \wspr{\wgrad\Hhz(\xvec)}{\wgrad^2\entqz(\xvec)} + \wspr{\wgrad\entqz(\xvec)}{\wgrad^2\Hhz(\xvec)} + \wgrad\potz(\xvec).
%\end{align}
%
%

%-----------------------------------------------------------------------------------------------------------------------------------------------------------------
\subsection{Discretization in time}\label{sec:timestepping_tf}
Next, the spatially discrete gradient flow equation
\begin{align}
  \label{eq:sdgradflow_tf}
  \dot{\xvec} = -\wgrad\DiVz(\xvec) 
\end{align}
is discretized also in time, using \emph{minimizing movements}.
To this end, fix a time step with $\tau>0$; 
we combine the spatial and temporal mesh widths in a single discretization parameter $\Delta=(\tau;\theh)$.
For each $\yvec\in\xseqN$, introduce the \emph{Yosida-regularized energy} $\DiVy(\cdot;\yvec):\xseqN\to\setR$ by
\begin{align*}
  \DiVy(\xvec;\yvec) = \frac1{2\tau}\wnrm{\xvec-\yvec}^2+\DiVz(\xvec).
\end{align*}
A fully discrete approximation $\xvec_\Delta=(\xvec_\Delta^0,\xvec_\Delta^1,\ldots,\xvec_\Delta^n,\ldots)$ of \eqref{eq:sdgradflow_tf} is now defined inductively
from a given initial datum $\xvec_\Delta^0$ by choosing each $\xvec_\Delta^n$ 
as a global minimizer of $\DiVy(\cdot;\xvec_\Delta^{n-1})$.
Below, we prove that such a minimizer always exists, see Lemma \ref{lem:cfl}.

In practice, one wishes to define $\xvec_\Delta^n$ as --- preferably unique --- solution 
of the Euler-Lagrange equations associated to $\DiVy(\cdot;\xvec_\Delta^{n-1})$,
which leads to the implicit Euler time stepping:
\begin{align}
  \label{eq:euler_tf}
  \frac{\xvec-\xvec_\Delta^{n-1}}{\tau} = -\wgrad\DiVz(\xvec)
  = -\frac1{\delta^2}\left(\grad^2\Hhz(\xvec)\cdot\grad\entqz(\xvec)+\grad^2\entqz(\xvec)\cdot\grad\Hhz(\xvec)\right).
\end{align}
Using \eqref{eq:gradHFzQz} and \eqref{eq:HessHFzQz}, a straight-forward calculation shows that 
\eqref{eq:euler_tf} is the precisely the numerical scheme \eqref{eq:dgf_tf} from the introduction.
Equivalence of \eqref{eq:euler_tf} and the minimization problem for $\DiVy$ is guaranteed at least for sufficiently small $\tau>0$.
\begin{prp}
  \label{prp:wellposed}
  For each discretization $\Delta$ and every initial condition $\xvec^0\in\xseqN$,
  the sequence of equations \eqref{eq:euler_tf} can be solved inductively.
  Moreover, if $\tau>0$ is sufficiently small with respect to $\delta$ and $\DiVz(\xvec^0)$,
  then each equation \eqref{eq:euler_tf} possesses a unique solution with $\DiVz(\xvec)\le\Diz(\xvec^0)$,
  and that solution is the unique global minimizer of $\DiVy(\cdot;\xvec_\Delta^{n-1})$.
\end{prp}
\begin{rmk}
  In principle, the proof of Lemma \ref{lem:cfl} below provides a criterion on the smallness of $\tau>0$ 
  that would guarantee the unique solvability of \eqref{eq:euler_tf}.
  We shall not make this criterion explicit, %which would be a quite involved expression,
  since in practice, we observe that the Newton method applied to \eqref{eq:euler_tf} and initialized with $\xvec_\Delta^{n-1}$
  always converges to ``the right'' solution $\xvec_\Delta^n$,
  even for comparatively large steps $\tau$ and in rather degenerate situations;
  we refer the reader to our numerical results in Section \ref{sec:num_tf}.
\end{rmk}
The proof of this proposition is a consequence of the following rather technical lemma.
\begin{lem}
  \label{lem:cfl}
  Fix a spatial discretization parameter $\delta$, and let $C:=\DiV(\xvec^0)$.
  Then for every $\yvec\in\xseqN$ with $\DiVz(\yvec)\le C$, the following are true:
  \begin{itemize}
  \item For each $\tau>0$, 
    the function $\DiVy(\cdot;\yvec)$ possesses at least one global minimizer $\xvec^*\in\xseqN$,
    and that $\xvec^*$ satisfies the Euler-Lagrange equation
    \[ \frac{\xvec^*-\yvec}\tau = -\wgrad\DiVz[\xvec^*]. \]
  \item There exists a $\tau_C>0$ independent of $\yvec$ such that for each $\tau\in(0,\tau_C)$,
    the global minimizer $\xvec^*\in\xseqN$ is strict and unique, 
    and it is the only critical point of $\DiVy(\cdot;\yvec)$ with $\DiVz(\xvec)\le C$.
  \end{itemize}
\end{lem}
\begin{proof}
  Fix $\yvec\in\xseqN$ with $\DiVz(\yvec)\leq C$, 
  and define the nonempty (since it contains $\yvec$) sublevel set $A_C:=\big(\DiVy(\cdot,\yvec)\big)^{-1}([0,C])\subset\xseqN$.
  %Since $\DiVz\ge0$, any $\xvec\in A_C$ satisfies $\wnrm{\yvec-\xvec}\le\sqrt{2\tau C}$,
  %and therefore
  %\begin{align}\label{eq:supp1}
    %y_0 - \sqrt{\frac{2\tau C}\delta} \le x_0<x_K\le y_K + \sqrt{\frac{2\tau C}\delta}.
  %\end{align}
  %Hence, there is an interval $[-L,L]$, such that all components of an arbitrary $\xvec\in A_C$ lie in $[-L,L]$.
  Let $\zvec=\cz_\theh[\xvec]$, and observe that $z_\kappa\ge\delta/(b-a)$ for each $z\in\hval$.
  From here, it follows further that
	\begin{align*}
    z_\kappa - \frac\delta{b-a}  
    &\le \sum_{k\in\ivalp}|z_\kph - z_\kmh| \\ %\label{eq:zbound}\\
    &\le \left(\sum_{k\in\ivalp}\frac{\delta}{z_\kph+z_\kmh}\right)^{\frac12} 
    \left(\delta\sum_{k\in\ivalp}(z_\kph+z_\kmh)\left(\frac{z_\kph- z_\kmh}{\delta}\right)^2\right)^{\frac12}\\
    &\le \big(2(b-a)\big)^{1/2}\DiVz(\xvec)^{1/2} \le (4(b-a)C)^{1/2}.
  \end{align*}
  This implies that the differences $x_\kappp-x_\kappm=\delta/z_\kappa$ have a uniform positive lower bound on $A_C$.
  It follows that $A_C$ is a compact subset in the interior of $\xseqN$.
  Consequently, the continuous function $\DiVy(\cdot;\yvec)$ attains a global minimum at $\xvec^*\in\xseqN$.
  Since $\xvec^*\in A_C$ lies in the interior of $\xvec$, 
  it satisfies $\grad\DiVy(\xvec^*;\yvec)=0$, which is the Euler-Lagrange equation.
  This proves the first claim.
  
  Since $\DiVz:\xseqN\to\setR$ is smooth, 
  its restriction to the compact set $A_C$ is $\lambda_C$-convex with some $\lambda_C\in\setR$,
  i.e., $\grad^2\DiVz(\xvec)\ge\lambda_C\eins_{K-1}$ for all $\xvec\in A_C$.
  Independently of $\yvec$, we have that
  \begin{align*}
    \grad^2\DiVy(\xvec,\yvec) = \grad^2\DiVz(\xvec) + \frac\delta\tau\eins_{K-1},
  \end{align*}
  which means that $\xvec\mapsto\DiVy(\xvec,\yvec)$ is strictly convex on $A_C$
  if
  \begin{align*}
    0< \sigma < \tau_C:=\frac\delta{(-\lambda_C)}.
  \end{align*}
  Consequently, each such $\DiVy(\cdot,\yvec)$ has at most one critical point $\xvec^*$ in the interior of $A_C$, 
  and this $\xvec^*$ is necessarily a strict global minimizer.
\end{proof}

%\begin{lem}
%The family of minimizers $u_\Delta$ has uniform bounded second moment.
%\end{lem}
%%
%\begin{proof}
%We prove, that
%\begin{align*}
	%\intom V(x) u_\Delta(x) \dd x < +\infty,
%\end{align*}
%which is actually more then the boundedness of the second moment, due to the growth property of $V$ for $|x|\to+\infty$.
%Since $u_\Delta$ is local constant, we can apply the mean value theorem on the anti-derivative of $V$ and get
%\begin{align*}
	%\intom V(x) u_\Delta(x) \dd x
	%= \sum_{k=1}^K \frac{\delta}{x_k-x_{k-1}}\int_{x_{k-1}}^{x_k} V(x) \dd x
	%= \delta\sum_{k=1}^K V\big((1-s_k)x_{k-1} + s_kx_k)\big),
%\end{align*}
%for values $s_k\in[0,1]$. 
%Due to the growth property of $V$, there exists a value $L>0$, such that $|x|^2\leq V(x)$ for any $|x|>L$. 
%Hence the $\lambda$-convexity and $|s_k|,|1-s_k|\leq 1$ yields %(we assume $\lambda<0$, the other case is trivial)
%\begin{align*}
	%&\delta\sum_{k=1}^K V\big((1-s_k)x_{k-1} + s_kx_k)\big) \\
	%\leq& \delta\sum_{k=1}^K (1-s_k)V(x_{k-1}) + s_k V(x_k) + \min\{0,\lambda\}\delta\sum_{k=1}^K\frac{(1-s_k)s_k}{2}|x_k-x_{k-1}|^2 \\
	%\leq& \delta\sum_{k=1}^{K-1} \big((1-s_k) + s_{K-1}\big)V(x_k) + \delta s_0V(x_0) + \delta (1-s_K)V(x_K) \\
	%&\quad			+ 2\delta L^2|\lambda| + \frac{\delta}{2}\sum_{\substack{k=1,\ldots,K \\ |x_k|>L}}|x_k|^2+|x_{k-1}|^2 \\
	%\leq& 3\delta\sum_{k=0}^K V(x_k) + 2L^2|\lambda|. 
%\end{align*}
%\end{proof}

%
%-----------------------------------------------------------------------------------------------------------------------------------------------------------------
\subsection{Spatial interpolations}
Consider a fully discrete solution $\xvec_\Delta=(\xvec_\Delta^0,\xvec_\Delta^1,\ldots)$.
For notational simplification, 
we write the entries of the vectors $\xvec_\Delta^n$ and $\zvec_\Delta^n=\cz_\theh[\xvec_\Delta^n]$ 
as $x_k$ and $z_\kappa$, respectively, whenever there is no ambiguity in the choice of $\Delta$ and the time step $n$.

Recall that $u_\Delta^n=\cf_\theh[\xvec_\Delta^n]\in\densN$ defines a sequence of densitites on $\Omega$
which are piecewise constant with respect to the (non-uniform) grid $(a,x_1,\ldots,x_{K-1},b)$.
To facilitate the study of convergence of weak derivatives, 
we introduce also \emph{piecewise affine} interpolations $\hatz_\Delta^n:\M\to\setRpp$ and $\hatu_\Delta^n:\Omega\to\setRpp$. 

In addition to $\xi_k=k\delta$ for $k\in\ival$, 
introduce the intermediate points $\xi_\kappa=\kappa\delta$ for $\kappa\in\hval$.
Accordingly, introduce the intermediate values for the vectors $\xvec_\Delta^n$ and $\zvec_\Delta^n$:
\begin{align*}
  x_\kappa = \frac12\big(x_\kappp+x_\kappm) \quad \text{for $\kappa\in\hval$}, \\
  z_k = \frac12\big(z_\kph+z_\kmh\big) \quad \text{for $k\in\ivalp$}.
\end{align*}
Now define 
\begin{itemize}
\item $\hatz_\Delta^n:\M\to\setR$ as the piecewise affine interpolation
  of the values $(z_{\frac12},z_{\frac32},\ldots,z_{K-\frac12})$ 
  with respect to the equidistant grid $(\xi_{\frac12},\xi_{\frac32},\ldots,\xi_{K-\frac12})$,
  and
\item $\hatu_\Delta^n:\Omega\to\setR$ as the piecewise affine function with
  \begin{align}
    \label{eq:locaffine0}
    \hatu_\Delta^n\circ\theX_\Delta^n = \hatz_\Delta^n.
    % \widetilde{\cf_\theh}[\xvec] \circ\cX_\theh[\xvec] = \widehat{\cfM_\theh}[\xvec].
  \end{align}
\end{itemize}
Our convention is that 
$\hatz_\Delta^n(\xi)=z_{\frac12}$ for $0\le\xi\le\delta/2$ and $\hatz_\Delta^n(\xi)=z_{K-\frac12}$ for $M-\delta/2\le\xi\le M$,
and accordingly
$\hatu_\Delta^n(x)=z_{\frac12}$ for $x\in[a,x_{\frac12}]$ and $\hatu_\Delta^n(x)=z_{K-\frac12}$ for $x\in[x_{K-\frac12},b]$.
The definitions have been made such that
\begin{align}
  \label{eq:interpol1}
  x_k = \theX_\Delta^n(\xi_k),\quad z_k = \hatz(\xi_k) = \hatu (x_k) \quad \text{for all $k\in\ival\cup\hval$}.
\end{align}
Notice that $\hatu_\Delta^n$ is piecewise affine with respect to the ``double grid'' $(x_0,x_\frac12,x_1,\ldots,x_{K-\frac12},x_K)$,
but in general not with respect to the subgrid $(x_0,x_1,\ldots,x_K)$.
By direct calculation, we obtain for each $k\in\ivalp$ that
\begin{equation}
  \label{eq:tildeux}
  \begin{split}
    \partial_x\hatu_\Delta^n\big|_{(x_\kmh,x_k)}  
    &=\frac{z_k - z_\kmh}{x_k-x_\kmh} = \frac{z_\kph - z_\kmh}{x_k-x_{k-1}} 
    =z_\kmh\frac{z_\kph - z_\kmh}{\delta}, \\
    \partial_x\hatu_\Delta^n\big|_{(x_k,x_\kph)} 
    &= \frac{z_\kph - z_k}{x_\kph-x_k} = \frac{z_\kph - z_\kmh}{x_{k+1}-x_k} 
    = z_\kph\frac{z_\kph - z_\kmh}{\delta}.
  \end{split}
\end{equation}
Trivially, we also have that $\partial_x\hatu$ vanishes identically on the intervals $(a,x_\imh)$ and $(x_\Kmh,b)$.

\subsection{A discrete Sobolev-type estimate}
The following inequality plays a key role in our analysis.
Recall the conventions that $z_{-\imh} = z_{\imh}$, $z_\Kph=z_\kmh$, and that $z_k=\frac12(z_\kph+z_\kmh)$.
\begin{lem}\label{lem:L4D2_tf}
  For any $\xvec\in\xseqN$,
  \begin{align}\label{eq:L4D2_tf}
    \delta\sum_{k\in\ivalp} z_k\left(\frac{z_\kph - z_\kmh}{\delta}\right)^4 
    \leq \frac{9}{4} \delta\sum_{\kappa\in\hval} z_\kappa^3\left(\frac{z_{\kappa+1}-2z_\kappa+z_{\kappa-1}}{\delta^2}\right)^2.
  \end{align}
\end{lem}
\begin{proof}
  Define the left-hand side in \eqref{eq:L4D2_tf} as $(A)$.
  Then:
  \begin{align*}
    (A)&=\delta^{-3}\sum_{k\in\ivalp} z_k(z_\kph - z_\kmh)^3(z_\kph - z_\kmh) \\
    &= \delta^{-3}\sum_{k=1}^K z_\kmh\left[z_{k-1}(z_\kmh-z_\kmd)^3 - z_k(z_\kph-z_\kmh)^3\right] \\ 
    &=\frac{\delta^{-3}}{2}\sum_{k=1}^K z_\kmh\left[(z_\kmh+z_\kmd)(z_\kmh-z_\kmd)^3 - (z_\kph+z_\kmh)(z_\kph-z_\kmh)^3\right] \\
    &=\frac{\delta^{-3}}{2}\sum_{k=1}^K z_\kmh\Big[(z_\kmd-z_\kmh)(z_\kmh-z_\kmd)^3 + (z_\kmh-z_\kph)(z_\kph-z_\kmh)^3 \\
    &\qquad\qquad\qquad\qquad			+ 2z_\kmh(z_\kmh-z_\kmd)^3 - 2z_\kmh(z_\kph-z_\kmh)^3 \Big].
  \end{align*}
  Rearranging terms yields
  \begin{equation*}
    \begin{split}%\label{eq:importantstep}
      (A) &= -(A) -\delta^{-3}\sum_{k=1}^K z_\kmh^2\left[(z_\kmh-z_\kmd)^3 - (z_\kph-z_\kmh)^3\right]
	\end{split}\end{equation*}
	and further using the identity $(a^3-b^3)=(a-b)(a^2+b^2+ab)$,
	\begin{equation*}
    \begin{split}%\label{eq:importantstep}
      (A) &= -\frac{\delta^{-3}}{2}\sum_{k=1}^K z_\kmh^2\left[(z_\kmh-z_\kmd)^3 - (z_\kph-z_\kmh)^3\right] \\
      &=-\frac{\delta^{-1}}{2}\sum_{k=1}^K z_\kmh^2[\D^2 z]_\kmh\Big[(z_\kmh-z_\kmd)^2 + (z_\kph-z_\kmh)^2 \\
					&\qquad\qquad\qquad\qquad \qquad\qquad\quad+ (z_\kmh-z_\kmd)(z_\kph-z_\kmh)\Big].
    \end{split}
  \end{equation*}
  Invoke H\"older's inequality and the elementary estimate $ab\leq\frac{1}{2}(a^2+b^2)$ to conclude that
  \begin{equation}
    \begin{split}%\label{eq:L4Hi}
      (A) &\leq 
      \frac{1}{2}\left(\delta\sum_{\kappa\in\hval} z_\kappa^3[\D^2 z]_\kappa^2\right)^{\frac{1}{2}}
      \left(\delta^{-3}\sum_{k=1}^K z_\kmh\frac{9}{4}\left[(z_\kmh-z_\kmd)^2 + (z_\kph-z_\kmh)^2\right]^2\right)^{\frac{1}{2}} \\
      &\leq \frac{3}{2}\left(\delta\sum_{\kappa\in\hval} z_\kappa^3[\D^2 z]_\kappa^2\right)^{\frac{1}{2}}
      \left(\delta^{-3}\sum_{k=1}^K \frac{z_\kmh}{2}\left[(z_\kmh-z_\kmd)^4 + (z_\kph-z_\kmh)^4\right]\right)^{\frac{1}{2}} \\
      &= \frac{3}{2}\left(\delta\sum_{\kappa\in\hval} z_\kappa^3[\D^2 z]_\kappa^2\right)^{\frac{1}{2}} (A)^{\frac{1}{2}},
    \end{split}
  \end{equation}
  where we have used an index shift and the conventions $z_{-\imh} = z_{\imh}$, $z_\Kph=z_\kmh$ in the last step
\end{proof}

%-----------------------------------------------------------------------------------------------------------------------------------------------------------------------
%-----------------------------------------------------------------------------------------------------------------------------------------------------------------------
%-----------------------------------------------------------------------------------------------------------------------------------------------------------------------
\section{A priori estimates and compactness}\label{sec:apriori}
%
%
%-----------------------------------------------------------------------------------------------------------------------------------------------------------------------
\subsection{Energy and entropy dissipation}
Fix some discretization parameters $\Delta=(\tau;\theh)$.
Below, we derive a priori bounds on fully discrete solutions $(\xvec_\Delta^n)_{n=0}^\infty$ 
that are independent of $\Delta$.
Specifically, we shall prove two essential estimates: 
the first one is monotonicity of the energy $\DiVz$, 
the second one is obtained from the dissipation of the auxiliary Lyapunov functional $\Hhz$.
We begin with the classical energy estimate.
\begin{prp}
  \label{prp:energy}
  One has that $\DiVz$ is monotone, i.e., $\DiVz(\xvec_\Delta^n)\le\DiVz(\xvec_\Delta^{n-1})$, 
  and further:
  \begin{align}
    &\DiVz(\xvec_\Delta^n) \leq \DiVz(\xvec_\Delta^0) \quad\textnormal{for all }n\geq 0, \label{eq:monoton}\\
    &\|\xvec_\Delta^{\overline n} - \xvec_\Delta^{\underline n}\|_\delta^2
    \leq 2\DiVz(\xvec_\Delta^0)\,(\overline n-\underline n)\tau
    \quad\textnormal{for all }\overline n\geq\underline n\geq 0, \label{eq:uniform_time_tf} 
    \\
    &\tau\sum_{n=1}^\infty\wnrm{\frac{\xvec_\Delta^n-\xvec_\Delta^{n-1}}\tau}^2
    = \tau \sum_{n=1}^\infty\wnrm{\wgrad\DiVz(\xvec_\Delta^n)}^2
    \le 2\DiVz(\xvec_\Delta^0). \label{eq:eee_tf}
  \end{align}
\end{prp}
\begin{proof}
  The monotonicity \eqref{eq:dissipation_tf} follows (by induction on $n$) from the definition of $\xvec_\Delta^n$ 
  as minimizer of $\DiVy(\cdot;\xvec_\Delta^{n-1})$:
  \begin{align}
    \label{eq:emon}
    \DiVz(\xvec_\Delta^n) &\le \frac1{2\tau}\|\xvec_\Delta^n-\xvec_\Delta^{n-1}\|_\delta^2 + \DiVz(\xvec_\Delta^n) 
    =\DiVy(\xvec_\Delta^n;\xvec_\Delta^{n-1}) \le \DiVy(\xvec_\Delta^{n-1};\xvec_\Delta^{n-1}) = \DiVz(\xvec_\Delta^{n-1}).
  \end{align}
  Moreover, summation of these inequalities from $n=\underline n+1$ to $n=\overline n$ yields
  \begin{align*}
    \frac\tau2\sum_{n=\underline n+1}^{\overline n} \bigg[\frac{\|\xvec_\Delta^n-\xvec_\Delta^{n-1}\|_\delta}{\tau}\bigg]^2
    \le \DiVz(\xvec_\Delta^{\underline n})-\DiVz(\xvec_\Delta^{\overline n}) \le \DiVz(\xvec_\Delta^0).
  \end{align*}
  For $\underline n=0$ and $\overline n\to\infty$, we obtain the first part of \eqref{eq:eee_tf}.
  The second part follows by \eqref{eq:euler_tf}.
  If instead we combine the estimate with Jensen's inequality,
  we obtain
  \begin{align*}
    \big\|\xvec_\Delta^{\overline n}-\xvec_\Delta^{\underline n}\big\|_\delta
    \le \tau\sum_{n=\underline n+1}^{\overline n}\frac{\big\|\xvec_\Delta^n-\xvec_\Delta^{n-1}\big\|_\delta}{\tau}
    \le \bigg(\tau\sum_{n=\underline n+1}^{\overline n} \bigg[\frac{\|\xvec_\Delta^n-\xvec_\Delta^{n-1}\|_\delta}{\tau}\bigg]^2\bigg)^{1/2}
    \big(\tau(\overline n-\underline n)\big)^{1/2},
  \end{align*}
  which leads to \eqref{eq:uniform_time_tf}.
\end{proof}
The previous estimates were completely general.
The following estimate is very particular for the problem at hand.
\begin{lem}\label{lem:dissipation_tf}
  One has that $\Hhz$ is monotone, i.e., $\Hhz(\xvec_\Delta^n)\le\Hhz(\xvec_\Delta^{n-1})$.
  Moreover, it holds for any $T>0$ that
  \begin{align}\label{eq:dissipation_tf}
    \tau\sum_{n=0}^N\delta\sum_{\kappa\in\hval}(z_\kappa^n)^3 \left(\frac{z_{\kappa+1}-2z_\kappa+z_{\kappa-1}}{\delta^2}\right)^2
    \le 4\big(\olHh + \Lambda M(T+1)\big),
  \end{align}
  for each $\Nt\in\setN$ with $\Nt\tau\in(T,T+1)$.
\end{lem}
\begin{proof}%[Proof of Lemma \ref{lem:dissipation_tf}]
  Convexity of $\Hhz$ implies that
  \begin{align*}
    \Hhz(\xvec_\Delta^{n-1}) - \Hhz(\xvec_\Delta^{n}) 
    \ge \wspr{\wgrad\Hhz(\xvec_\Delta^n)}{\xvec_\Delta^{n-1}-\xvec_\Delta^n}
    = \tau\wspr{\wgrad\Hhz(\xvec_\Delta^n)}{\wgrad\DiVz(\xvec_\Delta^n)},
  \end{align*}
  for each $n=1,\ldots,\Nt$.
  Summation of these inequalities over $n$ yield
  \begin{align}
    \label{eq:help963}
    \tau\sum_{n=1}^{\Nt}\wspr{\wgrad\Hhz(\xvec_\Delta^n)}{\wgrad\DiVz(\xvec_\Delta^n)}
    \le \Hhz(\xvec_\Delta^0)-\Hhz(\xvec_\Delta^N).
  \end{align}
  To estimate the right-hand side in \eqref{eq:help963},
  observe that $\Hhz(\xvec_\Delta^0)\le\olHh$ by hypothesis,
  and that $\Hhz(\xvec_\Delta^N)$ is bounded from below thanks to 
	the convexity of $s\mapsto s\ln(s)$ and Jensen's inequality,
	which yieds for any $\xvec\in\xseqN$
	\begin{align*}
		\Hhz(\xvec)=\intom\cf_\thep[\xvec]\ln\cf_\thep[\xvec]\dd x \geq M\ln\left(\frac{M}{b-a}\right).
	\end{align*}

  We turn to estimate the left-hand side in \eqref{eq:help963} from below.
  Recall that $\DiVz=\Diz+\potz$.
  For the component corresponding to $\potz$, we find, using \eqref{eq:gradHFzQz} and \eqref{eq:Vhypo},
  \begin{align*}
    \wspr{\wgrad\Hhz(\xvec_\Delta^n)}{\wgrad\potz(\xvec_\Delta^n)}
    &= \delta\sum_{\kappa\in\hval}z_\kappa \frac{V_x(x_\kappm)-V_x(x_\kappp)}{\delta} \\
    &\ge \left(\inf_{x\in\setR}V_{xx}(x)\right)\, \delta\sum_{\kappa\in\hval}z_\kappa\frac{x_\kappm-x_\kappp}{\delta}
    \ge -\Lambda M.
  \end{align*}
  The component corresponding to $\Diz$ is more difficult to estimate.
  Thanks to \eqref{eq:gradHFzQz}\&\eqref{eq:HessHFzQz},
  we have that
  \begin{align*}
    &4\wspr{\wgrad\Diz(\xvec)}{\wgrad\Hhz(\xvec)}  \\
    &= 4\wspr{\wgrad\Hhz(\xvec)}{\wgrad^2\entqz(\xvec)\wgrad\Hhz(\xvec)} + 4\wspr{\wgrad\entqz(\xvec)}{\wgrad^2\Hhz(\xvec)\wgrad\Hhz(\xvec)} \\
    &= 2\delta\sum_{\kappa\in\hval} z_\kappa^3 \left(\frac{z_{\kappa+1}-2z_\kappa+z_{\kappa-1}}{\delta^2}\right)^2
    +\delta\sum_{\kappa\in\hval} z_\kappa^2 \left(\frac{z_{\kappa+1}-2z_\kappa+z_{\kappa-1}}{\delta^2}\right)\left(\frac{z_{\kappa+1}^2-2z_\kappa^2 + z_{\kappa-1}^2}{\delta^2}\right).
  \end{align*}
  Further estimates are needed to control the second sum from below.
  Observing that
  \begin{align*}
    \frac{z_{\kappa+1}^2-2z_\kappa^2 + z_{\kappa-1}^2}{\delta^2}
    = 2z_\kappa\frac{z_{\kappa+1}-2z_\kappa+z_{\kappa-1}}{\delta^2} 
    + \left(\frac{z_{\kappa+1}-z_\kappa}\delta\right)^2 + \left(\frac{z_{\kappa-1}-z_\kappa}\delta\right)^2,
  \end{align*}
  and that $2ab\ge-\frac32a^2-\frac23b^2$ for arbitrary real numbers $a,\,b$,
  we conclude that
  \begin{align*}
    &4\wspr{\wgrad\Diz(\xvec)}{\wgrad\Hhz(\xvec)}  \\
    &\ge \left(4-\frac32\right)\delta\sum_{\kappa\in\hval} z_\kappa^3 \left(\frac{z_{\kappa+1}-2z_\kappa+z_{\kappa-1}}{\delta^2}\right)^2
    - \frac{2\delta}3\sum_{\kappa\in\ival} z_k \left(\frac{z_\kph-z_\kmh}{\delta}\right)^4.
  \end{align*}
  Now apply inequality \eqref{eq:L4D2_tf}.
\end{proof}

\subsection{Compactness}
The following lemma contains the key estimate to derive compactness of fully discrete solutions in the limit $\Delta\to0$.
Below, we prove that from the entropy dissipation \eqref{eq:dissipation_tf},
we obtain a control on the total variation of $\partial_x\hatu_\Delta^n$.

Several equivalent definitions of the \emph{total variation} of $f\in L^1(\Omega)$ exist.
In case of piecewise smooth functions with jump discontinuities, the most appropriate definition is
\begin{align}
  \label{eq:defTVL_tf}
  \tv{f} = \sup\left\{ \sum_{j=0}^{J-1} |f(r_{j+1})-f(r_{j})|\,:\,
    J\in\setN,\, a<r_0<r_2<\cdots<r_J<b\right\}.
\end{align}
%\begin{align}
  %\label{eq:defTV}
  %\tv{f} = \sup\left\{ \sum_{j=0}^J |f(r_{j+1})-f(r_{j})|\,:\,
  %J\in\setN,\,r_0<r_2<\cdots<r_J\right\}.
%\end{align}
Further recall the notation
\begin{align*}
  \llbracket f\rrbracket_{\bar x} = \lim_{x\downarrow\bar x}f(x) - \lim_{x\uparrow\bar x}f(x).
\end{align*}
for the height of the jump in $f(x)$'s value at $x=\bar x$. 
% For our proof of convergence, it suffice to show the boundedness of the total variation of $\partial_x\hatz_\Delta$
% on any arbitrary bounded set $[-L,L]\subseteq\setR$.
% One can also ask for the total variation of functions restricted on compact intervals $[-L,L]$ with $L>0$, 
%\begin{align}
  %\label{eq:defTVL_tf}
  %\tv{f} = \sup\left\{ \sum_{j=0}^J |f(r_{j+1})-f(r_{j})|\,:\,
  %J\in\setN,\,-L<r_0<r_2<\cdots<r_J<L\right\}.
%\end{align}
%
\begin{prp}
  \label{prp:BVbound_tf}
  For any $T>0$ and $\Nt\in\setN$ with $\tau\Nt\in(T,T+1)$, one has 
  \begin{align}
    \label{eq:BVbound_tf}
    \tau\sum_{n=1}^\Nt \tv{\partial_x\hatu_\Delta^n}^2 \le \check C_T .
    % (\olHh + \Lambda TM).%50L(\olH + \max\{\lambda,\Lambda\}\olQ).
  \end{align}
\end{prp}
\begin{rmk}
  The proof below yields for $\check C_T$ the explicit value
  \begin{align*}
    \check C_T = 50\big(\olHh + \Lambda M(T+1)\big). %25\hat C_T,
  \end{align*}
\end{rmk}
\begin{proof}
  Fix $n$. 
  % For the sake of simplicity, assume $[x_{-\imh},x_\Kph]\subseteq[-L,L]$. 
  The function $\partial_x\hatu_\Delta^n$ is locally constant on each interval $(x_\kmh,x_k)$, 
  % $k\in\ival\cup\hval\cup\{\Kph\}$ 
  and equal to zero elsewhere. 
  Therefore, the total variation of $\partial_x\hatu_\Delta^n$ is given by the sum over all jumps at the points of discontinuity,
  \begin{align}
    \label{eq:tvl0}
    \tv{\partial_x\hatu_\Delta^n} 
    &= \sum_{k\in\ivalp} \left|\llbracket \partial_x\hatu_\Delta^n\rrbracket_{x_k}\right| 
    + \sum_{\kappa\in\hval} \left|\llbracket \partial_x\hatu_\Delta^n\rrbracket_{x_\kappa}\right|.
  \end{align}
  The jumps can be evaluated by direct calculation:
  \begin{equation}
    \label{eq:hatux}
    \begin{split}
      \partial_x\hatu\big|_{(x_\kmh,x_k)}  
      &=\frac{z_k - z_\kmh}{x_k-x_\kmh} = \frac{z_\kph - z_\kmh}{x_k-x_{k-1}} 
      =z_\kmh\frac{z_\kph - z_\kmh}{\delta}, \quad\textnormal{for } k\in\ival\backslash\{0\}, \\
      \partial_x\hatu\big|_{(x_k,x_\kph)} 
      &= \frac{z_\kph - z_k}{x_\kph-x_k} = \frac{z_\kph - z_\kmh}{x_{k+1}-x_k}
      = z_\kph\frac{z_\kph - z_\kmh}{\delta}, \quad\textnormal{for } k\in\ival\backslash\{K\}.
    \end{split}
  \end{equation}
  This implies that
  \begin{align*}
    \left|\llbracket\partial_x\hatu_\Delta^n\rrbracket_{x_k}\right| = \delta\left(\frac{z^n_\kph-z^n_\kmh}\delta\right)^2
    &\quad\text{for $k\in\ivalp$},\\
    % \label{eq:jumpy}
    \left|\llbracket\partial_x\hatu_\Delta^n\rrbracket_{x_\kappa}\right| = \delta z^n_\kappa\left(\frac{z^n_{\kappa+1}-2z^n_\kappa+z^n_{\kappa-1}}{\delta^2}\right)
    &\quad\text{for $\kappa\in\hval$}.
  \end{align*}
  %Moreover, thanks to $z^n_{-\frac12}=z^n_{K+\frac12}=0$ and our definitions of $x^n_{-\imh}$, $x^n_{\Kph}$,
  %one verifies easily that
  %\begin{align*}
    %\llbracket \partial_x\hatu_\Delta^n\rrbracket_{x_0} &= \llbracket \partial_x\hatu_\Delta^n\rrbracket_{x_K} = 0 \\
    %\left|\llbracket \partial_x\hatu_\Delta^n\rrbracket_{\bar x_{-\imh}}\right| 
    %&= \frac{|z_0 - z_{-\imh}|}{x_0-x_{-\imh}} = \frac{z_\imh}{x_1-x_0} = \delta\left(\frac{z_\imh-z_{-\imh}}{\delta}\right)^2 \quad\textnormal{and} \\
    %\left|\llbracket \partial_x\hatu_\Delta^n\rrbracket_{\bar x_\Kph}\right| 
    %&= \frac{|z_\Kph - z_K|}{x_\Kph-x_K} = \frac{z_\Kmh}{x_K-x_{K-1}} = \delta\left(\frac{z_\Kph-z_\Kmh}{\delta}\right)^2
  %\end{align*}
  We substitute this into \eqref{eq:tvl0}, use Hölder's inequality, and apply \eqref{eq:L4D2_tf} 
  to obtain as a consequence of elementary estimates that
  \begin{align*}
    \tv{\partial_x\hatu_\Delta^n} 
    &\le\delta\sum_{k\in\ivalp}\left(\frac{z_\kph^n-z_\kmh^n}{\delta}\right)^2 
    +\delta\sum_{\kappa\in\hval}z_\kappa^n\left|[\D^2\zvec_\Delta^n]_\kappa\right| \\
    &\leq \left(\sum_{k\in\ivalp}\frac{\delta}{z_k}\right)^{\frac{1}{2}}
    \left[\left(\delta\sum_{k\in\ivalp}\left(\frac{z_\kph^n-z_\kmh^n}{\delta}\right)^4\right)^{\frac{1}{2}}
      +\left(\delta\sum_{\kappa\in\hval}(z_\kappa^n)^3[\D^2\zvec_\Delta^n]_\kappa^2\right)^{\frac{1}{2}}\right] \\
    &\leq \frac52(2(b-a))^{\frac{1}{2}}\left(\delta\sum_{\kappa\in\hval}(z_\kappa^n)^3[\D^2\zvec_\Delta^n]_\kappa^2\right)^{\frac{1}{2}}.
    % \leq 5\sqrt{2L\wspr{\wgrad\Diz(\xvec_\Delta^n)}{\wgrad\Hhz(\xvec_\Delta^n)}}.
  \end{align*}
  We take both sides to the square, multiply by $\tau$, and sum over $n=0,\ldots,\Nt$.
	An application of the entropy dissipation inequality \eqref{eq:dissipation_tf} yields the desired bound \eqref{eq:BVbound_tf}.
\end{proof}

%
%-----------------------------------------------------------------------------------------------------------------------------------------------------------------------
\subsection{Convergence of time interpolants}\label{sec:compactness}
\begin{lem}
  \label{lem:timeapriori_tf}
  There is a constant $C>0$ just dependent on $\olDi$ and $(b-a)$, such that
  the following estimates hold uniformly as $\Delta\to0$:
  The functions $\ti{u_\Delta}$ and $\ti{\hatu_\Delta}$ are uniformly bounded, and
  \begin{align}
    \label{eq:help008_tf}
    &\sup_{t\in\setRp}\|\partial_x\ti{\hatu_\Delta}(t)\|_{L^2(\Omega)} \leq C ,\\ %\bar C_T, \\ %2\olDi
    \label{eq:help007_tf}
    &\sup_{t\in\setRp}\|\ti{\hatu_\Delta}(t)-\ti{\baru_\Delta}(t)\|_{L^1(\Omega)} \le C\delta ,\\ %\bar C_T \delta^{3/4}, \\
    \label{eq:help066_tf}
    &\sup_{t\in\setRp}\|\ti{\hatu_\Delta}(t)\|_{L^\infty(\Omega)}\le C.
  \end{align}
\end{lem}
\begin{proof}
  For each $n\in\setN$,
  \begin{align*}
    \big\|\partial_x\hatu_\Delta^n\big\|_{L^2(\Omega)}^2
    &= \sum_{k\in\ivalp}
    \bigg[(x^n_\kph-x^n_k)\Big(\frac{z^n_\kph-z^n_k}{x^n_\kph-x^n_k}\Big)^2 
    +(x^n_k-x^n_\kmh) \Big(\frac{z^n_k-z^n_\kmh}{x^n_k-x^n_\kmh}\Big)^2\bigg] \\
    &= \sum_{k\in\ivalp}
    \bigg[\frac{x^n_{K-1}-x^n_k}2\Big(\frac{z^n_\kph-z^n_\kmh}{x^n_{K-1}-x^n_k}\Big)^2 
    +\frac{x^n_k-x^n_{k-1}}2 \Big(\frac{z^n_\kph -z^n_\kmh}{x^n_k-x^n_{k-1}}\Big)^2\bigg] \\
    &= \delta\sum_{k\in\ivalp}\frac{z_\kph^n + z_\kmh^n}{2}\Big(\frac{z^n_\kph-z^n_\kmh}\delta\Big)^2
    \le 2\Diz(\xvec_\Delta^n).
  \end{align*}
  This gives \eqref{eq:help008_tf}.
  For proving \eqref{eq:help007_tf}, we start with the elementary observation that 
  \[ |\baru_\Delta^n(x)-\hatu_\Delta^n(x)| \le |z^n_\kph-z^n_\kmh| \quad \text{for all $x\in[x_\kmh,x_\kph]$}. \]
  Therefore,
  \begin{align*}
    \|\baru_\Delta^n-\hatu_\Delta^n\|_{L^1(\Omega)}
    &\le\delta\sum_{k\in\ivalp}\delta\left|\frac{z_\kph^n - z_\kmh^n}{\delta}\right|
    \le \delta\left(\sum_{k\in\ivalp}\frac{\delta}{z^n_k}\right)^{1/2}
							\left(\delta\sum_{k\in\ivalp}z_k^n\left(\frac{z_\kph^n - z_\kmh^n}{\delta}\right)^2\right)^{1/2}\\
    &\le\delta (2(b-a))^{1/2}\Diz(\xvec_\Delta^n)^{1/2}, %\,\left(2\sum_{k=1}^K(x^n_k-x^n_{k-1})\right)^{1/2}
    %\le \delta \olDi^{1/2}(2(x^n_K-x^n_0))^{1/2}.
  \end{align*}
  %In view of the estimate \eqref{eq:quadmom},
  %it is clear that
  %\[ (x^n_0)^2+(x^n_K)^2 \le C_T\delta^{-1/2}, \]
  %with some constant that $C_T$ that is uniform in $\Delta\to0$ and with respect to $n\in\setN$ for which $n\tau<T$.
  which shows \eqref{eq:help007_tf}.
  Finally, \eqref{eq:help066_tf} is a consequence of \eqref{eq:help008_tf}\&\eqref{eq:help007_tf}.
  First, note that
  \begin{align*}
    \left\|\ti{\hatu_\Delta}(t)\right\|_{L^1(\Omega)} 
    \le \|\ti{u_\Delta}(t)\|_{L^1(\Omega)} + \|\ti{\hatu_\Delta}(t)-\ti{u_\Delta}(t)\|_{L^1(\Omega)}
    \le M + C\delta
  \end{align*}
  is uniformly bounded.
  Now apply the interpolation inequality
  \[ \|\ti{\hatu_\Delta}(t)\|_{L^\infty(\Omega)} \le C \|\partial_x\ti{\hatu_\Delta}(t)\|_{L^2(\Omega)}^{2/3}\|\ti{\hatu_\Delta}(t)\|_{L^1(\Omega)}^{1/3} \]
  to obtain the uniform bound in \eqref{eq:help066_tf}.
\end{proof}
%
% \begin{rmk}
% The above embedding $H^1(\Omega)\hookrightarrow C(\Omega)$ only proves boundedness of $u_\Delta$
% on each compact subset of $\setR$ (which is enough for the proof of convergence). 
% The uniform boundedness on $\setR$ is a consequence of 
% \begin{align*}
% 	\tv{u_\Delta^{3/2}} \leq C\wnrm{\wgrad\mathbf{E}_\delta(\xvec_\Delta)},
% 	\quad\textnormal{with}\quad \mathbf{E}_\delta(\xvec):=\delta\sum_{\kappa\in\hval}z_\kappa^{1/2},
% \end{align*}
% which can be analogously deduced as in \cite{dde}. 
% Here, the entropy $\mathbf{E}_\delta(\xvec)$
% is the restriction of $\E$ from \eqref{eq:entropyE} on $\densN$,
% whose dissipation along its own gradient flow generates the Hele-Shaw functional, see \cite{MMS} for more details.
% In the discrete setting, one can further show that
% \begin{align*}
% 	\wnrm{\wgrad\mathbf{E}_\delta(\xvec)}
% 	=\delta\sum_{k\in\ival}\left(\frac{z_\kph^{3/2}-z_\kmh^{3/2}}{\delta}\right)^2 
% 	\leq C\Diz(\xvec)
% \end{align*}
% for any $\xvec\in\xseqN$, which gives a uniform bound on the total variation of $u_\Delta^{3/2}$ and further implies 
% uniform boundedness of $u_\Delta$ on $[0,T]\times\setR$.
% \end{rmk}
%

\begin{prp}
  \label{prp:convergence1_tf}
  There exists a function $u_*:\setRp\times\Omega\to\setRp$
  that satisfies for any $T>0$
  \begin{align}\label{eq:reg1_tf}
    u_*\in C^{1/2}([0,T];\dens)\cap L^\infty([0,T];H^1(\Omega)).
  \end{align}
  Furthermore, there exists a subsequence of $\Delta$ (still denoted by $\Delta$), 
  such that the following are true:
  \begin{align}
    \label{eq:weak_tf}
    \ti{u_\Delta}(t) &\longrightarrow u_*(t) \quad\text{in $\dens$, uniformly with respect to time}, \\
    \label{eq:uniformLinfty_tf}
    \ti{\hatu_\Delta} &\longrightarrow u_* \quad\text{uniformly on $[0,T]\times\Omega$}, \\
    \label{eq:XuniformL2}
    \ti{\theX_\Delta}(t) &\longrightarrow \theX_*(t) \quad\text{in $L^2(\M)$, uniformly with respect to $t\in[0,T]$},
  \end{align}
  where $\theX_*\in C^{1/2}([0,T];L^2(\M))$ is the Lagrangian map of $u_*$.
\end{prp}
\begin{proof}
  From the discrete energy inequality \eqref{eq:uniform_time_tf}
  and the equivalence \eqref{eq:metricequivalent} of $\wassN$ with the usual $L^2$-Wasserstein metric $\wass$,
  it follows by elementary considerations that
  \begin{align}
    \label{eq:preholder}
    \wass\big(\ti{\baru_\Delta}(t),\ti{\baru_\Delta}(s)\big)^2 \le C\big(|t-s|\big),
  \end{align}
  for all $t,s\in[0,T]$.
  Hence the generalized version of the Arzela-Ascoli theorem from \cite[Proposition 3.3.1]{AGS} is applicable 
  and yields the convergence of a subsequence of $(\ti{\baru_\Delta})$ to a limit $u_*$ in $\dens$, 
  locally uniformly with respect to $t\in[0,\infty)$.
  The H\"older-type estimate \eqref{eq:preholder} implies $u\in C^{1/2}([0,\infty);\dens)$.
  The claim \eqref{eq:XuniformL2} is a consequence of the equivalence 
  between the Wasserstein metric on $\dens$ and the $L^2$-metric on $\xspc$, see Remark \ref{rmk:wass}.
  In addition, the limit function $u_*$ is bounded on $[0,T]\times\Omega$, thanks to \eqref{eq:help066_tf}.
  
  As an intermediate step towards proving uniform convergence of $\ti{\hatu_\Delta}$,
  we show that
  \begin{align}
    \label{eq:unifL2}
    \hatu_\Delta(t)\longrightarrow u_*(t) \quad\text{in $L^2(\Omega)$, uniformly in $t\in[0,T]$}.
  \end{align}
  For $t\in[0,T]$, we expand the $L^2$-norm as follows:
  \begin{align*}
    \|\ti{\hatu_\Delta}(t)-u_*(t)\|_{L^2(\Omega)}^2
    &= \intom \Big[\big(\ti{\hatu_\Delta}-u_*\big) \ti{\baru_\Delta}\Big](t,x) \dd x \\
    &\qquad + \intom \Big[\big(\ti{\hatu_\Delta}-u_*\big) \big(\ti{\hatu_\Delta} - \ti{\baru_\Delta}\big)\Big](t,x) \dd x \\
    &\qquad - \intom \Big[\big(\ti{\hatu_\Delta}-u_*\big) u_*\Big](t,x) \dd x.
  \end{align*}
  On the one hand, observe that
  \begin{align*}
    &\sup_{t\in[0,T]}\intom\Big[\big(\ti{\hatu_\Delta}-u_*\big) \big(\ti{\hatu_\Delta} - \ti{u_\Delta}\big)\Big](t,x) \dd x \\
    &\leq\sup_{t\in[0,T]}\left(\big(\|\ti{\hatu_\Delta}(t)\|_{L^{\infty}(\Omega)}
      +\|u_*(t)\|_{L^\infty(\Omega)}\big)\|\ti{\hatu_\Delta}(t) - \ti{\baru_\Delta}(t)\|_{L^1(\Omega)}\right)
  \end{align*}
  which converges to zero as $\Delta\to0$, using both conclusions from Lemma \ref{lem:timeapriori_tf}.
  On the other hand, we can use a change of variables to write
  \begin{align*}
    &\intom \Big[\big(\ti{\hatu_\Delta}-u_*\big) \ti{u_\Delta}\Big](t,x) \dd x 
    - \intom \Big[\big(\ti{\hatu_\Delta}-u_*\big) u_*\Big](t,x) \dd x \\
    &= \intM \Big[\ti{\hatu_\Delta}-u_*\Big]\big(t,\ti{\cX_\Delta}(t,x)\big)\dd\xi
    - \intM \Big[\ti{\hatu_\Delta}-u_*\Big]\big(t,\theX_*(t,\xi)\big)\dd\xi.
  \end{align*}
	%Note that the last equality holds, due to $\operatorname{supp}(u_\Delta)=[x_0,x_K]$.
  We regroup terms under the integrals and use the triangle inequality.
  For the first term, we obtain
  \begin{align*}    
    &\sup_{t\in[0,T]} \left|\intM \left(\ti{\hatu_\Delta}\big(t,\ti{\cX_\Delta}(t,\xi)\big)
      -\ti{\hatu_\Delta}\big(t,\theX_*(t,\xi)\big)\right) \dd\xi \right|\\
    &\le \sup_{t\in[0,T]}\intM \int_{\theX_*(t,\xi)}^{\ti{\cX_\Delta}(t,\xi)} \left|\partial_x\ti{\hatu_\Delta}\right|(t,y)\dd y \dd\xi \\
    &\le \sup_{t\in[0,T]}\intM \|\ti{\hatu_\Delta}\|_{H^1(\Omega)} |\theX_*-\ti{\cX_\Delta}|(t,\xi)^{1/2} \dd\xi\\
    &\le \sup_{t\in[0,T]}\Big(\|\ti{\hatu_\Delta}(t)\|_{H^1(\Omega)} \|\theX_*(t)-\ti{\cX_\Delta}(t)\|_{L^2(\M)}^{1/4}\Big).
  \end{align*}
  A similar reasoning applies to the integral involving $u_*$ in place of $\ti{\hatu_\Delta}$.
  Together, this proves \eqref{eq:unifL2},
  and it further proves that $u_*\in L^\infty([0,T];H^1(\Omega))$, 
  since the uniform bound on $\hatu_\Delta$ from \eqref{eq:help008_tf} is inherited by the limit.

  Now the Gagliardo-Nirenberg inequality \eqref{eq:GN} provides the estimate
  \begin{align}\label{eq:unif1}
    \|\ti{\hatu_\Delta}(t)-u_*(t)\|_{C^{1/6}(\Omega)}
    \leq C \|\ti{\hatu_\Delta}(t)-u_*(t)\|_{H^1(\Omega)}^{2/3} \|\ti{\hatu_\Delta}(t)-u_*(t)\|_{L^2(\Omega)}^{1/3}.
  \end{align}
  Combining the convergence in $L^2(\Omega)$ by \eqref{eq:unifL2} with the boundedness in $H^1(\Omega)$ from \eqref{eq:help008_tf},
  it readily follows that $\hatu_\Delta(t)\to u_*(t)$ in $C^{1/6}(\Omega)$, uniformly in $t\in[0,T]$.
  This clearly implies that $\ti{\hatu_\Delta}\to u_*$ uniformly on $[0,T]\times\Omega$.
\end{proof}
\begin{prp}
  \label{prp:convergence2_tf}
  In the setting of Proposition \ref{prp:convergence1_tf},
  we have that $u_*\in \Lloc^\infty(\setRp;H^1(\Omega))$, and
  \begin{align}
    \label{eq:strong_tf}
    \ti{\hatu_\Delta}\to u_* \quad \text{strongly in $L^2([0,T];H^1(\Omega))$}
  \end{align}
  for any $T>0$ as $\Delta\to0$.
\end{prp}
\begin{proof}
Fix $[0,T]\subseteq\setRp$.
Remember that $\hatu_\Delta^n$ is differentiable
with local constant derivatives on any interval $(x_\kappm,x_\kappa]$ for $\kappa\in\ivalp\cup\hval\cup\{K\}$,
and it especially holds $\partial_x\hatu_\Delta^n(x) = 0$ for all $x\in(a,a+\delta/2)$ and all $x\in(b-\delta/2,b)$.
Therefore, integration by parts and a rearrangement of the terms yields
\begin{align*}
	\left\|\partial_x\hatu_\Delta^n\right\|_{L^2(\Omega)}^2
	&= \sum_{\kappa\in\ivalp\cup\hval\cup\{K\}} \int_{x_\kappm}^{x_\kappa} \partial_x\hatu_\Delta^n\partial_x\hatu_\Delta^n \dd x 
	= \sum_{\kappa\in\ivalp\cup\hval\cup\{K\}}
			\Big[\hatu_\Delta^n(x)\partial_x\hatu_\Delta^n(x)\Big]_{x = x_{\kappm}+0}^{x = x_{\kappa}-0}\\
	&\leq \left\|\hatu_\Delta^n\right\|_{L^\infty(\Omega)}\tv{\partial_x\hatu_\Delta^n}.
\end{align*}
Take further two arbitrary discretizations $\Delta_1,\Delta_2$ and apply the above result on the difference 
$\ti{\hatu_{\Delta_1}} - \ti{\hatu_{\Delta_2}}$.
Using that $\tv{f-g}\leq\tv{f}+\tv{g}$ 
we obtain by integration w.r.t. time that
\begin{align*}
	&\int_0^T\left\|\partial_x\ti{\hatu_{\Delta_1}} - \partial_x\ti{\hatu_{\Delta_2}}\right\|_{L^2(\Omega)}^2 \dd t \\
	&\leq T^{1/2}
	\sup_{t\in[0,T]}\left\|\ti{\hatu_{\Delta_1}} - \ti{\hatu_{\Delta_2}}\right\|_{L^\infty(\Omega)}
	\left(2\int_0^T \tv{\partial_x\ti{\hatu_{\Delta_1}}}^2 + \tv{\partial_x\ti{\hatu_{\Delta_2}}}^2 \dd t\right)^{1/2}.
\end{align*}
This shows that $\ti{\hatu_\Delta}$ is a Cauchy-sequence in $L^2([0,T];H^1(\Omega))$
--- remember \eqref{eq:BVbound_tf} and especially the convergence result in \eqref{eq:uniformLinfty_tf} ---
and its limit has to coincide with $u_*$ in the sense of distributions, 
due to the uniform convergence of $\ti{\hatu_\Delta}$ to $u_*$ on $[0,T]\times\Omega$.
\end{proof}
\section{Weak formulation of the limit equation}\label{sec:weak}
%
%In the continuous case one can find a suitable weak formulation and its derivation for \eqref{eq:thinfilm} %with boundary conditions \eqref{eq:BC}
%e.g.\ in \cite{MMS,GST}, which is given as follows:
%Find $u:[0,+\infty)\times\Omega\to\setR_{\geq0}$ sufficiently smooth, 
%such that for any $\rho\in C^{\infty}(\Omega)$ with $\rd(a)=\rd(b)=0$, 
%and $\eta\in C_c^{\infty}((0,+\infty))$
%\begin{align}\label{eq:weakf}
	%\int_0^{+\infty}\intom \partial_t\eta\rho u \dd t\dd x - \int_0^{+\infty} \eta N(u,\rho)\dd t = 0,
%\end{align}
%where $N$ is given by \eqref{eq:N_tf}.
%%
%This weak formulation is obtained by studying the variation of the entropy $\Di$ along a Wasserstein gradient flow generated by an arbitrary
%test function $\rho$, which describes a transport along the velocity field $\rd$. The corresponding entropy functional is $\Phi(u)=\int_\setR\rho(x)u(x)\dd x$.
%It is thereofore obvious to adapt this idea to gain a discrete analogue of the weak formulation for our variational numerical scheme.

In the continuous theory a suitable weak formulation for \eqref{eq:thinfilm} is attained by applying purely variational methods, see for instance \cite{MMS,GST}.
More precisely, the weak formulation in \eqref{eq:weak_intro_tf} 
is obtained by studying the variation of the entropy $\Di$ along a Wasserstein gradient flow generated by an arbitrary
spatial test function $\rho$, which describes a transport along the velocity field $\rd$. 
The corresponding entropy functional is $\Phi(u)=\int_\setR\rho(x)u(x)\dd x$.
It is therefore obvious to adapt this idea -- similar as in \cite{dde,dlssv3} -- 
to show that $\ti{\baru_\Delta}$ inherits a discrete analogue to the weak formulation \eqref{eq:weak_intro_tf}. 
Hence, we study the variations of the entropy $\Diz$ along the vector field 
generated by the potential
\begin{align*}
	\Phiz(\xvec) = \intM \rho(\cX_\theh[\xvec])\dd\xi
\end{align*}
for any arbitrary smooth test function $\rho\in C^\infty(\Omega)$ with $\rd(a)=\rd(b)=0$.
That is why we define
\begin{align}\label{eq:weak_flow}
	\vvec(\xvec) = \wgrad\Phiz(\xvec),\quad\textnormal{where}\quad\big[\grad\Phiz(\xvec)\big]_k = \intM \rd(\cX(\xi)) \hatf_k(\xi) \dd\xi,
	\quad k = 1,\ldots,K-1.
\end{align}
Later on, we will use the compactness results from section \ref{sec:compactness} to pass to the limit, 
which yields the weak formulation of our main result in Theorem \ref{thm:main_tf}. 
Therefore, the aim of this section is to show the following:
\begin{prp}\label{prp:weakgoal_tf}
For every $\rho\in C^\infty(\Omega)$ with $\rd(a)=\rd(b)=0$,
and for every $\eta\in C^\infty_c(\setRpp)$,
the limit curve $u_*$ satisfies
\begin{align}\label{eq:weakgoal_tf}
	\int_0^\infty\intom 
	\partial_t\varphi u_* \dd t\dd x + \int_0^\infty N(u_*,\varphi)\dd t = 0,
\end{align}
where the highly nonlinear term $N$ from \eqref{eq:N_tf} is given by
\begin{align}\label{eq:N_weak_tf}
	N(u,\rho) = \frac12\intom (u^2)_x \rddd + 3 u_x^2 \rdd \dd x + \intom V_x u \rd \dd x.
\end{align}
\end{prp}
The proof of this statement will be treated in two essential steps
\begin{enumerate}
\item Show the validity of a discrete weak formulation for $\ti{u_\Delta}$, using a discrete flow interchange estimate
\item Passing to the limit using Proposition \ref{prp:convergence2_tf}.
\end{enumerate}
For definiteness, fix a spatial test function $\rho\in C^\infty(\Omega)$ with $\rd(a)=\rd(b)=0$, 
and a temporal test function $\eta\in C^\infty_c(\setRpp)$ with $\operatorname{supp}\eta\subseteq(0,T)$ for a suitable $T>0$.
Denote again by $\Nt\in\setN$ an integer with $\tau\Nt\in(T,T+1)$.
Let $\rkap>0$ be chosen such that
\begin{align}
  \label{eq:smoothbound_tf}
  \|\rho\|_{C^4(\Omega)}\le \rkap
	\quad\text{and}\quad
	\|\eta\|_{C^1(\setRp)}\le \rkap.
\end{align}
For convenience, we assume $\delta<1$ and $\tau<1$.
In the estimates that follow, the non-explicity constants possibly depend on $\Omega$, $T$, $\rkap$, and $\olDi$,
but not on $\Delta$.
\begin{lem}[discrete weak formulation]\label{lem:weakfd}
For any functions $\rho\in C^\infty(\Omega)$  with $\rd(a)=\rd(b)=0$, and $\eta\in C_c^{\infty}(\setRpp)$, 
the solution $\xvec_\Delta^n$ with $u_\Delta^n=\cf_\theh[\xvec_\Delta^n]$ 
of the minimization problem \eqref{eq:dmm} fulfills
\begin{align}\label{eq:weakfd}
	\tau\sum_{n=0}^{\infty}\eta((n-1)\tau)\left(\frac{\Phiz(\xvec_\Delta^n)-\Phiz(\xvec_\Delta^{n-1})}{\tau} 
				- \wspr{\wgrad\DiVz(\xvec_\Delta^n)}{\rvec}\right)
	= \mathcal{O}(\tau) + \mathcal{O}(\delta^{1/4}), 
\end{align}
where we use the short-hand notation $\rvec(\xvec) := \left(\rd(x_1),\ldots,\rd(x_{K-1})\right)$ for any $\xvec\in\xseqN$.
\end{lem}
\begin{proof}
As a first step, we prove that both vectors $\rvec(\xvec)$ and $\vvec(\xvec)$ nearby coinside for any $\xvec\in\xseqN$, i.e. it holds
\begin{align}\label{eq:rvec_estimate}
	\wnrm{\vvec(\xvec) - \rvec(\xvec)} \leq 2\delta^{1/2}C,
\end{align}
for a constant $C>0$ that only depends on $\rkap$ and $\Omega$.
Hence, denote by $\cX=\cX_\theh[\xvec]$ the corresponding Lagrangian map of $\xvec$ and choose any $k=0,\ldots,K$,
then one easily gets 
\begin{align*}
	|\vvec(x_k)-\rvec(x_k)| = \big|\big[\wgrad\Phiz(\xvec)\big]_k - \rvec(x_k)\big| 
	\leq \delta^{-1}\intM \big|\rd(\cX(\xi)) \hatf_k(\xi) - \rd(x_k)\hatf_k(\xi)\big| \dd\xi.
\end{align*}
First assume $\xi\in[\xi_{k-1},\xi_k]$, then a Taylor expansion for $\rd(\cX(\xi))$ yields
\begin{align*}
	\rd(\cX(\xi)) = \rd(x_k) - \rd'(\cX(\widetilde{\xi}_k))\cX_\xi(\widetilde{\xi}_k)(\xi-\xi_k) 
	= \rd(x_k) - \rd'(\cX(\widetilde{\xi}_k))(x_k-x_{k-1})\hatf_{k-1}(\xi)
\end{align*}
for a certain $\widetilde{\xi}_k\in[\xi_{k-1},\xi_k]$.
Consequently the valitdity of $\int_{\xi_{k-1}}^{\xi_k}\hatf_k\hatf_{k-1}\dd\xi = \frac{3\delta}{2}$ yields
\begin{align*}
	\delta^{-1}\int_{\xi_{k-1}}^{\xi_k} \big|\rd(\cX(\xi)) \hatf_k(\xi) - \rd(x_k)\hatf_k(\xi)\big| \dd\xi
	%&\leq \int_{\xi_{k-1}}^{\xi_k} \big|\rd'((\cX(\widetilde{\xi}_k))(x_k-x_{k-1})\hatf_{k-1}(\xi)\hatf_k(\xi)\big| \dd\xi \\
	&\leq \frac{3}{2}\rkap(x_k-x_{k-1}).
\end{align*}
Similarly one proves the analogue statement for $\xi\in[\xi_{k},\xi_{K-1}]$, hence
\begin{align*}
	|\vvec(x_k)-\rvec(x_k)|
	\leq \frac{3}{2}\big((x_k-x_{k-1}) + (x_{k+1}-x_k)\big) = \frac{3}{2}\rkap(x_{k+1}-x_{k-1}).
\end{align*}
Squaring the above term and summing-up over all $k=1,\ldots,K-1$ finally proves \eqref{eq:rvec_estimate}, 
due to $(x_{k+1}-x_{k-1})\leq 2(b-a)$ and
\begin{align*}
	\wnrm{\vvec(\xvec) - \rvec(\xvec)}^2
	&\leq \frac{9\delta}{4}\rkap^2 \sum_{k=1}^{K-1} (x_{k+1}-x_{k-1})^2
	\leq \frac{9\delta}{2}\rkap^2 (b-a)\sum_{k=1}^{K-1} (x_{k+1}-x_{k-1}) \leq C \delta.
\end{align*}
Let us now invoke the proof of \eqref{eq:weakfd}. 
a Taylor expansion of $\rho$ for $\theX,\theX'\in\xspc$ yields
\begin{align*}
	\rho(\theX) - \rho(\theX') - \frac{\rkap}{2}(\theX-\theX')^2 \leq \rd(\theX)(\theX-\theX'),
\end{align*}
which implies for $\theX'=\cX_\theh[\xvec_\Delta^{n-1}]$ and $\theX=\cX_\theh[\xvec_\Delta^n]$
\begin{equation}\begin{split}\label{eq:wstep1}
	\Phiz(\xvec_\Delta^n)-\Phiz(\xvec_\Delta^{n-1}) - \frac{\rkap}{2}\wnrm{\xvec_\Delta^n-\xvec_\Delta^{n-1}}^2
	&\leq \sum_{k=1}^{K-1}[\xvec_\Delta^n-\xvec_\Delta^{n-1}]_k\intM\rd(\cX_\theh[\xvec_\Delta^n])\hatf_k(\xi)\dd\xi \\
	&= \wspr{\xvec_\Delta^n-\xvec_\Delta^{n-1}}{\wgrad\Phiz(\xvec_\Delta^n)}
	=\tau\wspr{\wgrad\DiVz(\xvec_\Delta^n)}{\wgrad\Phiz(\xvec_\Delta^n)}.
\end{split}\end{equation}
Thanks to \eqref{eq:rvec_estimate}, the last term can be estimated as follows:
\begin{equation}\begin{split}\label{eq:wstep2}
	\tau\wspr{\wgrad\DiVz(\xvec_\Delta^n)}{\wgrad\Phiz(\xvec_\Delta^n)}
	&= \tau\wspr{\wgrad\DiVz(\xvec_\Delta^n)}{\rvec(\xvec_\Delta^n)} + \tau\wspr{\wgrad\DiVz(\xvec_\Delta^n)}{\vvec(\xvec_\Delta^n) - \rvec(\xvec_\Delta^n)} \\
	&\leq \tau\wspr{\wgrad\DiVz(\xvec_\Delta^n)}{\rvec(\xvec_\Delta^n)} + \tau\wnrm{\wgrad\DiVz(\xvec_\Delta^n)}\wnrm{\vvec(\xvec_\Delta^n) - \rvec(\xvec_\Delta^n)} \\
	&\leq \tau\wspr{\wgrad\DiVz(\xvec_\Delta^n)}{\rvec(\xvec_\Delta^n)} + C\tau \delta^{1/2}\wnrm{\wgrad\DiVz(\xvec_\Delta^n)}.
\end{split}\end{equation}
So combine \eqref{eq:wstep1} and \eqref{eq:wstep1}, 
add $\eta((n-1)\tau)$ and summing-up over $n=1,\ldots,N$, one attains
\begin{align*}
	&\left|\tau\sum_{n=1}^{\Nt}\eta((n-1)\tau)
	\left(\frac{\Phiz(\xvec_\Delta^n)-\Phiz(\xvec_\Delta^{n-1})}{\tau} - \wspr{\wgrad\DiVz(\xvec_\Delta^n)}{\rvec(\xvec_\Delta^n)}\right)\right| \\
	\leq &\|\eta\|_{C^0([0,T])}\tau\sum_{n=1}^{\Nt}\frac{\rkap}{2}\wnrm{\xvec_\Delta^n-\xvec_\Delta^{n-1}}^2 
	+ C\|\eta\|_{C^0([0,T])}\tau\sum_{n=1}^{\Nt}\delta^{1/2}\wnrm{\wgrad\DiVz(\xvec_\Delta^n)},
\end{align*}
where the right hand side is of order $\mathcal{O}(\tau)+\mathcal{O}(\delta^{1/2})$, due to \eqref{eq:eee_tf}. 
An analog calculation replacing $\rho$ with $-\rho$ leads finally to \eqref{eq:weakfd}.
\end{proof} 
The identification of the weak formulation in \eqref{eq:weakgoal_tf} 
with the limit of \eqref{eq:weakfd}
is splitted in two main steps:
In the first one, we estimate the term that more or less describes the error
that is caused by approximating the time derivative in \eqref{eq:weakgoal_tf} 
with the respective difference quotient in \eqref{eq:weakfd},
\begin{align}
  \label{eq:weakform1_tf}
  \begin{split}
    \er_{1,\Delta}:=\Bigg|\int_0^T \left(\eta'(t)\intom \rho(x)\ti{\baru_\Delta}(t,x) \dd x
      +\eta(t)\ti{\wspr{\wgrad\DiVz(\xvec_\Delta^n)}{\rvec(\xvec_\Delta^n)}}(t)\right)\dd t\Bigg| \\
    \le C\big(\tau + \delta^{1/2}\big).
  \end{split}
\end{align}
The second much more challenging step is to prove the error estimate
\begin{align}
  \label{eq:weakform2_tf}
  \begin{split}
    \er_{2,\Delta}:=\Bigg|\int_0^T \eta(t)
		\Bigg(\frac12\intom \rddd(x)\partial_x(\ti{u_\Delta}^2)(t,x) + 3\rdd(x) \partial_x\ti{u_\Delta}^2(t,x) \dd x \\
				+ \intom V_x(x) \ti{u_\Delta} \rd(x) \dd x
    -\ti{\wspr{\wgrad\DiVz(\xvec_\Delta^n)}{\rvec(\xvec_\Delta^n)}}(t)\Bigg)\dd t\Bigg|
    \le C\delta^{1/4},
  \end{split}
\end{align}
which, heuristically spoken, gives a rate of convergence of 
$\ti{\wspr{\wgrad\DiVz(\xvec_\Delta^n)}{\rvec(\xvec_\Delta^n)}}$ towards the 
nonlinear term $N(u_*,\rho)$ from \eqref{eq:N_weak_tf}.

The first estimate in \eqref{eq:weakform1_tf} is a consequence of Lemma \ref{lem:weakfd}:

\begin{proof2}{\eqref{eq:weakform1_tf}}
  Using that $\eta(\tn)=0$ for any $n\geq\Nt$, we obtain after ``summation by parts'':
  \begin{equation}
    \label{eq:dummy814}
    \begin{split}
      -\int_0^T&\ed(t)\left(\int_\Omega \rho(x)\ti{u_\Delta}(t,x) \dd x\right)\dd t 
      = -\sum_{n=1}^{\Nt} \left(\int_{\tnm}^{\tn}\ed(t)\dd t \int_\Omega\rho(x)\bar u_\Delta^n(x)\dd x\right) \\
      &= -\tau\sum_{n=1}^{\Nt}\left(\frac{\eta(\tn)-\eta(\tnm)}{\tau}\,\intM\rho\circ\theX_\Delta^n(\xi)\dd\xi\right) \\
      &= \tau\sum_{n=1}^{\Nt}\left(\eta(\tnm)\,\intM\frac{\rho\circ\cX_\Delta^n(\xi)-\rho\circ\cX_\Delta^{n-1}(\xi)}{\tau}\dd\xi\right).
    \end{split}
  \end{equation}
Finally observe that
  \begin{align*}
    \Res:=&\left|\int_0^T\eta(t)\ti{\wspr{\rvec(\xvec_\Delta)}{\wgrad\DiVz(\xvec_\Delta)}}(t)\dd t 
      - \tau\sum_{n=1}^{\Nt}\eta(\tnm)\wspr{\wgrad\DiVz(\xvec_\Delta^n)}{\rvec(\xvec_\Delta^n)}\right|\\
    &\le \left(\tau\sum_{n=1}^{\Nt}\left|\frac{1}{\tau}\int_{\tnm}^{\tn}\eta(t)\dd t-\eta(\tnm)\right|^2\right)^{1/2}
    \left(\tau\sum_{n=1}^\infty\rkap^2\wnrm{\wgrad\DiVz(\xvec_\Delta^n)}^2\right)^{1/2}\\
    &\le \big((T+1)\rkap^2\tau^2\big)^{1/2}(2\rkap^2\DiVz(\xvec_\Delta^0))^{1/2} = C'\DiVz(\xvec_\Delta^0)^{1/2}\tau,
  \end{align*}
using the energy estimate \eqref{eq:eee_tf}. We conclude that
  \begin{align*}
    \er_{1,\Delta}
    &\stackrel{\eqref{eq:dummy814}}{\le} \Res + \tau\sum_{n=1}^{\Nt}\left(\big|\eta(\tnm)\big|
      \,\left|\intM\frac{\rho\circ\cX_\Delta^n(\xi)-\rho\circ\cX_\Delta^{n-1}(\xi)}{\tau}\dd\xi
        -\wspr{\wgrad\DiVz(\xvec_\Delta^n)}{\rvec(\xvec_\Delta^n)}\right|\right) \\
    &= \mathcal{O}(\tau) + \mathcal{O}(\delta^{1/4})
  \end{align*}
  where we have used \eqref{eq:weakfd}, keeping in mind that $\Phiz(\xvec_\Delta^n) = \int_0^M\rho(\cX_\Delta^n)\dd\xi$
\end{proof2}

The proof of \eqref{eq:weakform2_tf} is treated essentially in 2 steps.
In the first one 
we rewrite the term $\wspr{\wgrad\DiVz(\xvec_\Delta^n)}{\rvec(\xvec_\Delta^n)}$ (see Lemma \ref{lem:rhs_tf}),
and use Taylor expansions to identify it with the corresponding integral terms of
\eqref{eq:N_weak_tf} up to some additional error terms, see Lemmata \ref{lem:R1_tf}-\ref{lem:R7_tf}.
Then we use the strong compactness result of Proposition~\ref{prp:convergence2_tf} to pass to the limit as $\Delta\to0$ in the second step.
\begin{lem}\label{lem:rhs_tf}
  With the short-hand notation $\rvec(\xvec) = \left(\rd(x_1),\ldots,\rd(x_{K-1})\right)$ for any $\xvec\in\xseqN$, one has that
  \begin{align}
    \label{eq:sevenAs_tf}
    -\wspr{\wgrad\DiVz(\xvec_\Delta^n)}{\rvec(\xvec_\Delta^n)} = A^n_1 + A^n_2 + A^n_3 - A^n_4 + A_5^n + A_6^n + A_7^n,
  \end{align}
  where
  \begin{align*}
    &\Ai = \delta
		\sum_{k\in\ivalp}\left(\frac{z^n_\kph-z^n_\kmh}{\delta}\right)^2
											\left((z_\kph^n)^2 + (z_\kmh^n)^2 + z_\kph^n z_\kmh^n\right)\left(\frac{\rd(x^n_{k+1})-\rd(x^n_{k-1})}{2\delta}\right), \\
		&\Aii = \frac{\delta}{4}
		\sum_{k\in\ivalp}\left(\frac{z^n_\kph-z^n_\kmh}{\delta}\right)^2
											(z_\kph^n)^2\left(\frac{\rd(x^n_{k+1})-\rd(x^n_k)}{\delta}\right), \\
		&\Aiii = \frac{\delta}{4}
		\sum_{k\in\ivalp}\left(\frac{z^n_\kph-z^n_\kmh}{\delta}\right)^2
											(z_\kmh^n)^2\left(\frac{\rd(x^n_k)-\rd(x^n_{k-1})}{\delta}\right), \\
		&\Aiv = \delta\sum_{k\in\ivalp}\left(\frac{z^n_\kph-z^n_\kmh}{\delta}\right)^2
											\left(\frac{(z^n_\kph)^3+(z^n_\kmh)^3}{2 z^n_\kph z^n_\kmh}\right)\rdd(x^n_k), \\
    &\Av = \delta\sum_{k\in\ivalp}\left(\frac{z^n_\kph-z^n_\kmh}{\delta}\right)
											\left(\frac{(z^n_\kph)^3+(z^n_\kmh)^3}{2}\right)
											\left(\frac{\rd(x^n_{k+1})-\rd(x^n_k)-(x^n_{k+1}-x^n_k)\rdd(x^n_k)}{\delta^2}\right), \\
    &\Avi = \delta\sum_{k\in\ivalp}\left(\frac{z^n_\kph-z^n_\kmh}{\delta}\right)
											\left(\frac{(z^n_\kph)^3+(z^n_\kmh)^3}{2}\right)
											\left(\frac{\rd(x^n_{k-1})-\rd(x^n_{k})-(x^n_{k-1}-x^n_{k})\rdd(x^n_k)}{\delta^2}\right), \\
		&\Avii = \delta\sum_{k\in\ivalp} V(x_k^n)\rd(x_k^n).
  \end{align*}
\end{lem}
\begin{proof}
Fix some time index $n\in\setN$ (omitted in the calculations below).
Recall the representation of $\wgrad\DiVz$ as
\begin{align*}
	\wgrad\DiVz(\xvec)
  = \frac1{\delta^2}\left(\grad^2\Hhz(\xvec)\grad\entqz(\xvec)+\grad^2\entqz(\xvec)\grad\Hhz(\xvec)\right)
\end{align*}
with corresponding gradients and hessians in \eqref{eq:gradHFzQz} and \eqref{eq:HessHFzQz}. 
Multiplication with $\rvec(\xvec_\Delta)$ then yields
\begin{align*}
	-\wspr{\wgrad\DiVz(\xvec_\Delta^n)}{\rvec(\xvec_\Delta)}
	&= \frac{\delta}{2}\sum_{\kappa\in\hval}z_\kappa^3\left(\frac{z_\kappp - 2z_\kappa + z_\kappm}{\delta^2}\right)
																			\left(\frac{\rd(x_{\kappa+\frac{1}{2}})-\rd(x_{\kappa-\frac{1}{2}})}{\delta}\right) \\%\label{eq:l1}\\
	&\quad
	+ \frac{\delta}{4}\sum_{\kappa\in\hval}z_\kappa^2\left(\frac{z_\kappp^2 - 2z_\kappa^2 + z_\kappm^2}{\delta^2}\right)
													\left(\frac{\rd(x_{\kappa+\frac{1}{2}})-\rd(x_{\kappa-\frac{1}{2}})}{\delta}\right) \\%\label{eq:l2}\\
	&\quad+ \delta\sum_{k\in\ivalp} V(x_k)\rd(x_k). %\notag
\end{align*}
Observing that
\begin{align*}
	\frac{z_{\kappa+1}^2-2z_\kappa^2 + z_{\kappa-1}^2}{\delta^2}
	= 2z_\kappa\frac{z_{\kappa+1}-2z_\kappa+z_{\kappa-1}}{\delta^2} 
	+ \left(\frac{z_{\kappa+1}-z_\kappa}\delta\right)^2 + \left(\frac{z_{\kappa-1}-z_\kappa}\delta\right)^2,
\end{align*}
we further obtain that
\begin{align*}
	-\wspr{\wgrad\DiVz(\xvec_\Delta^n)}{\rvec(\xvec_\Delta)}
	&= \delta\sum_{\kappa\in\hval}z_\kappa^3\left(\frac{z_\kappp - 2z_\kappa + z_\kappm}{\delta^2}\right)
																			\left(\frac{\rd(x_{\kappa+\frac{1}{2}})-\rd(x_{\kappa-\frac{1}{2}})}{\delta}\right) \\%\label{eq:l1}\\
	&\quad
	+ \Aiio + \Aiiio + \Aviio.
\end{align*}
It hence remains to show that $(A) = \Aio - \Aivo + \Avo + \Avio$, where
\begin{align*}%\label{eq:l1}
	(A) := \delta\sum_{\kappa\in\hval}z_\kappa^3\left(\frac{z_\kappp - 2z_\kappa + z_\kappm}{\delta^2}\right)
																			\left(\frac{\rd(x_{\kappa+\frac{1}{2}})-\rd(x_{\kappa-\frac{1}{2}})}{\delta}\right).
\end{align*}
After ``summation by parts''
and an application of the elementary equality (for arbitrary numbers $p_\pm$ and $q_\pm$)
\begin{align*}
	p_+q_+ - p_-q_- = \frac{p_+ + p_-}2(q_+ - q_-) + (p_+ - p_-)\frac{q_+ + q_-}2,
\end{align*}
one attains
\begin{align}
	(A) 
	&= \frac{\delta}{2}\sum_{k\in\ivalp}\left(\frac{z_\kph-z_\kmh}{\delta}\right)\left(\frac{z_\kmh^3-z_\kph^3}{\delta}\right)
																					\left(\frac{\rd(x_{k+1})-\rd(x_{k-1})}{\delta}\right) \notag \\
	&\quad + \frac{\delta}{2}\sum_{k\in\ivalp}\left(\frac{z_\kph-z_\kmh}{\delta}\right)\left(\frac{z_\kmh^3+z_\kph^3}{\delta}\right)
																					\left(\frac{\rd(x_{k+1})-2\rd(x_k) + \rd(x_{k-1})}{\delta}\right) \notag \\
	&= \Aio + \frac{\delta}{2}\sum_{k\in\ivalp}\left(\frac{z_\kph-z_\kmh}{\delta}\right)\left(\frac{z_\kmh^3+z_\kph^3}{\delta}\right)
																					\left(\frac{\rd(x_{k+1})-2\rd(x_k) + \rd(x_{k-1})}{\delta}\right) , \label{eq:help1_tf}
	%&= \frac{\delta}{2}\sum_{k\in\ivalp}\left(\frac{z_\kph-z_\kmh}{\delta}\right)^2\left(z_\kmh^2+z_\kph^2 + z_\kph z_\kmh\right)
																					%\left(\frac{\rd(x_{k-1})-\rd(x_{k+1})}{\delta}\right) \\
	%&\quad + \frac{\delta}{2}\sum_{k\in\ivalp}\left(\frac{z_\kph-z_\kmh}{\delta}\right)\left(\frac{z_\kmh^3+z_\kph^3}{\delta}\right)
																					%\left(\frac{\rd(x_{k-1})-2\rd(x_k) + \rd(x_{k+1})}{\delta}\right) , 
\end{align}
where we additionally used the identity $(p^3-q^3) = (p-q)(p^2+q^2+pq)$ in the last step.
In order to see that the last sum in \eqref{eq:help1_tf} equals to $- \Aivo + \Avo + \Avio$,
simply observe that the identity
\begin{align*}
	\frac{x_{k+1}-x_k}\delta + \frac{x_{k-1}-x_k}\delta = \frac1{z_\kph}-\frac1{z_\kmh}=-\frac{z_\kph-z_\kmh}{z_\kph z_\kmh}
\end{align*}
makes the coefficient of $\rdd(x_k)$ vanish.
\end{proof}
For the analysis of the terms in \eqref{eq:sevenAs_tf},
we need some sophisticated estimates presented in the following two lemmata.
The first one gives a control on the oscillation of the $z$-values at neighboring grid points:
\begin{lem}
For any $p,q\in\{1,2\}$ with $p+q\leq 3$ one has that
\begin{align}\label{eq:weakoscillation_tf}
	\sum_{n=1}^{\Nt}\delta\sum_{k\in\ivalp} z_k^n\left|\frac{z_\kph^n-z_\kmh^n}{\delta}\right|^p
																											\left|\frac{(z_{k\pm\frac{1}{2}}^n)^q}{(z_{k\mp\frac{1}{2}}^n)^q}-1\right|
	\leq C\delta^{1/4}.
\end{align}

\end{lem}
\begin{proof} 
Instead of \eqref{eq:weakoscillation_tf}, we are going to prove that
\begin{align}\label{eq:weakoscillation_pre_tf}
	\tau\sum_{n=1}^{\Nt}\delta\sum_{k\in\ivalp} z_k^n\left|\frac{z_\kph^n-z_\kmh^n}{\delta}\right|^p
																											\left|\frac{z_{k\pm\frac{1}{2}}^n}{z_{k\mp\frac{1}{2}}^n}-1\right|^q
	\leq C\delta^{1/4}
\end{align}
is satisfied for any $p,q\in\{1,2\}$ with $p+q\leq 3$, 
which implies \eqref{eq:weakoscillation_tf} because of the following considerations:
%which imlpies \eqref{eq:weakoscillation_tf} immediately for $q=1$. 
The situation is clear for $q=1$, thus assume $q=2$ in \eqref{eq:weakoscillation_tf}.
Then \eqref{eq:weakoscillation_pre_tf} is an upper bound on \eqref{eq:weakoscillation_tf}, due to
\begin{align*}
	\frac{(z_{k\pm\frac{1}{2}}^n)^2}{(z_{k\mp\frac{1}{2}}^n)^2}-1 
	= \left(\frac{z_{k\pm\frac{1}{2}}^n}{z_{k\mp\frac{1}{2}}^n}-1\right)^2 + 2\left(\frac{z_{k\pm\frac{1}{2}}^n}{z_{k\mp\frac{1}{2}}^n}-1\right)
\end{align*}
for any $n=1,\ldots,\Nt$.

To prove \eqref{eq:weakoscillation_pre_tf}, we first apply H\"older's inequality,
\begin{equation}\begin{split}\label{eq:Rstep1}
	&\tau\sum_{n=1}^{\Nt}\delta\sum_{k\in\ivalp} z_k^n\left|\frac{z_\kph^n-z_\kmh^n}{\delta}\right|^p\left|\frac{z_{k\pm\frac{1}{2}}^n}{z_{k\mp\frac{1}{2}}^n}-1\right|^q
	= \tau\sum_{n=1}^{\Nt}\delta\sum_{k\in\ivalp}z_k^n\left|\frac{z_\kph^n-z_\kmh^n}{\delta}\right|^{p+q}\left(\frac{\delta}{z_{k\mp\frac{1}{2}}^n}\right)^q \\
	\leq&\left(\tau\sum_{n=1}^{\Nt}\sum_{k\in\ivalp}\delta z_k^n\left(\frac{z_\kph^n-z_\kmh^n}{\delta}\right)^4\right)^{\frac{p+q}{4}}
			\left(\delta\tau\sum_{n=1}^{\Nt}\sum_{k\in\ivalp}z_k^n\left(\frac{\delta}{z_{k\pm\frac{1}{2}}^n}\right)^{\frac{q}{\alpha}}\right)^{\alpha},
\end{split}\end{equation}
with $\alpha = 1-\frac{p+q}{4}$. 
The first factor is uniformly bounded due to \eqref{eq:L4D2_tf} and \eqref{eq:dissipation_tf}. 
For the second term, we use \eqref{eq:help066_tf} and \eqref{eq:xpower} to achieve
\begin{align*}
	\delta\tau\sum_{n=1}^{\Nt}\sum_{k\in\ivalp}z_k^n\left(\frac{\delta}{z_{k\pm\frac{1}{2}}^n}\right)^{\frac{q}{\alpha}}
	\leq (T+1)\delta(b-a)^{\frac{q}{\alpha}}\|\ti{\hatu}\|_{L^\infty([0,T]\times\Omega)}
	\leq C(T+1)\delta(b-a)^{\frac{q}{\alpha}},
\end{align*}
which shows \eqref{eq:weakoscillation_pre_tf}, due to $\alpha\geq\frac{1}{4}$.
\end{proof}
\begin{lem}
For any $p\in\{1,2\}$ one obtains that
\begin{align}\label{eq:weakoscillation2_tf}
	\tau\sum_{n=1}^{\Nt}\delta\sum_{k\in\ivalp} z_k^n\left|\frac{z_\kph^n-z_\kmh^n}{\delta}\right|^2 (x_{k+1}^n-x_{k-1}^n)^p
	\leq C\delta^{1/2}.
\end{align}
\end{lem}
\begin{proof}
Appling  H\"older's inequality,
\begin{align*}
	&\tau\sum_{n=1}^{\Nt}\delta\sum_{k\in\ivalp} z_k^n\left|\frac{z_\kph^n-z_\kmh^n}{\delta}\right|^2 (x_{k+1}^n-x_{k-1}^n)^p \\
	\leq& 
	\left(\tau\sum_{n=1}^{\Nt}\delta\sum_{k\in\ivalp} z_k^n\left|\frac{z_\kph^n-z_\kmh^n}{\delta}\right|^4 \right)^{1/2}
	\left(\tau\sum_{n=1}^{\Nt}\delta\sum_{k\in\ivalp} z_k^n (x_{k+1}^n-x_{k-1}^n)^{2p} \right)^{1/2}.
\end{align*}
The first sum is uniformly bounded thanks to \eqref{eq:L4D2_tf} and \eqref{eq:dissipation_tf},
and the second one satisfies
\begin{align*}
	\left(\tau\sum_{n=1}^{\Nt}\delta\sum_{k\in\ivalp} z_k^n (x_{k+1}^n-x_{k-1}^n)^{2p} \right)^{1/2}
	\leq \delta^{1/2}(T+1)^{1/2}\|\ti{\hatu}\|_{L^\infty([0,T]\times\Omega)}^{1/2} (b-a)^p,
\end{align*}
where we used \eqref{eq:help066_tf} and \eqref{eq:xpower}.
\end{proof}

%\clearpage
\begin{lem}\label{lem:R1_tf}
  There is a constant $C_1>0$ expressible in $\Omega$, $T$, $\rkap$ and $\olDi$
	such that
  \begin{align*}
    \Ri:=\tau\sum_{n=1}^{\Nt}\bigg|\Ai - 3\intM \hatz_\Delta^n(\xi)\partial_\xi\hatz^n_\Delta(\xi)^2\rdd\circ\theX^n_\Delta(\xi)\dd\xi\bigg|
    \le C_1\delta^{1/4}.
  \end{align*}
\end{lem}
\begin{proof}
Let us introduce the term
\begin{align*}
	B_1^n := \delta\sum_{k\in\ivalp}z_k^n\left(\frac{z^n_\kph-z^n_\kmh}{\delta}\right)^2\frac{3}{\delta}\int_{\xi_\kmh}^{\xi_\kph}\rdd\circ\theX^n_\Delta(\xi)\dd\xi.
\end{align*}
First observe that by definition of $\hatz_\Delta^n$,
\begin{align*}
	\intM \hatz^n_\Delta(\xi)\partial_\xi\hatz^n_\Delta(\xi)^2\rdd\circ\theX^n_\Delta(\xi)\dd\xi
	= \sum_{k\in\ivalp} \left(\frac{z^n_\kph-z^n_\kmh}{\delta}\right)^2 \int_{\xi_\kmh}^{\xi_\kph}\hatz^n_\Delta(\xi) \rdd\circ\theX^n_\Delta(\xi)\dd\xi,
\end{align*}
hence we get for $B_1^n$
\begin{align*}\begin{split}%\label{eq:R1_step1_tf}
	%\left|\intM \big(\hatz_\Delta^n(\xi)-z_k\big)\partial_\xi\hatz^n_\Delta(\xi)^2\rdd\circ\theX^n_\Delta(\xi)\dd\xi\right| 
	&\left|3\intM \hatz_\Delta^n(\xi)\partial_\xi\hatz^n_\Delta(\xi)^2\rdd\circ\theX^n_\Delta(\xi)\dd\xi
	      - B_1^n \right| \\
	\leq& 3\rkap\sum_{k\in\ivalp} \left(\frac{z^n_\kph-z^n_\kmh}{\delta}\right)^2 \int_{\xi_\kmh}^{\xi_\kph} \big|\hatz_\Delta^n(\xi)-z_k\big| \dd\xi 
	\leq 3\rkap\delta\sum_{k\in\ivalp} z_k^n\left|\frac{z^n_\kph-z^n_\kmh}{\delta}\right|^2 \left|\frac{z_\kph^n}{z_\kmh^n}-1\right|.
\end{split}\end{align*}
This especially implies, due to \eqref{eq:weakoscillation_tf} that
  \begin{align}\label{eq:R1_step1_tf}
    \tau\sum_{n=1}^{\Nt}\bigg|B_1^n - 3\intM \hatz_\Delta^n(\xi)\partial_\xi\hatz^n_\Delta(\xi)^2\rdd\circ\theX^n_\Delta(\xi)\dd\xi\bigg|
    \le C\delta^{1/4}.
  \end{align}
%\begin{align}\begin{split}\label{eq:R1_step1_tf}
	%%\left|\intM \big(\hatz_\Delta^n(\xi)-z_k\big)\partial_\xi\hatz^n_\Delta(\xi)^2\rdd\circ\theX^n_\Delta(\xi)\dd\xi\right| 
	%&\left|3\intM \hatz_\Delta(\xi)\partial_\xi\hatz^n_\Delta(\xi)^2\rdd\circ\theX^n_\Delta(\xi)\dd\xi
	      %- \delta\sum_{k\in\ivalp}z_k\left(\frac{z^n_\kph-z^n_\kmh}{\delta}\right)^2\frac{3}{\delta}\int_{\xi_\kmh}^{\xi_\kph}\rdd\circ\theX^n_\Delta(\xi)\dd\xi \right| \\
	%\leq& 3\rkap\sum_{k\in\ivalp} \left(\frac{z^n_\kph-z^n_\kmh}{\delta}\right)^2 \int_{\xi_\kmh}^{\xi_\kph} \big|\hatz_\Delta^n(\xi)-z_k\big| \dd\xi 
	%\leq 3\rkap\delta\sum_{k\in\ivalp} z_k^n\left|\frac{z^n_\kph-z^n_\kmh}{\delta}\right|^2 \left|\frac{z_\kph^n}{z_\kmh^n}-1\right|.
%\end{split}\end{align}
For simplification of $\Ri$, let us fix $n$ (omitted in the following), 
and introduce $\tilde x_k^+\in[x_k,x_{k+1}]$ and $\tilde x_k^-\in[x_{k-1},x_k]$ such that
\begin{align*}
	\frac{\rd(x_{k+1})-\rd(x_{k-1})}{2\delta} 
	&= \frac{\rd(x_{k+1})-\rd(x_k)}{2\delta} + \frac{\rd(x_k)-\rd(x_{k-1})}{2\delta} \\
	&= \rdd(\tilde x_k^+)\frac{x_{k+1}-x_k}{2\delta} + \rdd(\tilde x_k^+)\frac{x_{k+1}-x_k}{2\delta}
	= \frac{1}{2}\left(\frac{\rdd(\tilde x_k^+)}{z_\kph} + \frac{\rdd(\tilde x_k^-)}{z_\kmh}\right).
\end{align*}
Recalling that
\begin{align}
	\label{eq:dummy808_tf}
	\int_{\xi_{k-1}}^{\xi_{k+1}}\hatf_k(\xi)\dd\xi = \delta,
\end{align}
one has for each $k\in\ivalp$,
\begin{align*}\begin{split}%\label{eq:R1_step2_tf}
	(A):=&\frac{1}{2}\left(z_\kph^2 + z_\kmh^2 + z_\kph z_\kmh\right) \Big(\frac{\rdd(\tilde x_k^+)}{z_\kph} + \frac{\rdd(\tilde x_k^-)}{z_\kmh}\Big) 
	- \frac3\delta\int_{\xi_\kmh}^{\xi_\kph}z_k\rdd\circ\theX_\Delta\dd\xi \\
	=&\frac{1}{4}z_\kmh\left(2+\frac{z_\kmh}{z_\kph}\right)\rdd(\tilde x_k^+) + \frac{1}{4}z_\kph\left(2+\frac{z_\kmh^2}{z_\kph^2}\right)\rdd(\tilde x_k^+) \\
	 &+\frac{1}{4}z_\kph\left(2+\frac{z_\kph}{z_\kmh}\right)\rdd(\tilde x_k^-) + \frac{1}{4}z_\kmh\left(2+\frac{z_\kph^2}{z_\kmh^2}\right)\rdd(\tilde x_k^-) 
	 - \frac3\delta\int_{\xi_\kmh}^{\xi_\kph}z_k\rdd\circ\theX_\Delta\dd\xi,
\end{split}\end{align*}
and furthermore
\begin{align*}\begin{split}
	(A)=&\frac{1}{4}\left[z_\kmh\left(\frac{z_\kmh}{z_\kph}-1\right) + z_\kph\left(\frac{z_\kmh^2}{z_\kph^2}-1\right)\right]\rdd(\tilde x_k^+) \\
	 &+\frac{1}{4}\left[z_\kph\left(\frac{z_\kph}{z_\kmh}-1\right) + z_\kmh\left(\frac{z_\kph^2}{z_\kmh^2}-1\right)\right]\rdd(\tilde x_k^-) \\
	 &-\frac{3}{2\delta}\int_{\xi_\kmh}^{\xi_\kph}z_k\big[\rdd\circ\theX_\Delta-\rdd(\tilde x_k^+)\big]\dd\xi
		-\frac{3}{2\delta}\int_{\xi_\kmh}^{\xi_\kph}z_k\big[\rdd\circ\theX_\Delta-\rdd(\tilde x_k^-)\big]\dd\xi.
\end{split}\end{align*}
Applying the trivial identity (for arbitrary numbers $p$ and $q$)
\begin{align*}
	q\left(\frac{p^2}{q^2}-1\right) = (p+q)\left(\frac{p}{q}-1\right),
\end{align*}
the above term finally reads as
\begin{align}\begin{split}\label{eq:R1_step2_tf}
	%&\frac{1}{2}\left(z_\kph^2 + z_\kmh^2 + z_\kph z_\kmh\right) \Big(\frac{\rdd(\tilde x_k^+)}{z_\kph} + \frac{\rdd(\tilde x_k^-)}{z_\kmh}\Big) 
	%- \frac3\delta\int_{\xi_\kmh}^{\xi_\kph}z_k\rdd\circ\theX_\Delta\dd\xi \\
	(A)=&\frac{1}{4}\left[z_\kmh\left(\frac{z_\kmh}{z_\kph}-1\right) + 2z_k\left(\frac{z_\kmh}{z_\kph}-1\right)\right]\rdd(\tilde x_k^+) \\
	 &+\frac{1}{4}\left[z_\kph\left(\frac{z_\kph}{z_\kmh}-1\right) + 2z_k\left(\frac{z_\kph}{z_\kmh}-1\right)\right]\rdd(\tilde x_k^-) \\
	 &-\frac{3}{2\delta}\int_{\xi_\kmh}^{\xi_\kph}z_k\big[\rdd\circ\theX_\Delta-\rdd(\tilde x_k^+)\big]\dd\xi
		-\frac{3}{2\delta}\int_{\xi_\kmh}^{\xi_\kph}z_k\big[\rdd\circ\theX_\Delta-\rdd(\tilde x_k^-)\big]\dd\xi.
\end{split}\end{align}
  Since $\tilde x_k^+$ lies between the values $x_k$ and $x_{k+1}$, 
	and $\theX_\Delta(\xi)\in[x_k,x_\kph]$ for each $\xi\in[\xi_k,\xi_\kph]$,
  we conclude that $|\theX_\Delta(\xi)-\tilde x_k^+|\le x_{k+1}-x_k$,
  and therefore
  \begin{align}
    \label{eq:rhotox_tf}
    \frac{3}{2\delta}\int_{\xi_k}^{\xi_\kph}z_k\big|\rdd\circ\theX_\Delta(\xi)-\rdd(\tilde x_k^+)\big|\dd\xi
    \le \frac{3}{2}\rkap z_k(x_{k+1}-x_k).
  \end{align}
  A similar estimate is valid for the other integral over $[\xi_\kmh,\xi_k]$ and for the integrals with $\rdd(\tilde x_k^-)$.
  Thus, combining \eqref{eq:R1_step2_tf} and \eqref{eq:rhotox_tf} with 
	$z_{k\pm\frac{1}{2}}^n\leq 2z_k^n$ and the definition of $\Ai$, one attains that
  \begin{align*}
		\left|\Ai - B_1^n\right| 
		&\leq 2\rkap\sum_{k\in\ivalp}
					z_k^n\left(\frac{z^n_\kph-z^n_\kmh}{\delta}\right)^2
					\left[\left(\frac{z_\kph^n}{z_\kmh^n}-1\right) + \left(\frac{z_\kmh^n}{z_\kph^n}-1\right)\right] \\
		&\quad+ 3\rkap\sum_{k\in\ivalp}
					z_k^n\left(\frac{z^n_\kph-z^n_\kmh}{\delta}\right)^2(x_{k+1}^n-x_{k-1}^n),
  \end{align*}
	and further, applying \eqref{eq:weakoscillation_tf} and \eqref{eq:weakoscillation2_tf},
	\begin{align}\label{eq:R1_step3_tf}
		\tau\sum_{n=1}^{\Nt}\left|\Ai - B_1^n\right|  \leq C\delta^{1/4}.
	\end{align}
By triangle inequality, \eqref{eq:R1_step1_tf} and \eqref{eq:R1_step3_tf} provide the claim.
\end{proof}

Along the same lines, one proves the analogous estimate for $\Aii$ and $\Aiii$ in place of $\Ai$:
\begin{lem}\label{lem:R23_tf}
  There are constants $C_2>0$ and $C_3>0$ expressible in $\Omega$, $T$, $\rkap$ and $\olDi$
	such that
  \begin{align*}
    \Rii&:=\tau\sum_{n=1}^{\Nt}\bigg|\Aii - \frac{1}{4}\intM \hatz_\Delta^n(\xi)\partial_\xi\hatz^n_\Delta(\xi)^2\rdd\circ\theX^n_\Delta(\xi)\dd\xi\bigg|
    \le C_2\delta^{1/4}, \\
		\Riii&:=\tau\sum_{n=1}^{\Nt}\bigg|\Aiii - \frac{1}{4}\intM \hatz_\Delta^n(\xi)\partial_\xi\hatz^n_\Delta(\xi)^2\rdd\circ\theX^n_\Delta(\xi)\dd\xi\bigg|
    \le C_3\delta^{1/4}.
  \end{align*}
\end{lem}
\begin{lem}\label{lem:R4_tf}
  There is a constant $C_4>0$ expressible in $\Omega$, $T$, $\rkap$ and $\olDi$
	such that
  \begin{align*}
    \Riv&:=\tau\sum_{n=1}^{\Nt}\bigg|\Aiv - \intM \hatz_\Delta^n(\xi)\partial_\xi\hatz^n_\Delta(\xi)^2\rdd\circ\theX^n_\Delta(\xi)\dd\xi\bigg|
    \le C_4\delta^{1/4}.
  \end{align*}
\end{lem}
\begin{proof}
The proof is almost identical to the one for Lemma \ref{lem:R1_tf} above.
As before, we introduce the term
\begin{align*}
	B_4^n := \delta\sum_{k\in\ivalp}z_k^n\left(\frac{z^n_\kph-z^n_\kmh}{\delta}\right)^2\frac{1}{\delta}\int_{\xi_\kmh}^{\xi_\kph}\rdd\circ\theX^n_\Delta(\xi)\dd\xi
\end{align*}
and get due to \eqref{eq:weakoscillation_tf}, analogously to \eqref{eq:R1_step1_tf}, that
\begin{align}\label{eq:R4_step1_tf}
	\tau\sum_{n=1}^{\Nt}\bigg|B_4^n - \intM \hatz_\Delta^n(\xi)\partial_\xi\hatz^n_\Delta(\xi)^2\rdd\circ\theX^n_\Delta(\xi)\dd\xi\bigg|
	\le C\delta^{1/4}.
\end{align}
%Fix some time index $n\in\setN$ (omitted in the calculations below).
By writing
\begin{align*}
	\frac{(z^n_\kph)^3+(z^n_\kmh)^3}{2 z^n_\kph z^n_\kmh} 
	= z_k^n\Big(\frac{z^n_\kmh}{z^n_\kph}-1\Big) + z_k^n\Big(\frac{z^n_\kph}{z^n_\kmh}-1\Big) + z_k^n,
\end{align*}
one obtains that
\begin{align*}
	&\frac{(z^n_\kph)^3+(z^n_\kmh)^3}{2 z^n_\kph z^n_\kmh} \rdd(x_k^n) - \frac{1}{\delta}\int_{\xi_\kmh}^{\xi_\kph}z_k\rdd\circ\theX^n_\Delta(\xi)\dd\xi \\
	=& z_k^n\Big(\frac{z^n_\kmh}{z^n_\kph}-1\Big) + z_k^n\Big(\frac{z^n_\kph}{z^n_\kmh}-1\Big) 
		-\frac{1}{\delta}\int_{\xi_\kmh}^{\xi_\kph}z_k^n\big[\rdd\circ\theX^n_\Delta(\xi) - \rdd(x_k^n)\big]\dd\xi.
\end{align*}
Observing --- in analogy to \eqref{eq:rhotox_tf} --- that
\begin{align*}
	\frac1\delta\int_{\xi_\kmh}^{\xi_\kph}z_k^n\big|\rdd\circ\theX_\Delta^n(\xi)-\rdd(x_k^n)\big|\dd\xi
	\le \rkap z_k^n (x_\kph^n-x_\kmh^n),
\end{align*}
we obtain the same bound on $|\Aiv-B_4^n|$ as before on $|\Ai-B_1^n|$, i.e. 
\begin{align*}
	\left|\Aiv - B_4^n\right| 
	&\leq \rkap\sum_{k\in\ivalp}
				z_k^n\left(\frac{z^n_\kph-z^n_\kmh}{\delta}\right)^2
				\left[\left(\frac{z_\kph^n}{z_\kmh^n}-1\right) + \left(\frac{z_\kmh^n}{z_\kph^n}-1\right)\right] \\
	&\quad+ \rkap\sum_{k\in\ivalp}
				z_k^n\left(\frac{z^n_\kph-z^n_\kmh}{\delta}\right)^2(x_{k+1}^n-x_{k-1}^n).
\end{align*}
Again, applying \eqref{eq:weakoscillation_tf} and \eqref{eq:weakoscillation2_tf}, we get
\begin{align}\label{eq:R4_step3_tf}
	\tau\sum_{n=1}^{\Nt}\left|\Aiv - B_4^n\right|  \leq C\delta^{1/4},
\end{align}
and the estimates \eqref{eq:R4_step1_tf} and \eqref{eq:R4_step3_tf} imply the desired bound on $\Riv$.
\end{proof}

\begin{lem}\label{lem:R5}
  There is a constant $C_5>0$ expressible in $\Omega$, $T$, $\rkap$ and $\olDi$
  such that 
  \begin{align*}
    \Rv:=\tau\sum_{n=1}^{\Nt}\left|\Av - \frac{1}{2}\intM \hatz_\Delta^n(\xi)\partial_\xi\hatz^n_\Delta(\xi)\rddd\circ\theX^n_\Delta(\xi)\dd\xi\right|
    \le C_5\delta^{1/4}.
  \end{align*}
\end{lem}
\begin{proof}
The idea of the proof is the same as in the previous proofs.
Let us define similar to $B_1^n$ the term
\begin{align*}
	B_5^n := \delta\sum_{k\in\ivalp}z_\kph^n\left(\frac{z^n_\kph-z^n_\kmh}{\delta}\right)\frac{1}{2\delta}\int_{\xi_\kmh}^{\xi_\kph}\rddd\circ\theX^n_\Delta(\xi)\dd\xi.
\end{align*}
Note in particular that we weight the integral here with $z_\kph^n$. 
Then
\begin{align*}\begin{split} 
	\left|\frac{1}{2}\intM \hatz_\Delta^n(\xi)\partial_\xi\hatz^n_\Delta(\xi)\rddd\circ\theX^n_\Delta(\xi)\dd\xi
	      - B_5^n \right|
	\leq \frac{1}{2}\rkap\sum_{k\in\ivalp} \left(\frac{z^n_\kph-z^n_\kmh}{\delta}\right) \int_{\xi_\kmh}^{\xi_\kph} \big|\hatz_\Delta^n(\xi)-z_\kph^n\big| \dd\xi  \\
	\leq \frac{1}{2}\rkap\delta\sum_{k\in\ivalp} z_\kph^n\left|\frac{z^n_\kph-z^n_\kmh}{\delta}\right| \left|\frac{z_\kmh^n}{z_\kph^n}-1\right|,
\end{split}\end{align*}
where we used that by definition of $\hatz_\Delta^n$,
\begin{align*}
	\intM \hatz^n_\Delta(\xi)\partial_\xi\hatz^n_\Delta(\xi)\rddd\circ\theX^n_\Delta(\xi)\dd\xi
	= \sum_{k\in\ivalp} \left(\frac{z^n_\kph-z^n_\kmh}{\delta}\right) \int_{\xi_\kmh}^{\xi_\kph}\hatz^n_\Delta(\xi) \rddd\circ\theX^n_\Delta(\xi)\dd\xi.
\end{align*}
This especially implies, due to \eqref{eq:weakoscillation_tf} that
\begin{align}\label{eq:R5_step1_tf}
	\tau\sum_{n=1}^{\Nt}\bigg|B_5^n - \frac{1}{2}\intM \hatz_\Delta^n(\xi)\partial_\xi\hatz^n_\Delta(\xi)\rddd\circ\theX^n_\Delta(\xi)\dd\xi\bigg|
	\le C\delta^{1/4}.
\end{align}
Furthermore, one can introduce intermediate values $\tilde x_k^+$ such that
\begin{align*}
	\rd(x^n_{k+1})-\rd(x^n_k)-(x^n_{k+1}-x^n_k)\rdd(x^n_k)
	=\frac12(x^n_{k+1}-x^n_k)^2\rddd(\tilde x_k^+)
	=\frac{\delta^2}{2(z^n_\kph)^2}\rddd(\tilde x_k^+).
\end{align*}
Using the identity
\begin{align*}
	\left(\frac{(z_\kph^n)^3+(z_\kmh^n)^3}2\right)\frac{1}{2(z^n_\kph)^2}
	= \frac{z_\kph^n}{2} + \frac{z_\kmh^n}{4}\left(\frac{(z_\kmh^n)^2}{(z_\kph^n)^2}-1\right) + \frac{z_\kph^n}{4}\left(\frac{z_\kmh^n}{z_\kph^n}-1\right),
\end{align*}
we thus have --- using again \eqref{eq:dummy808_tf} --- that
\begin{align*}
	&\left(\frac{(z^n_\kph)^3+(z^n_\kmh)^3}2\right)
	\left(\frac{\rd(x^n_{k+1})-\rd(x^n_k)-(x^n_{k+1}-x^n_k)\rdd(x^n_k)}{\delta^2}\right)-\frac1{2\delta}\int_{\xi_\kmh}^{\xi_\kph}z_k^n\rddd\circ\theX^n_\Delta\dd\xi \\
	=& \frac{z_\kmh^n}{4}\left(\frac{(z_\kmh^n)^2}{(z_\kph^n)^2}-1\right) + \frac{z_\kph^n}{4}\left(\frac{z_\kmh^n}{z_\kph^n}-1\right)
	- \frac1{2\delta}\int_{\xi_\kmh}^{\xi_\kph}z_\kph^n\big[\rddd\circ\theX^n_\Delta-\rddd(\tilde x_k^+)\big]\dd\xi.
\end{align*}
Observing --- in analogy to \eqref{eq:rhotox_tf} --- that
\begin{align*}
	\frac1{2\delta}\int_{\xi_\kmh}^{\xi_\kph}z_k^n\big|\rdd\circ\theX_\Delta^n(\xi)-\rdd(x_k^n)\big|\dd\xi
	\le \frac{\rkap}{2} z_\kph^n (x_\kph^n-x_\kmh^n),
\end{align*}
and $z_\kph^n\leq 2z_k^n$, we obtain the following bound on $|\Av-B_5^n|$:
\begin{align*}
	\left|\Av - B_5^n\right| 
	&\leq \frac{\rkap}{4}\sum_{k\in\ivalp}
				z_k^n\left(\frac{z^n_\kph-z^n_\kmh}{\delta}\right)
				\left[\left(\frac{(z_\kmh^n)^2}{(z_\kph^n)^2}-1\right) + \left(\frac{z_\kmh^n}{z_\kph^n}-1\right)\right] \\
	&\quad+ \rkap\sum_{k\in\ivalp}
				z_k^n\left(\frac{z^n_\kph-z^n_\kmh}{\delta}\right)(x_{k+1}^n-x_{k-1}^n).
\end{align*}
Again, applying \eqref{eq:weakoscillation_tf} and \eqref{eq:weakoscillation2_tf}, we get
\begin{align}\label{eq:R5_step3_tf}
	\tau\sum_{n=1}^{\Nt}\left|\Av - B_5^n\right|  \leq C\delta^{1/4},
\end{align}
and the estimates \eqref{eq:R5_step1_tf} and \eqref{eq:R5_step3_tf} imply the desired bound on $\Rv$.
\end{proof}
Arguing like in the previous proof, one shows the analogous estimate for $\Avi$ in place of $\Av$.
It remains to analyze the potential term $\Avii$,
where we instantaneously identify the $\xi$-integral with the $x$-integral:
\begin{lem}\label{lem:R7_tf}
  There is a constant $C_7>0$ expressible in $\Omega$, $T$ and $\rkap$
  such that 
  \begin{align*}
    \Rvii:=\tau\sum_{n=1}^{\Nt}\left|\Avii - \intom V_x(x) u_\Delta^n(x) \dd x\right|
    \le C_7\delta.
  \end{align*}
\end{lem}
\begin{proof}
Since the product $V_x\rho_x$ is a smooth function on the domain $\Omega$, 
we can invoke the mean-value theorem and find intermediate values $\tilde{x}_k$, such that
\begin{align*}
	\left|\delta\sum_{k\in\ivalp} V_x(x_k^n)\rho_x(x_k^n) - \intM V_x(\cX_\Delta^n(\xi))\rho_x(\cX_\Delta^n(\xi))\dd\xi\right| 
	&\leq \delta\sum_{k\in\ivalp} \partial_x\big(V_x\rho_x\big)(\tilde{x}_k)(x_\kappp^n-x_\kappm^n) \\
	&\leq \delta(b-a)\sup_{x\in\Omega}\big|V_x(x)\rho_x(x)\big|.
\end{align*}
The claim then follows by a change of variables.
\end{proof}
It remains to identify the integral expressions inside $\Ri$ to $\Rv$ with those in the weak formulation \eqref{eq:weakgoal_tf}.
\begin{lem}
  One has that
  \begin{align}
    \label{eq:cov1_tf}
    &\intM \hatz_\Delta^n(\xi)\partial_\xi\hatz_\Delta^n(\xi)\rddd\circ\theX^n_\Delta(\xi)\dd\xi 
    = \frac{1}{2}\intom \partial_x\big(\hatu_\Delta^n(x)\big)^2 \rddd(x)\dd x, \\
    \label{eq:cov2_tf}
    &\Rviii:=\tau\sum_{n=1}^{\Nt}\left|\intM \hatz_\Delta^n(\xi)(\partial_\xi\hatz_\Delta^n)^2(\xi)\rdd\circ\theX^n_\Delta(\xi)\dd\xi 
      - \intom (\partial_x\hatu_\Delta^n)^2(x)\rdd(x)\dd x\right|
    \le C_8\delta^{1/4}
 \end{align}
for a constant $C_8>0$ expressible in $\Omega$, $T$, $\rkap$ and $\olDi$.
\end{lem}
\begin{proof}
  The starting point is the relation \eqref{eq:locaffine0} between the locally affine interpolants $\hatu_\Delta^n$ and $\hatz_\Delta^n$ that is 
  \begin{align}
    \label{eq:locaffine2_tf}
    \hatz_\Delta^n(\xi)=\hatu_\Delta^n\circ\cX_\Delta^n(\xi)     
  \end{align}
  for all $\xi\in\M$.
  Both sides of this equation are differentiable at almost every $\xi\in\M$, with
  \begin{align*}
    \partial_\xi\hatz_\Delta^n(\xi) =
    \partial_x\hatu_\Delta^n\circ\cX_\Delta^n(\xi)\partial_\xi\cX_\Delta^n(\xi).
  \end{align*}
  Substitute this expression for $\partial_\xi\hatz_\Delta^n(\xi)$ into the left-hand side of \eqref{eq:cov1_tf},
  and perform a change of variables $x=\cX_\Delta^n(\xi)$ to obtain the integral on the right.
  
Next observe that the $x$-integral in \eqref{eq:cov2_tf} can be written as
\begin{align}\label{eq:dummy201_tf}
	\intom (\partial_x\hatu_\Delta^n)^2(x)\rdd(x)\dd x
	= \intM (\partial_\xi\hatz_\Delta^n)^2(\xi)\frac{1}{\partial_\xi\cX_\Delta^n}\rdd\circ\cX_\Delta^n(\xi) \dd\xi,
\end{align}
using \eqref{eq:locaffine2_tf}.
It hence remains to estimate the difference between the $\xi$-integral in \eqref{eq:cov2_tf} and \eqref{eq:dummy201_tf}, respectively.
To this end, observe that for each $\xi\in(\xi_k,\xi_\kph)$ with some $k\in\ivalp$, 
one has $\partial_\xi\cX_\Delta^n(\xi)=1/z^n_\kph$ and $\hatz_\Delta(\xi)\in[z_\kmh,z_\kph]$.
Hence, for those $\xi$,
\begin{align*}
	\left|1-\frac1{\hatz_\Delta^n(\xi)\partial_\xi\cX_\Delta^n(\xi)}\right| \le \left|1-\frac{z^n_\kph}{z^n_\kmh}\right|.
\end{align*}
If instead $\xi\in(\xi_\kmh,\xi_k)$, then this estimate is satisfied with the roles of $z^n_\kph$ and $z^n_\kmh$ interchanged.
Consequently, 
\begin{align*}
	&\left|\intM (\partial_\xi\hatz_\Delta^n)^2(\xi) \hatz_\Delta^n(\xi)\rdd\circ\cX_\Delta^n(\xi) \dd\xi
	- \intM (\partial_\xi\hatz_\Delta^n)^2(\xi)\frac{1}{\partial_\xi\cX_\Delta^n(\xi)}\rdd\circ\cX_\Delta^n(\xi) \dd\xi\right| \\
	\leq& \rkap\intM (\partial_\xi\hatz_\Delta^n)^2(\xi) \hatz_\Delta^n(\xi)\left(1-\frac{1}{\hatz_\Delta^n(\xi)\partial_\xi\cX_\Delta^n(\xi)}\right) \dd \xi \\
	\leq& \rkap\delta\sum_{k\in\ivalp}z_k^n\left(\frac{z_\kph^n-z_\kmh^n}{\delta}\right)^2 
			\left(\left|\frac{z_\kph}{z_\kmh}-1\right| + \left|\frac{z_\kmh}{z_\kph}-1\right|\right),
\end{align*}
which is again at least of order $\mathcal{O}(\delta^{\frac{1}{4}})$, as we have seen before in \eqref{eq:weakoscillation_tf}.
\end{proof}

\begin{proof2}{\eqref{eq:weakform2_tf}}
  Combining the discrete weak formulation \eqref{eq:sevenAs_tf},
  the change of variables formulae \eqref{eq:cov1_tf} and \eqref{eq:cov2_tf},
  and the definitions of $\Ri$ to $\Rviii$,
  it follows that
  \begin{align*}
    \er_{2,\Delta}\le \rkap\Rviii + \rkap\tau\sum_{n=1}^{\Nt}\Big|
      \intM \hatz_\Delta^n(\xi)\partial_\xi\hatz_\Delta^n(\xi)\rddd\circ\theX^n_\Delta(\xi)\dd\xi 
						+ \frac{3}{2}\intM \hatz_\Delta^n(\xi)(\partial_\xi\hatz_\Delta^n)^2(\xi)\rdd\circ\theX^n_\Delta(\xi)\dd\xi \\
				+ \intom V_x(x) \ti{u_\Delta}(x) \rd(x) \dd x
      - \sum_{i=1}^7 \Ais^n\Big| \\
    \le \rkap\sum_{i=1}^7 \Ris^n \le \rkap\sum_{i=1}^7 C_i \delta^{1/4}.
  \end{align*}
  This implies the desired inequality \eqref{eq:weakform2_tf}.
\end{proof2}

We are now going to finish the proof of this section's main result, Proposition \ref{prp:weakgoal_tf}.

\begin{proof2}{Proposition \ref{prp:weakgoal_tf}}
Owing to \eqref{eq:weakform1_tf} and \eqref{eq:weakform2_tf}, we know that
\begin{align*}
	\Bigg|\int_0^T \ed(t)\intom \rho(x)\ti{\baru_\Delta}(t,x) \dd x 
	+\eta(t)\frac12\intom \rddd(x)\partial_x(\ti{u_\Delta}^2)(t,x) + 3\rdd(x) (\partial_x\ti{u_\Delta})^2(t,x) \dd x \\
			+ \intom V_x(x) \ti{u_\Delta}(t,x) \rd(x)\dd t\Bigg| \\
	\le \er_{1,\Delta} + \er_{2,\Delta} \le C\big(\tau + \delta^{1/4}\big).
\end{align*}
To obtain \eqref{eq:weakgoal_tf} in the limit $\Delta\to0$, 
we still need to show the convergence of the integrals to their respective limits,
but this is no challenging task anymore:
Note that \eqref{eq:strong_tf} implies
\begin{align}\label{eq:strong2_tf}
	\partial_x\ti{u_\Delta} \longrightarrow \partial_x u_* \quad \text{strongly in $L^2([0,T]\times\Omega)$},
\end{align}
hence $(\partial_x\ti{u_\Delta})^2$ converges to $(\partial_x u_*)^2$ in $L^1([0,T]\times\Omega)$.
Furthermore, we have that
\begin{align}\label{eq:strong3_tf}
	\partial_x(\ti{u_\Delta}^2) = 2 \ti{u_\Delta} \partial_x\ti{u_\Delta}
	\longrightarrow 2 u_* \partial_x(u_*) = \partial_x(u_*^2)
\end{align}
in $L^2([0,T]\times\Omega)$. Here we used \eqref{eq:strong2_tf} and that $\ti{u_\Delta}$ converges to $u_*$ uniformly on $[0,T]\times\Omega$
due to \eqref{eq:uniformLinfty_tf}.
Hence, \eqref{eq:strong2_tf} and \eqref{eq:strong3_tf} suffice to pass to the limit in the second integral.
Finally remember the weak convergence result in \eqref{eq:weak_tf},
$\ti{u_\Delta}\to u_*$ in $\dens$ with respect to time,
hence the convergence of the first and third integral is assured as well.

\end{proof2}

%
% #########################################################################################################################################################
% #########################################################################################################################################################
%
\section{Numerical results}\label{sec:num_tf}
%
% #########################################################################################################################################################
%
%-----------------------------------------------------------------------------------------------------------------------------------------------------------------
\subsection{Non-uniform meshes}
An equidistant mass grid --- as used in the analysis above --- leads to a good spatial resolution of regions where the value of $u^0$ is large,
but provides a very poor resolution in regions where $u^0$ is small.
Since we are interested in regions of low density, and especially in the evolution of supports,  
it is natural to use a \emph{non-equidistant} mass grid with an adapted spatial resolution, like the one defined as follows:
The mass discretization of $\M$ is determined by a vector $\vech=(\xi_0,\xi_1,\xi_2,\ldots,\xi_{K-1},\xi_K)$,
with $0=\xi_0 < \xi_1 < \cdots < \xi_{K-1} < \xi_K = M$
and we introduce accordingly the distances (note the convention $\xi_{-1} = \xi_{K-1} = 0$)
\begin{align*}
  \delta_\kappa = \xi_\kappp-\xi_\kappm, 
  \quad\text{and}\quad 
  \delta_k = \frac12(\delta_\kph+\delta_\kmh) %= \frac12(\xi_{k+1}-\xi_{k-1})
\end{align*}
for $\kappa\in\hval$ and $k\in\ivalp$, respectively.
The piecewise constant density function $u\in\densNN$ corresponding to a vector $\xvec\in\setR^{K-1}$
is now given by
\begin{align*}
  u(x) = z_\kappa \quad\text{for $x_\kappm<x<x_\kappp$}, \quad
  \text{with}\quad z_\kappa = \frac{\delta_\kappa}{x_\kappp-x_\kappm}.
\end{align*}
The Wasserstein-like metric (and its corresponding norm) needs to be adapted as well:
the scalar product $\wspr{\cdot}{\cdot}$ is replaced by
\begin{align*}
  \langle\vvec,\wvec\rangle_\vech = \sum_{k\in\ivalp} \delta_kv_kw_k
	\quad\textnormal{and}\quad \wnrm{\vvec} = \langle\vvec,\vvec\rangle_\vech.
\end{align*}
Hence the metric gradient $\nabla_\vech f(\xvec)\in\setR^{K-1}$ of a function $f:\xseqNN\to\setR$ at $\xvec\in\xseqNN$
is given by
\begin{align*}
  \big[\nabla_\vech f(\xvec)\big]_k = \frac1{\delta_k}\partial_{x_k}f(\xvec).
\end{align*}
Otherwise, we proceed as before:
the entropy is discretized by restriction, and the discretized information functional is the self-dissipation of the discretized entropy.
Explicitly, the resulting fully discrete gradient flow equation attains the form
\begin{align}\label{eq:nonuniform}
  \frac{\xvec_\Delta^n-\xvec_\Delta^{n-1}}\tau = - \nabla_\vech\DiVz(\xvec_\Delta^n)
\end{align}

%
%-----------------------------------------------------------------------------------------------------------------------------------------------------------------
\subsection{Implementation}
%
%To guarantuees the existence of an initial vector $\xvec_\Delta^0\in\xseqN$, 
%which "reaches" any mass point of $u^0$, i.e.  
%$[x_0^0,x_K^0]\subseteq\operatorname{supp}(u^0)$,
%one has to consider initial density functions $u^0$ with compact support.
%
Starting from the initial condition $\xvec_\Delta^0$, the fully discrete solution is calculated inductively 
by solving the implicit Euler scheme \eqref{eq:nonuniform} for $\xvec_\Delta^n$, given $\xvec_\Delta^{n-1}$.
In each time step, a damped Newton iteration is performed, with the solution from the previous time step as initial guess.

%
%-----------------------------------------------------------------------------------------------------------------------------------------------------------------
\subsection{Numerical experiments}
In the following numerical experiments, we fix $\Omega=(0,1)$.

\subsubsection{Evolution of discrete solutions}
In a paper of Gruen and Beck \cite{BGruen}, the authors analyzed, among other things, the behaviour of equation \eqref{eq:thinfilm}
on the bounded domain $(0,1)$ with Neumann-boundary conditions and the initial datum
\begin{align}\label{eq:u0}
	u_\eps^0(x) = (x-0.5)^4 + \eps,\quad x\in(0,1),\quad\textnormal{with mass } M=0.0135,
\end{align}
with $\eps=10^{-3}$. 
This case is interesting insofar as the observed film seems to rip at time $t=0.012$. 
%Although our numerical scheme is formulated on the spatial domain $\setR$, we can describe the same problem just by choosing
%\begin{align*}
	%V(x) = \begin{cases} 0, & x\in[0,1], \\ +\infty, &\textnormal{otherwise}\end{cases},
%\end{align*}
%or by fixing $x_0$ and $x_K$ as in \cite{dlssv3}.
Figure \ref{fig:fig1} shows the evolution of $u_\Delta$ for $K=400$ and $\tau=10^{-7}$ at times $t=0,0.0022,0.012,0.04$, 
the associated particle flow is printed in figure \ref{fig:fig2}/left.

\begin{figure}[t]
  \centering
  \includegraphics[width=0.45\textwidth]{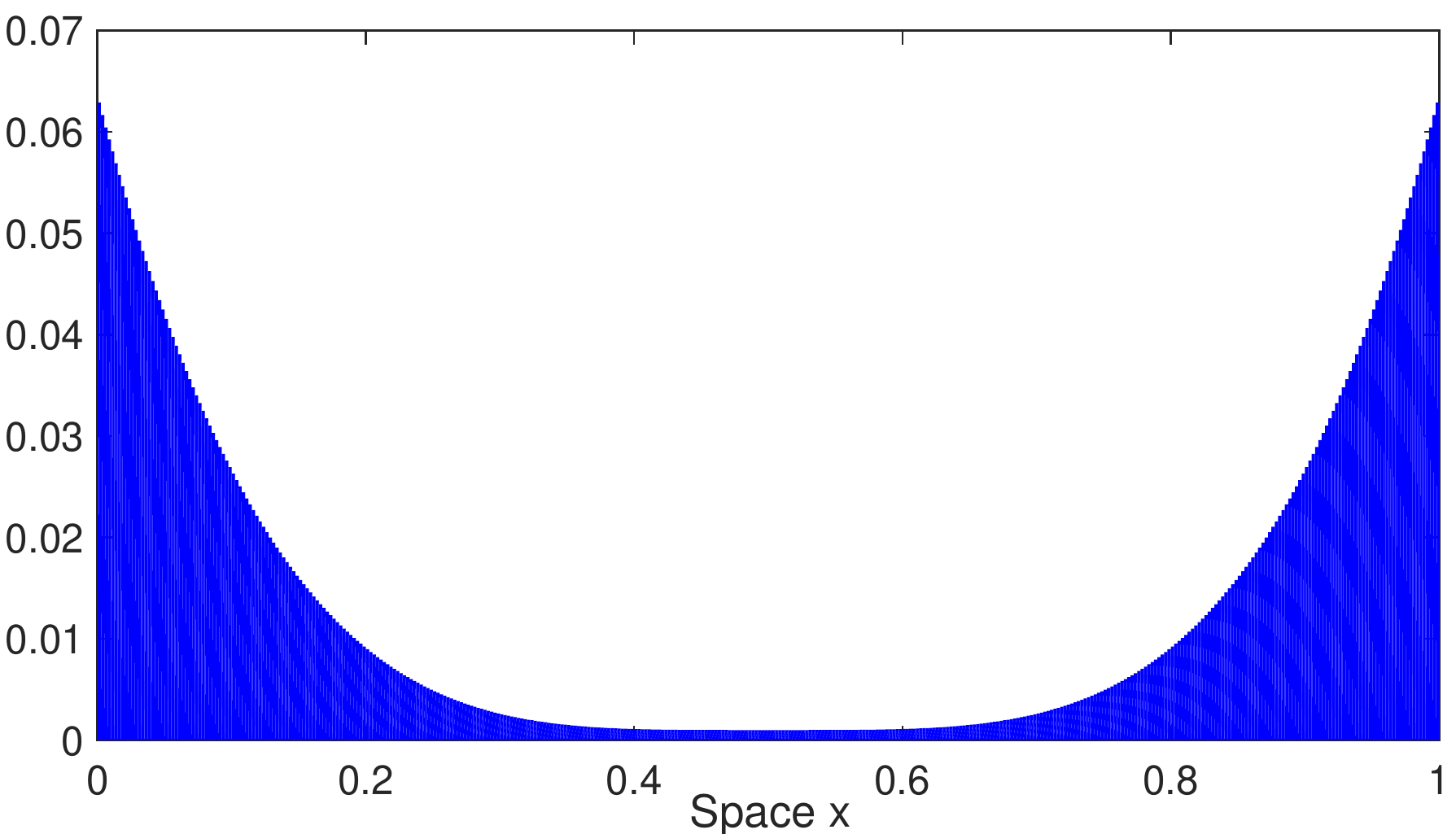}
  \hfill
  \includegraphics[width=0.45\textwidth]{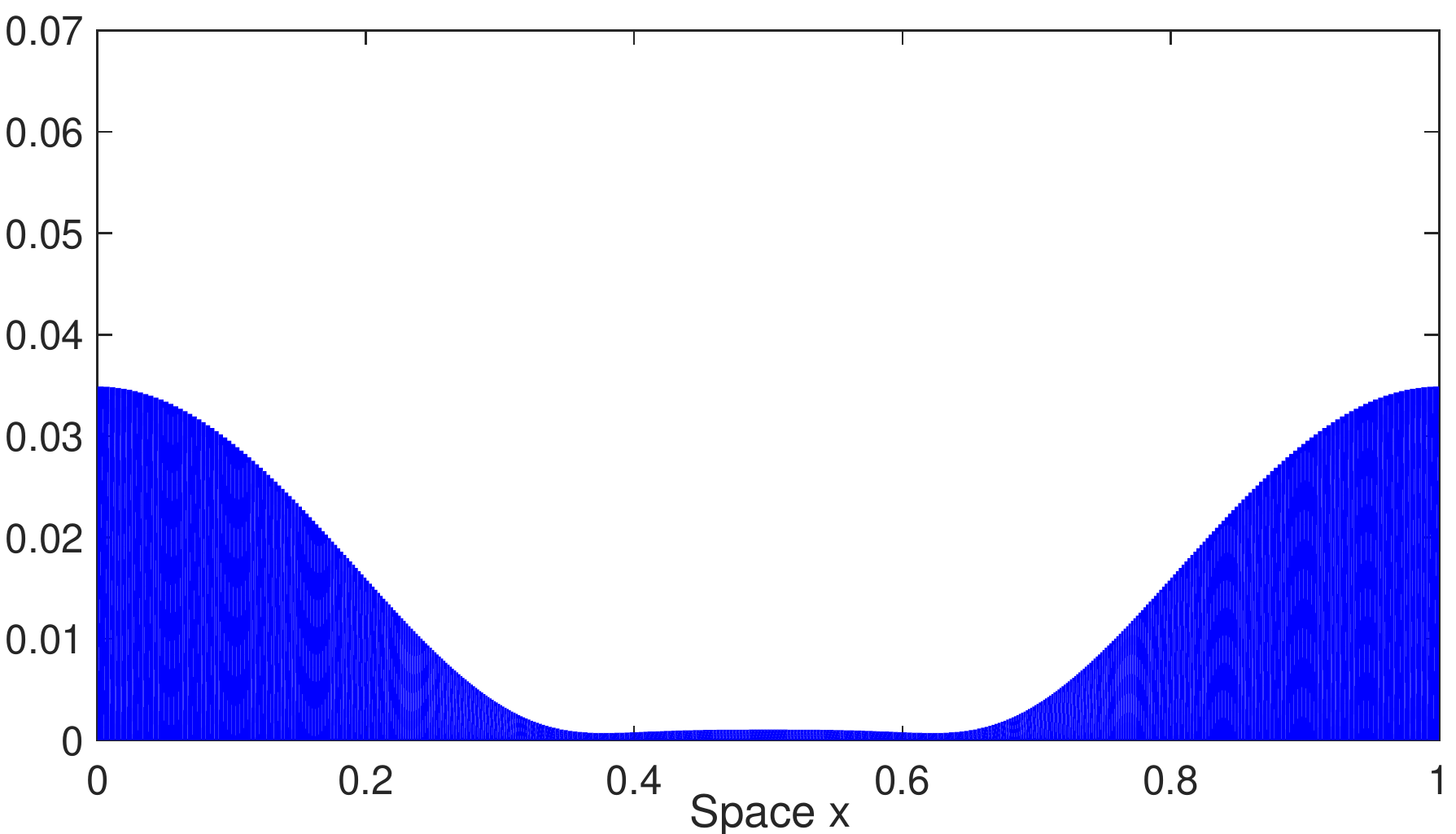}
  \newline
  \includegraphics[width=0.45\textwidth]{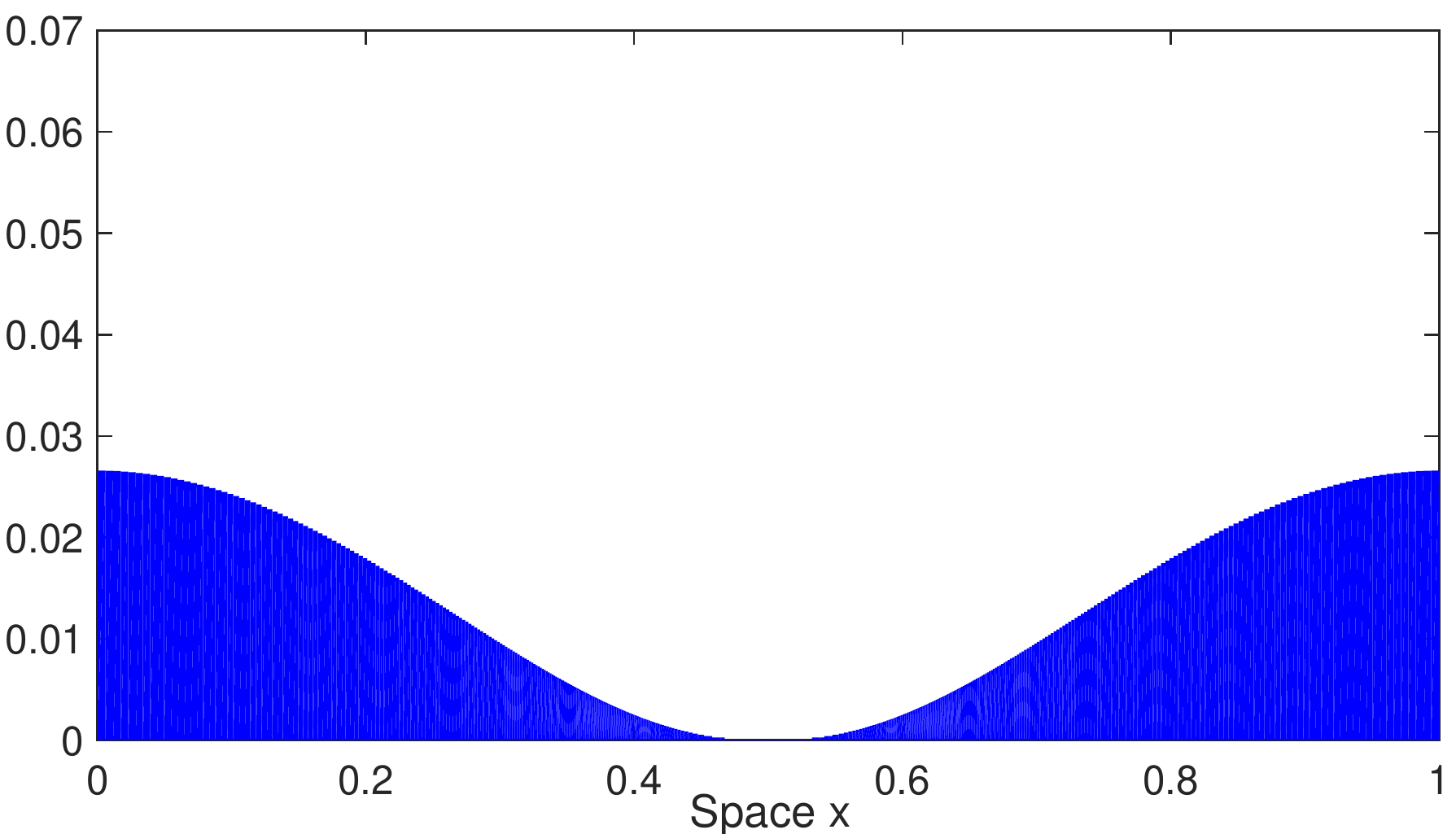}
  \hfill
  \includegraphics[width=0.45\textwidth]{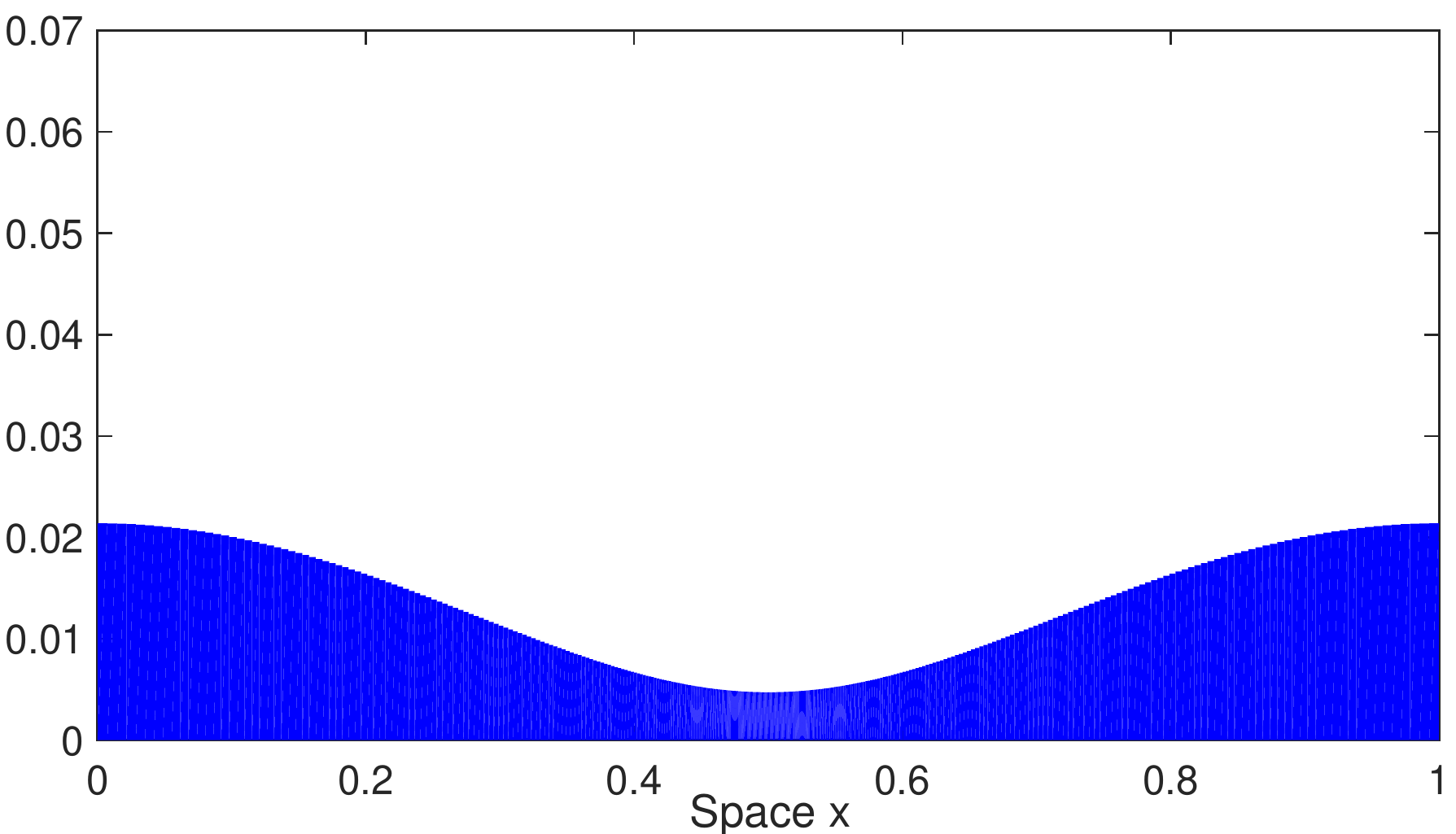}
  \caption{Evolution of a discrete solution $u_\Delta$, evaluated at different times $t = 0,0.002,0.012,0.04$ (from top left to bottom right)}
  \label{fig:fig1}
\end{figure}

\begin{figure}[t]
  \centering
  \subfigure{\includegraphics[width=0.455\textwidth]{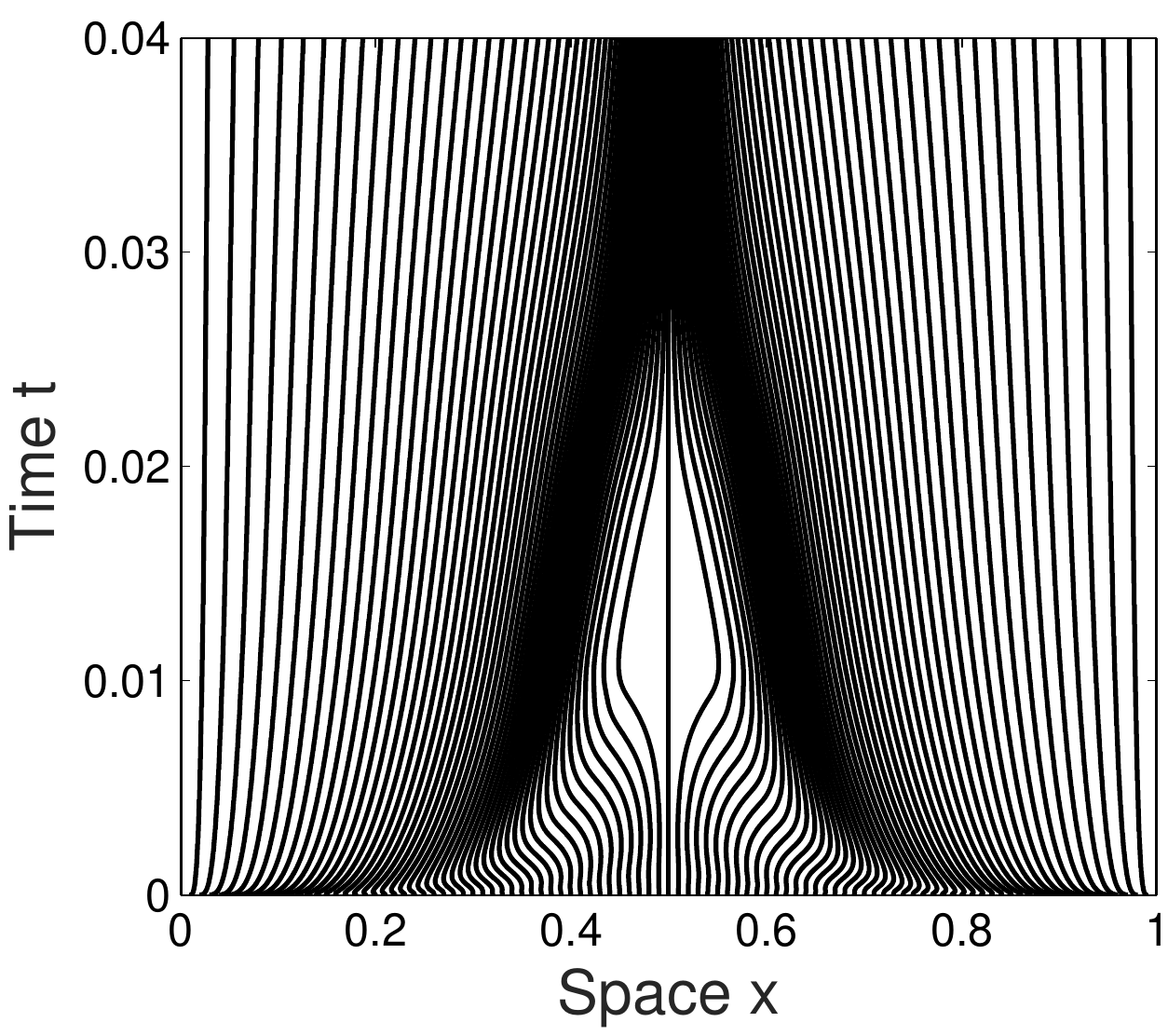}}
   \hfill
   \subfigure{\includegraphics[width=0.45\textwidth]{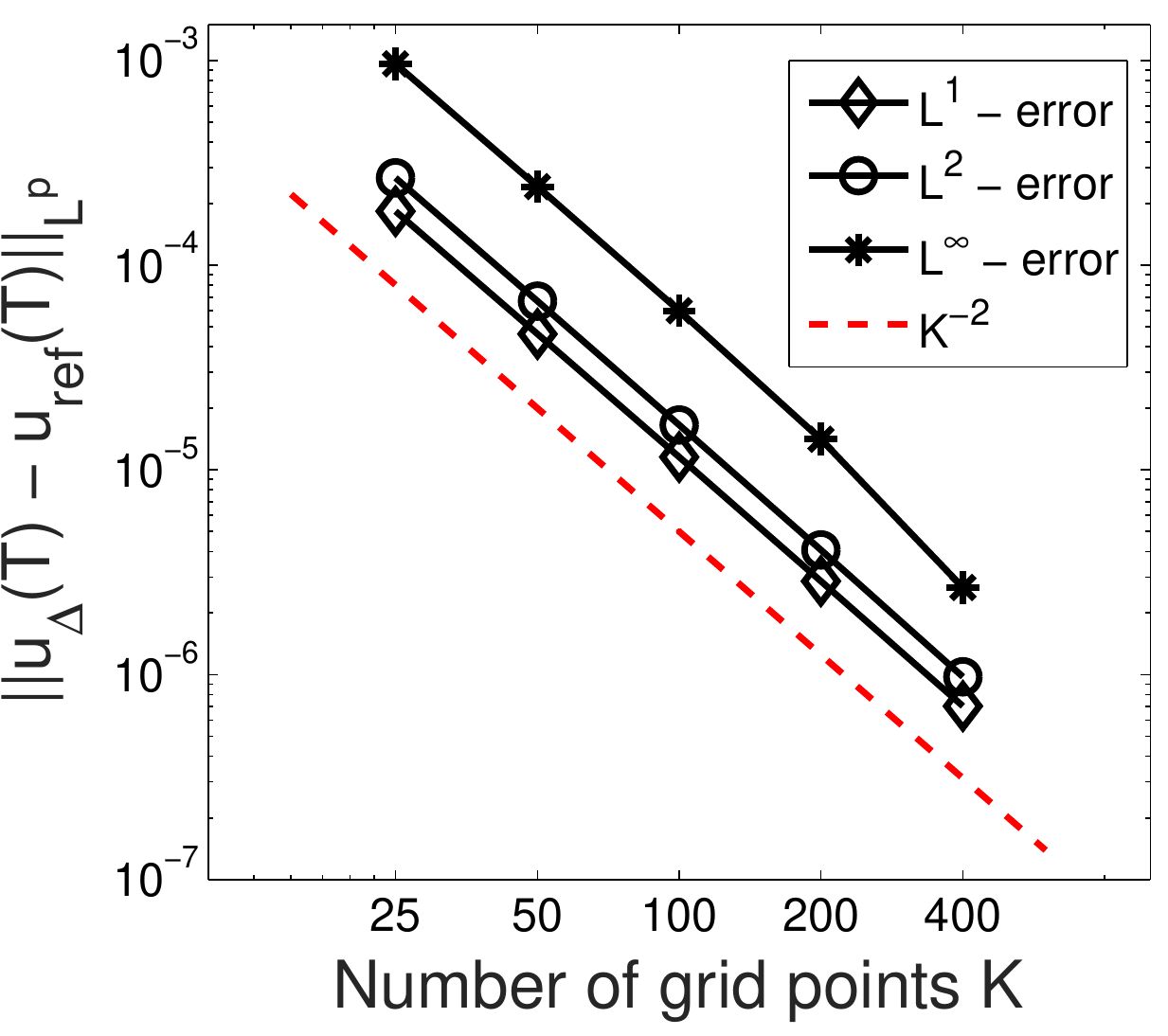}}
   \caption{\emph{Left}: Associated particle flow of $u_\Delta$ for initial datum \eqref{eq:u0}.
     \emph{Right}: Rate of convergence, using $K=25,50,100,200,400$ and $\tau=10^{-7}$. 
     The errors are evaluated at time $t=10^{-4}$.}
   \label{fig:fig2}
\end{figure}

\subsubsection{Rate of convergence}\label{sec:rate}
For the analysis of the scheme's convergence with initial datum $u_\eps^0$ with $\eps=10^{-3}$ from \eqref{eq:u0},
we fix $\tau=10^{-7}$ and calculate solutions $u_\Delta$ to our scheme with $K=25,50,100,200,400$. 
%The calculation of a reference solution is realized in two different ways: 
A reference solution $u_{\widetilde{\Delta}}$ is obtained by solving \eqref{eq:nonuniform} 
on a much finer grid, which is $\widetilde{\Delta}=(K_{\operatorname{ref}}^{-1};\tau_{\operatorname{ref}})$
with $K_{\operatorname{ref}}=1600$ and $\tau=5\cdot10^{-8}$. 
In figure \ref{fig:fig2}/right, we plot
the $L^1(\Omega)$, $L^2(\Omega)$, and $L^\infty(\Omega)$-norms of the differences $|u_\Delta(t,\cdot)-u_{\widetilde{\Delta}}(t,\cdot)|$
at time $t=10^{-4}$. 
It is clearly seen that the errors decay with an almost perfect rate of $\delta^2\propto K^{-2}$.

\begin{figure}[t]
  \centering
  \subfigure{\includegraphics[width=0.46\textwidth]{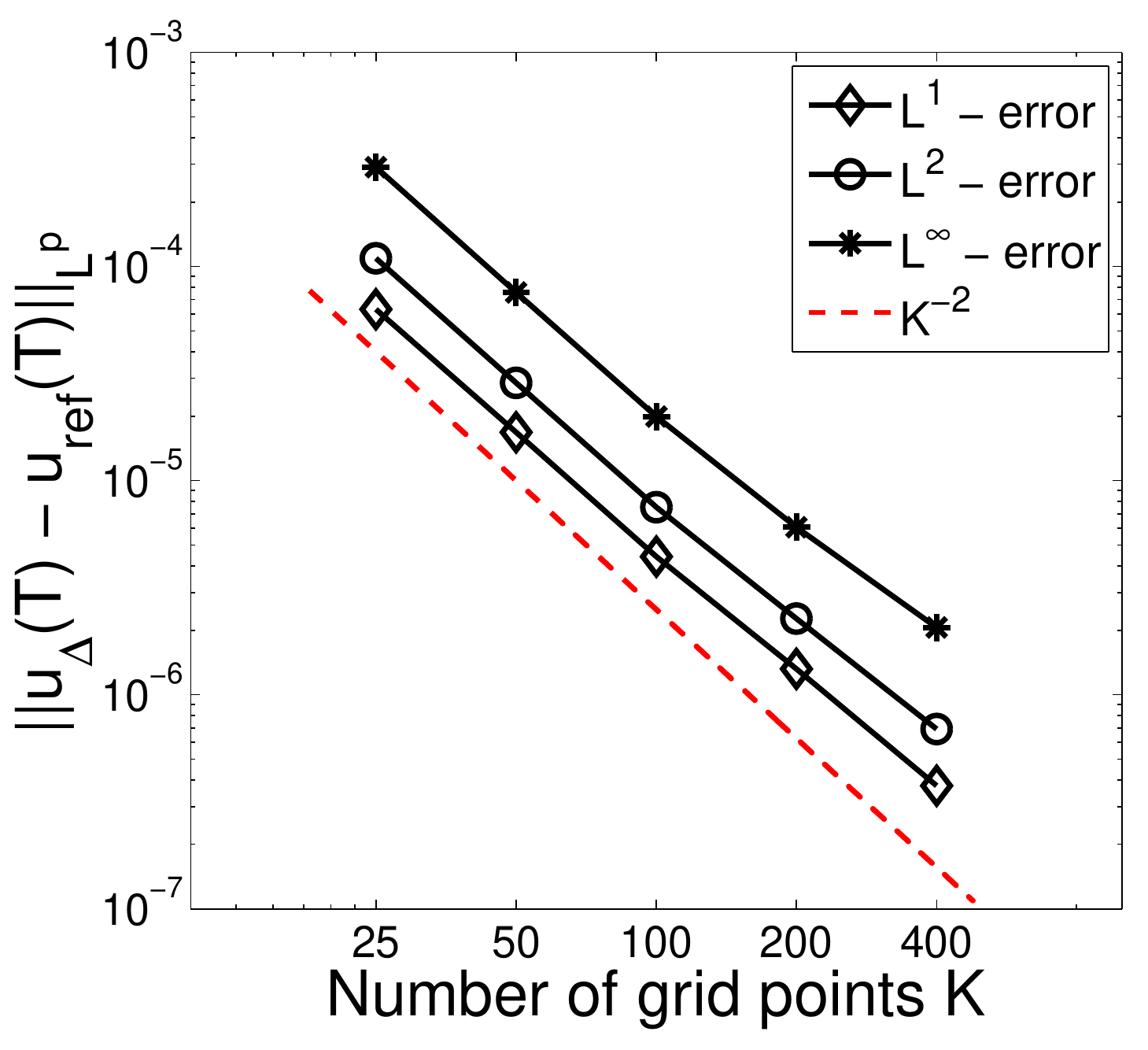}}
   \hfill
   \subfigure{\includegraphics[width=0.46\textwidth]{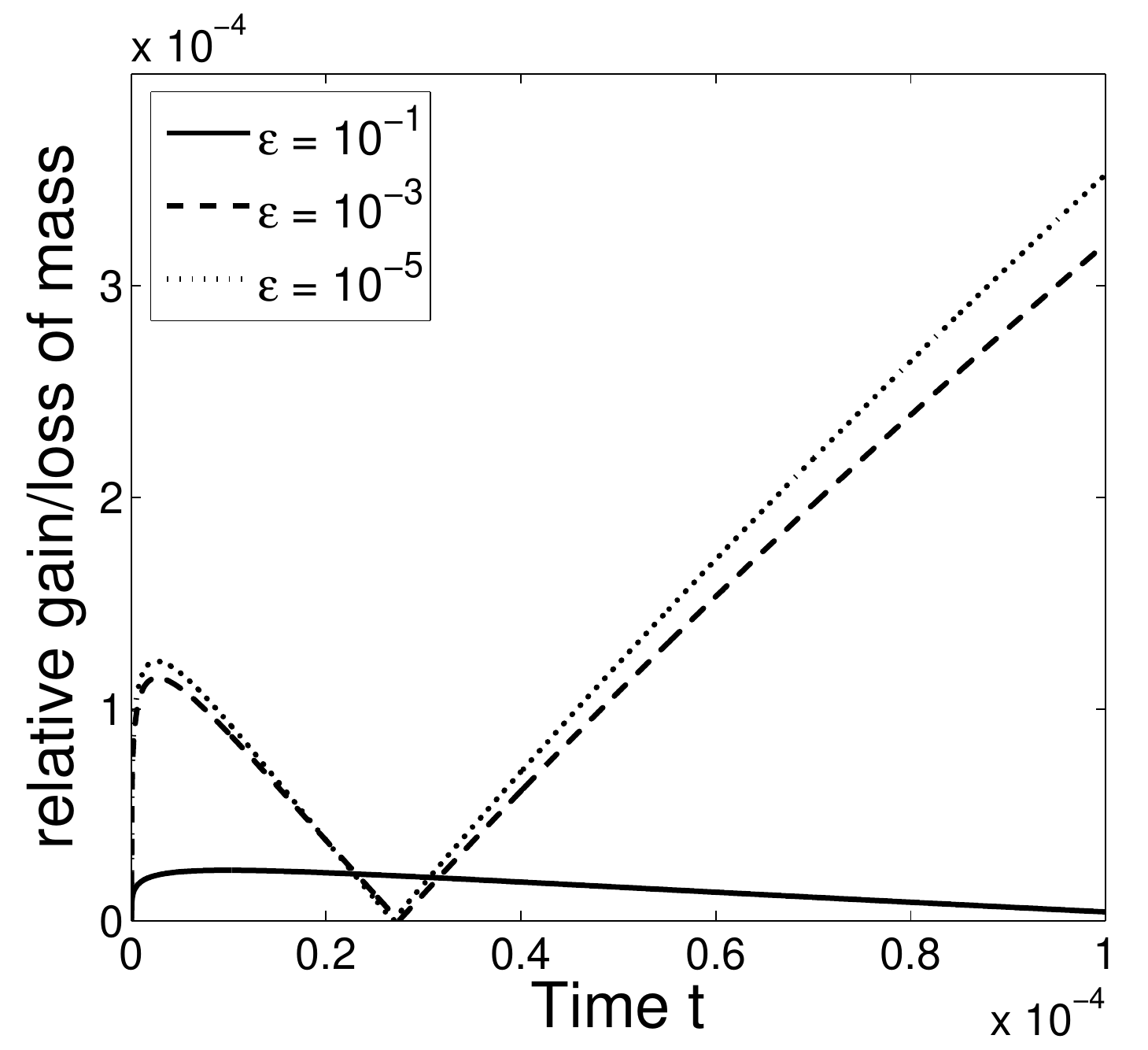}}
   \caption{\emph{Left}: Rate of convergence, using $K=25,50,100,200,400$ and $\tau=10^{-7}$.
     The discrete solutions are compared with a reference solution of the scheme in \eqref{eq:fd_scheme}, and the errors are evaluated at time $t=10^{-4}$.
     \emph{Right}: Loss of mass preservation using a standard finite-difference scheme and $u_\eps^0$ of \eqref{eq:u0} with $\eps=10^{-1},10^{-3}$ and $10^{-5}$.}
   \label{fig:fig3}
\end{figure}

\subsubsection{Comparison with a standard numerical scheme}
For an alternative verification of our scheme's quality, 
we use a reference solution that ist calculated by means of a structurally different discretization of \eqref{eq:thinfilm}.
Specifically, we employ a finite-difference approximation 
with step sizes $\tau_\reff$ and $h_\reff=(b-a)/K_\reff$
in the $t$- and $x$-directions, respectively.
More precisely, with $t_n:=n\tau_\reff$, 
and with $x_k$ for $k=0,\ldots,K_\reff$ being the $K_\reff+1$ equidistant grid points in $[a,b]$,
the numerical approximation $u^n_k\approx u(t_n;x_k)$ of \eqref{eq:thinfilm}
is obtained --- inductively with respect to $n$ --- 
for given vector $u^{n-1}=(u^{n-1}_0,\ldots,u^{n-1}_{K_\reff})$
by solving the fully-implicit difference equation 
\begin{align}
	\label{eq:fd_scheme}
    \frac{u^{n}-u^{n-1}}{\tau_\reff}
    = u^n\cdot\operatorname{D}_\reff^4u^n + \operatorname{D}_\reff^1u^n\cdot\operatorname{D}_\reff^3u^n,
\end{align}
where $u^n=(u^n_0,\ldots,u^n_{K_\reff})$ and 
$\operatorname{D}_\reff^i$ are standard finite difference approximations of the $i$th derivative with equidistant steps $h_\reff$.
The product ``$\cdot$'' of two vectors in \eqref{eq:fd_scheme} shall be understood to act component-by-component.
The boundary conditions \eqref{eq:bc} are enforced using values $u^n_k$ at ``ghost points'' in the obvious way,
that is
\begin{align*}
  u^n_{-1}=u^n_0, \quad u^n_{-2}=u^n_1, \quad u^n_{K_\reff+1}=u^n_{K_\reff},\quad u^n_{K_\reff+2}=u^n_{K_\reff-1}.
\end{align*}
To produce reference solutions for the examples discussed below, 
the scheme above is implemented with $K_\reff=6400$ spatial grid points and a time step $\tau_\reff=5\cdot10^{-8}$.
At a given time $T=N\tau_\reff$, the respective reference profile $x\mapsto u_\reff(T,x)$ 
is defined via piecewise linear interpolation of the respective values $u^n_k$.

Before comparing the advantages and disadvantages of our scheme and the reference scheme in \eqref{eq:fd_scheme}, 
let us repeat the experiment of section \ref{sec:rate}: 
Using $u_\eps^0$ from \eqref{eq:u0} with $\eps=10^{-1}$ instead of $\eps=10^{-3}$, 
we plot
the $L^1(\Omega)$, $L^2(\Omega)$, and $L^\infty(\Omega)$-norms of the differences $|u_\Delta(t,\cdot)-u_\reff(t,\cdot)|$
at time $t=10^{-4}$ in figure \ref{fig:fig3}/left. 
Obviously, the new experiment confirms the rate of convergence $\delta^2\propto K^{-2}$ gained in section \ref{sec:rate}.

The following remarks concerning both schemes can now be made:
\begin{itemize}
\item\emph{Computational cost:} Using Newton's method to solve both approximations \eqref{eq:nonuniform} and \eqref{eq:fd_scheme}, 
			the finite-difference scheme is more efficient, unsurprisingly. 
			The reason for this is the complicate structure of the Jacobian matrix of $\wgrad\DiVz(\xvec)$, 
			whereas the Jacobian matrix of the right hand side of \eqref{eq:fd_scheme} is easy and quick to implement. 
			Numerical experiments show that the finite-difference scheme can be approximately 5-times faster than our scheme,
			using the same values for $K$ and $\tau$.
\item\emph{Conservation of mass:} It is generally known that standard numerical schemes as the finite-difference approximation in \eqref{eq:fd_scheme} 
			do \emph{not} preserve mass. 
			Depending on the initial datum, the loss or gain of mass can decrease quickely with time and causes inaccurate solutions.
			In figure \ref{fig:fig3}/right, we plot the relative change of mass $|\frac1{M}\intom u_\reff(t,x)\dd x - 1|$ with $M=\intom u_\eps^0(x)\dd x$
			for $\eps=10^{-1},10^{-3},10^{-5}$ and $t\in[0,10^{-4}]$. 
			One can observe that the preservation of mass of solutions to the finite-difference scheme is seriously harmed in case of smaller choices of $\eps$.
			This is why we used $\eps=10^{-1}$ in the second experiment for the rate of convergence, 
			since smaller values for $\eps$ produce reference solutions whose change of mass yield to significant distortions 
			of the $L^p$-errors.
\item\emph{Conservation of positivity:} In general, one can expect positivity of the discrete solution 
			to the finite-difference scheme starting with a sufficiently positive initial function. 
			This situation changes dramatically if one considers initial densities with regions of small values or even zero values.
			Take for example the initial datum $u_\eps^0$ in \eqref{eq:u0} with $\eps=0$. 
			Then the solution to the scheme in \eqref{eq:fd_scheme}
			--- again using $K_\reff=3200$ spatial grid points and a time step $\tau_\reff=5\cdot10^{-8}$ ---
			contains negative values after the very first time iteration and finally loses any physical meaning after some more iterations.
			In contrast, our scheme can still handle the case when $\eps=0$ in \eqref{eq:u0}, 
			although one usually has to assume strict positivity for initial values in our approach.
\end{itemize}
Conclusively, our scheme has a major advantage in comparison with standard numerical solvers
if one is interested in a stable and structure-preserving discretization for \eqref{eq:thinfilm}.
Moreover, the slightly plus of the finite-difference scheme and of similiar approximations discussed in the first point --- less computational cost --- 
is invalidated by the fact that one needs much finer discretization parameters
compared to our structure-preserving scheme 
to gain solutions with an adequate physical meaning.

%%%%%%%%%%%%%%%%%%%%%%%%%%%%%%%%%%%%%%%%%%%%%%%%%%%%%%%%%%%%%%%%%%%%%%%%%%%%%
% APPENDIX
%%%%%%%%%%%%%%%%%%%%%%%%%%%%%%%%%%%%%%%%%%%%%%%%%%%%%%%%%%%%%%%%%%%%%%%%%%%%%

\begin{appendix}
  \section{Appendix}
  \begin{lem}[Gargliardo-Nirenberg inequality]
    \label{lem:GN}
    For each $f\in H^1(\Omega)$, one has that
    \begin{align}
      \label{eq:GN}
      \|f\|_{C^{1/6}(\Omega)} \leq (9/2)^{1/3} \|f\|_{H^1(\Omega)}^{2/3} \|f\|_{L^2(\Omega)}^{1/3}.
    \end{align}
  \end{lem}
  \begin{proof}
    Assume first that $f\ge0$.  
    Then, for arbitrary $x<y$, the fundamental theorem of calculus and H\"older's inequality imply that
    \begin{align*}
      \big|f(x)^{3/2}-f(y)^{3/2}\big| \le \frac32\int_x^y 1\cdot f(z)^{1/2}|f'(z)|\dd z \le \frac32|x-y|^{1/4}\|f\|_{L^2(\Omega)}^{1/2}\|f'\|_{L^2(\Omega)}.
    \end{align*}
    Since $f\ge0$, we can further estimate
    \begin{align*}
      |f(x)-f(y)| \le \big|f(x)^{3/2}-f(y)^{3/2}\big|^{2/3} \le (3/2)^{2/3}|x-y|^{1/6}\|f\|_{L^2(\Omega)}^{1/3}\|f\|_{H^1(\Omega)}^{1/3}.
    \end{align*}
    This shows \eqref{eq:GN} for non-negative functions $f$.  
    A general $f$ can be written in the form $f=f_+-f_-$, where $f_\pm\ge0$.  
    By the triangle inequality, and since
    $\|f_\pm\|_{H^1(\Omega)}\le\|f\|_{H^1(\Omega)}$,
    \begin{align*}
      \|f\|_{C^{1/6}(\Omega)} \le \|f_+\|_{C^{1/6}(\Omega)}+\|f_-\|_{C^{1/6}(\Omega)} \le 2(3/2)^{2/3}\|f\|_{L^2(\Omega)}^{1/3}\|f\|_{H^1(\Omega)}^{1/3}.
    \end{align*}
    This proves the claim.
  \end{proof}
\begin{lem}
	For each $p\geq1$ and $\xvec\in\xseqN$ with $\zvec=\cz_\thep[\xvec]$, one has that
	\begin{align}
		\label{eq:xpower}
		\sum_{\kappa\in\hval}\left(\frac\delta{z_\kappa}\right)^p =\sum_{\kappa\in\hval}(x_\kappp-x_\kappm)^p \le (x_K-x_0)^p.
	\end{align}
\end{lem}
\begin{proof}
	The first equality is simply the definition \eqref{eq:zvec} of $z_\kappa$.
	Since trivially $x_\kappp-x_\kappm \leq x_K-x_0$ for each $\kappa\in\hval$,
	and since $p-1\geq0$,
	it follows that
	\begin{align*}
		\sum_{\kappa\in\hval}(x_\kappp-x_\kappm)^p\le (x_K-x_0)^{p-1}\sum_{\kappa\in\hval}(x_\kappp-x_\kappm) = (x_K-x_0)^p.
		&\qedhere
	\end{align*}
\end{proof}

\end{appendix}

\bibliography{Horst}
\bibliographystyle{siam}%\bibliographystyle{plain}

\end{document}